\documentclass[12pt]{amsart}

\usepackage[margin=2.70cm]{geometry}
\usepackage{amscd}
\usepackage{amssymb}
\usepackage{amsmath}
\usepackage{amsthm}
\usepackage{enumerate}
\usepackage{tikz}
\usepackage{wrapfig, framed, caption}
\usepackage{float}
\usetikzlibrary{arrows}
\usepackage[font=small,labelfont=bf]{caption}
\usepackage{xcolor}
\usepackage{mathtools}
\usepackage[colorlinks=true, linkcolor=blue, citecolor=blue,
pagebackref=true]{hyperref}
\usepackage{fouriernc}
\usepackage{ulem}
\usepackage{verbatim}

\newtheorem{theorem}{Theorem}
\newtheorem*{theorem*}{Theorem}

\newtheorem*{theoremY*}{Theorem Y}

\newtheorem*{theoremAB*}{Theorem AB}

\newtheorem*{linearformsmtp*}{Mass transference principle for linear forms}
\newtheorem{corollary}{Corollary}
\newtheorem*{corollary*}{Corollary}
\newtheorem{proposition}{Proposition}
\newtheorem{lemma}{Lemma}

\newtheorem*{claim*}{Claim}

\theoremstyle{definition}
\newtheorem{definition}{Definition}
\theoremstyle{remark}
\newtheorem{remark}{Remark}
\newtheorem*{remark*}{Remark}

\newcommand{\bp}{\mathbf{p}}

\renewcommand{\Bbb}[1]{\mathbb{#1}}
\newcommand{\C}{{\Bbb C}}         % complex numbers

\newcommand{\F}{{\Bbb F}}

\newcommand{\M}{{\Bbb M}}
\newcommand{\N}{{\Bbb N}}         % natural numbers

\newcommand{\Q}{{\Bbb Q}}         % rational numbers
\newcommand{\R}{{\Bbb R}}        % real numbers
\newcommand{\Z}{{\Bbb Z}}         % integer numbers

\newcommand{\cA}{\mathcal{A}}
\newcommand{\cB}{{\mathcal B}}
\newcommand{\cC}{{\mathcal C}}

\newcommand{\cH}{\mathcal{H}}

\newcommand{\cK}{\mathcal{K}}

\newcommand{\bfb}{\mathbf{b}}
\newcommand{\bfw}{\mathbf{w}}
\newcommand{\x}{\mathbf{x}}
\newcommand{\y}{\mathbf{y}}
\newcommand{\p}{\mathbf{p}}
\newcommand{\q}{\mathbf{q}}

\newcommand{\0}{\mathbf{0}}

\newcommand{\diam}{r}
\newcommand{\dist}{\operatorname{dist}}

\renewcommand{\le}{\leq}

\DeclareMathOperator{\dimh}{\dim_H}

\newcommand{\bfa}{\textbf{a}}
\newcommand{\Qp}{\Q_{p}}
\newcommand{\Zp}{\Z_{p}}
\newcommand{\bq}{\textbf{q}}

\newcommand{\LL}{\mathcal{L}}

\numberwithin{equation}{section}

\begin{document}

\title{Weighted approximation for limsup sets}

\author[G. Gonz\'alez Robert]{Gerardo Gonz\'alez Robert}
\address{G. Gonz\'alez Robert,  Department of Mathematical and Physical Sciences,  La Trobe University, Bendigo 3552, Australia. }
\email{G.Robert@latrobe.edu.au}

\author[M. Hussain]{Mumtaz Hussain}
\address{Mumtaz Hussain,  Department of Mathematical and Physical Sciences,  La Trobe University, Bendigo 3552, Australia. }
\email{m.hussain@latrobe.edu.au}

\author[N. Shulga]{Nikita Shulga}
\address{Nikita Shulga,  Department of Mathematical and Physical Sciences,  La Trobe University, Bendigo 3552, Australia. }
\email{n.shulga@latrobe.edu.au}

\author[B. Ward]{ Benjamin Ward}
\address{Benjamin Ward,  {Department of Mathematics, University of York, Heslington, YO10 5DD.} }
\email{benjamin.ward@york.ac.uk, ward.ben1994@gmail.com }
%\date{\today}
\thanks{This research is supported by the ARC Discovery Project 200100994.}
% We will be adding more results in the later versions of this manuscript.}
%\author{Author 1 \footnote{funding for Author 1} \\ Author 1 Affiliation \and Author 2 \\ Author 2 Affiliation}
%
%\date{\footnotesize{\it Dedication}}
%

\frenchspacing
\maketitle
\begin{abstract}
Theorems of Khintchine, Groshev, Jarn\'ik, and Besicovitch in Diophantine approximation are fundamental results on the metric properties of $\Psi$-well approximable sets. These foundational results have since been generalised to the framework of weighted Diophantine approximation for systems of real linear forms (matrices). In this article, we prove analogues of these weighted results in a range of settings including the $p$-adics (Theorems \ref{p-adic weighted KG} and \ref{JB p-adic}),  complex numbers (Theorems \ref{TEO:MEASURE} and \ref{TEO:HDIM}),  quaternions (Theorems \ref{TEO:Q:MEASURE} and \ref{TEO:Q:HDIM}), and formal power series (Theorems \ref{formalt1} and \ref{formalt2}). The key tools in proving the main parts of these results are the weighted ubiquitous systems and weighted mass transference principle introduced recently by Kleinbock--Wang [Adv. Math. (2023)] and  Wang--Wu [Math. Ann. (2021)].

\end{abstract}

\tableofcontents
\section{Introduction}
Let $n,m\geq 1$ be integers and $\Psi=(\psi_{1},\dots, \psi_{n})$ be an $n$-tuple of multivariable approximation functions of the form $\psi_{i}:\Z^{m} \to \R_{+}$ with
\begin{equation*}
    \psi_{i}(\bq) \to 0 \quad \text{ as } \|\bq\|:=\max(|q_1|,\ldots, |q_m|)\to \infty.
\end{equation*}
%Fix some $\theta \in \R^{n} $  and 
For $\bp\in\Z^m$ and $i\in\{1,\ldots,n\}$, let $p_i$ be the $i$-th coordinate of $\bp$. Let
\begin{equation*}
    W_{n,m}(\Psi):=\left\{ X \in [0,1]^{m\times n} : \begin{array}{l}|\bq X_{i}+p_{i}|<\psi_{i}(\bq) \quad 1 \leq i \leq n,\\
    \text{ for i. m. } (\bp,\bq) \in \Z^{n}\times \Z^{m}\backslash\{\textbf{0}\}
    \end{array}
    \right\}.
\end{equation*}
Here and in what follows ``i. m.'' stands for ``infinitely many''. $X=(x_{i,j})_{1\leq i\leq m, 1\leq j \leq n}$ is an $m\times n$ matrix with entries $x_{i,j}\in[0,1]$, and $X_{i}$ denotes the $i$th column vector of $X$.  So
\begin{equation*}
    \bq X+\bp=\left( \begin{array}{c} \bq X_{1}+p_{1} \\
    \vdots\\ 
    \bq X_{n} + p_{n} 
    \end{array} \right)=\left( \begin{array}{c} q_{1}x_{1,1} + \dots + q_{m}x_{m,1} + p_{1}  \\
    \vdots \\
    q_{1}x_{1,n} + \dots + q_{m} x_{m,n} + p_{n} \end{array} \right).
\end{equation*}
%Generally, a Khintchine-Groshev type theorem tells us the $nm$-dimensional Lebesgue measure of $W_{n,m}(\Psi)$, which is either zero or full depending upon the convergence or divergence of a certain series, respectively. Naturally, the series is dependent upon the nature of the approximation functions.
A Khintchine-Groshev type theorem tells us the $nm$-dimensional Lebesgue measure of $W_{n,m}(\Psi)$  is either zero or full, depending upon the convergence or divergence of a certain series, respectively. Naturally, the series is dependent upon the nature of the approximation functions.  

There are many variations of Khintchine-Groshev type theorems. To highlight a few, we recall the following definitions. When $\psi_{1}=\dots = \psi_{n}$ we say the approximation function $\Psi$ is \textit{non-weighted} and simply denote it by $\psi$, otherwise, it is called \textit{weighted}. If the approximation functions $\psi_{i}$ are of the form $\phi_{i}:\R_{+} \to \R_{+}$ with $\psi_{i}(\bfa)=\phi_{i}(\|\bfa\|)$ for each $1\leq i \leq n$ we say the approximation function is \textit{univariable}. A generalisation of univariable approximation function is when the sup norm is replaced by a quasi-norm $\|\cdot\|_{v}=\max_{1\leq i\leq m}|\cdot|^{1/v_{i}}$ with vector $v=(v_{1}, \dots, v_{m})$ for each $v_{i}>0$ and $\sum v_{i} = m$. If each approximation function is of the form $\psi_{i}(\bfa)=\phi_{i}(\|\bfa\|_{v})$ we follow \cite{BadBerVel13} and say $\Psi$ has \textit{property P}. If each approximation function $\psi_{i}$ is monotonic decreasing, we say $\Psi$ is monotonic decreasing. In the case of multivariable approximation, this means
\begin{equation*}
\psi_{i}(\bfa)<\psi_{i}(\bfb) \quad \forall \, \, \, \|\bfb\|<\|\bfa\|.
\end{equation*}
%Lastly, when $\theta=0$ we are considering the classical case of \textit{homogeneous} approximation. In this setting we will often denote $W_{n,m}^{0}(\Psi)=W_{n,m}(\Psi)$. \textit{Inhomogeneous} approximation refers to when $\theta$ is non-zero. \par 

%{\color{red} probably use $\mathcal L$ for the $mn$-dimensional Lebesgue measure. Measure $\mu$ when you define the product measure in Theorem 1 instead of $m$. For p-adic $\mu_p$, complex $\mu_{\mathbb C}$ etc. Similarly, for all other notations. Any thoughts}
The following result, due to Kleinbock and Wang \cite[Theorem 2.7]{KW23} provides the most modern version of the Khintchine-Groshev theorem in the weighted monotonic univariable setting.
\begin{theorem}[{\cite{KW23}}] \label{kleinbock-Wang-KG}
Let $\Psi=(\psi_{1},\dots,\psi_{n})$ be an $n$-tuple of monotonic decreasing approximation functions satisfying property P. Then
\begin{equation*}
    \mu^{\R}_{m\times n}\left(W_{n,m}(\Psi)\right)=\begin{cases}
        0 \quad \text{\rm if }\quad  \sum\limits_{r=1}^{\infty}r^{m-1}\prod_{i=1}^{n}\psi_{i}(r) < \infty ,\\
        1 \quad \text{\rm if }\quad  \sum\limits_{r=1}^{\infty}r^{m-1}\prod_{i=1}^{n}\psi_{i}(r) = \infty.
    \end{cases}
\end{equation*}
\end{theorem}

Here and throughout, we denote the $m\times n$ dimensional Lebesgue measure to be $\mu_{m\times n}^{\R}$. It should be noted that Kleinbock \& Wang's result was more general still. In particular, rather than property P one can bound each $q_{i}$ by some function $\Phi_{i}$ satisfying certain properties, see \cite{KW23} for more details. \par
Below, we briefly highlight the results preceding this. These include but are not limited to
\begin{itemize}
    \item $n=m=1$, $\psi$ monotonic, Khintchine \cite{K24}.
    \item $n=m=1$, $\psi$ non-monotonic, conjectured by Duffin and Schaeffer (1941) and proven by Maynard and Koukoulopoulos \cite{KM20}.
    \item $n>1, m=1$, $\psi$ is monotonic, Khintchine \cite{K26}.
    \item $nm\geq 1$, $\psi$ is univariable, proven by Groshev \cite{Groshev}. In fact, in the original theorem, there were some stronger assumptions on the approximating functions $\psi$, that is, $r^{\max(1, m-1)}\psi(r)^n$ to be monotonic. The reason that this assumption is unnecessary is due to \cite{BDV06}.
    \item $nm > 1$, $\psi$ is univariable, non-monotonic, by Beresnevich \& Velani \cite{BV10}. Prior to this, for $m=1$, $n\geq 1$ and $\psi$ univariable non-monotonic see \cite{Gallagher}.
   % \item $nm \geq 1$, $\Psi$ simultaneous univariable monotonic inhomogeneous Sz\"usz \cite{Sz58}, Schmidt \cite{Sch64}, and non-monotonic by Sprindzhuk \cite{S79} for $n \geq 3$, and Yu \cite{Yu21} for $n=1$ and $m\geq 3$
    %\item $nm>2$, $\Psi$ simultaneous univariable non-monotonic inhomogeneous Allen \& Ram\'irez \cite{AR23}.
    \item $n\geq1$ $m> 1$, for univariable non-monotonic weighted $\Psi$ by Hussain \& Yusupova \cite{HY14}.
    %\item In slightly different settings, for the small linear forms, the analogue of Theorem \ref{kleinbock-Wang-KG} was established by Fischler, Hussain, Kirstensen, \& Levesley \cite{FHKL}.
\end{itemize}
Similar results for the more general {\em inhomogeneous approximation} are known, but, as far as the authors are aware, the case of weighted multivariable inhomogeneous approximation remains unexplored. \par 
When the $n$-tuple of approximation functions $\Psi$ decreases sufficiently fast such that $\mu^{\R}_{m\times n}\left(W_{n,m}(\Psi)\right)=0$, we desire a more delicate notion of "size". The Hausdorff measure and dimension provide us with such a tool.  The following Hausdorff measure and dimension results are known for the set $W_{n,m}(\Psi)$. 
\begin{theorem}[{\cite[Theorem 11.1]{WW19}}] \label{wang wu KG}
    Let $\Psi=(\psi_{1},\dots,\psi_{n})$ be an $n$-tuple of approximation functions of the form
    \begin{equation*}
        \psi_{i}(\bq)=\bq^{-\tau_{i}}\, , \quad (1\leq i \leq n)
    \end{equation*}
    for $\boldsymbol{\tau}=(\tau_{1}, \dots , \tau_{n})\in\R^{n}_{+}$ with
    \begin{equation*}
        \sum\limits_{i=1}^{n}\tau_{i}>m.
    \end{equation*}
    Then
    \begin{equation*}
        \dimh W_{n,m}(\Psi) = \min_{1\leq k \leq n} \left\{ n(m-1)+\frac{n+m+\sum\limits_{j: \tau_{j}<\tau_{k}}(\tau_{k}-\tau_{j})}{1+\tau_{k}} \right\}.
    \end{equation*}
\end{theorem}
It should be noted that the above result was proven for a more general range of non-increasing approximation functions, with the $\tau_{i}$'s in the dimension result being replaced by limit points of vectors related to $\psi_{i}$. For example, in the univariable case $\tau$ is replaced by the lower order at infinity of $\psi$, that is, $\tau=\liminf_{q\to \infty} \tfrac{-\log \psi(q)}{\log q}$. \par 
As with the result of Kleinbock and Wang various results preceded this. We highlight a few below. For brevity, we stick to considering the approximation function of the form considered in the above theorem.
\begin{itemize}
    \item $n=m=1$ was independently proven by Jarn\'ik and Besicovitch. 
    \item for all linear forms but for the non-weighted case, it was proven by Bovey and Dodson in \cite{BoveyDodson}. This result was further generalised to the lower-order settings by Dodson in \cite{Dodson_lowerorder}.
    \item for $m=1$ and $n\geq 1$ it was proven in the weighted case by Rynne \cite{Rynne}. In fact, the result of Rynne was more general, as he restricted the rational denominators to an infinite subset of integers.  This result was further generalised to dual settings by Dickinson and Rynne \cite{DickinsonRynne}.
   % \item $n=1$ was proven in the inhomogeneous letting by Levesley.\\
   % \item {\color{red} ... }
\end{itemize}

The aim of this paper is to prove analogies of the above two theorems in a variety of settings. In order to prove Lebesgue dichotomy statements (analogue of Theorem \ref{kleinbock-Wang-KG}) in a variety of settings, we use the recently developed notion of weighted ubiquitous systems. The notion of weighted ubiquity was introduced by Wang and Wu in \cite{WW19} and was developed further by Kleinbock and Wang in \cite{KW23}. In the paper \cite{WW19}, the authors established a very powerful weighted mass transference principle that, under weighted ubiquity assumption, allows for the Hausdorff measure/dimension results in the weighted settings. We refer the reader to \cite{AllenDaviaud, AT19}
 for a comprehensive survey of the mass transference principle.
 
In the following section, we recall this theory. In the subsequent sections, along with some other things, we apply this framework to a variety of settings including Diophantine approximation in $p$-adics, formal power series, complex numbers, and quaternions.

\section{Toolbox}\label{SEC:TOOLBOX}
This section consists of a range of tools used to determine metric properties of $\limsup$ sets. The first subsection recalls Dirichlet's and Minkowski's theorems on linear forms. The second subsection provides the generalised setup of weighted ubiquitous systems into which our applications fall. The third subsection gives the techniques required to determine ambient measure results on $\limsup$ sets. The last subsection recalls the definition of Hausdorff measure and dimension and gives the tools required to prove Hausdorff measure and dimension results on $\limsup$ sets.
\subsection{Dirichlet and Minkowski's theorems}
A basic problem of Diophantine approximation is to approximate a given real number by rational numbers to a certain degree. For example, it is obvious that for a given $\alpha\in\R$ and any $q\in\N$, there is some integer $p$ such that
\[
|q\alpha - p|< \frac{1}{2}.
\]
A classical result by Dirichlet improves this observation. Namely, given $\alpha\in\R$, for each real number $Q>1$ there is some $(q,p)\in\Z^{2}$ such that
\[
1\leq q\leq Q
\quad\text{ and }\quad
|q\alpha - p|< \frac{1}{Q}.
\]
As a consequence, there are infinitely many pairs $(q,p)\in\N\times \Z$ satisfying
\[
|q\alpha - p|< \frac{1}{q}.
\]
Dirichlet's result can be proven using the pigeonhole principle. Furthermore, the argument can be easily extended to linear forms (see, for example, \cite[Chapter II, Theorem 1E]{Schmidt80}).

\begin{theorem}[Dirichlet, 1842]\label{TheoremDirichlet}
Let $m,n\in\N$ and $A\in \R^{m\times n}$ be given. For any $Q>1$, there exists a non-zero $\q=(q_1,\ldots,q_m)\in\Z^{m}$ and some $\p=(p_1,\ldots, p_n)\in\Z^n$ such that
\begin{align*}
    |\q A_{i}-p_i| &< Q^{-1} \quad (1\leq i\leq n),\\ 
    |q_j| &\leq Q^{\frac{n}{m}} \quad(1\leq j\leq m).
\end{align*}
\end{theorem}

Take $m,n\in\N$. Define the function $\Psi:\N\to\R_{+}$ by $\Psi(r)=r^{-m/n}$. For every $\q=(q_1,\ldots,q_m)\in\Z^m$, write $\|\q\|=\max\{|q_1|,\ldots,|q_m|\}$. By Theorem \ref{TheoremDirichlet}, for every matrix $A\in\R^{m\times n}$ there are infinitely many vectors $0\neq \q\in\Z^m$ and $\p\in\Z^n$ such that
\begin{equation}\label{Eq:ThmDirMink}
    |\q A_i - p_i|< \Psi(\|\q\|) \quad (1\leq i\leq n).
\end{equation}

Dirichlet's theorem and its consequences, despite their foundational character, are not strong enough for many applications including the ones we study in this article. This is because we might need to replace the function $\Psi$ in \eqref{Eq:ThmDirMink} by positive and non-increasing functions $\Psi_i$ in each coordinate axis $i\in\{1,\ldots, n\}$. Minkowski's theorem for linear forms (Theorem \ref{TheoremMinkowski} below) is a key tool towards this goal.

Recall that a \textit{lattice} $\Lambda$ on $\R^n$, $n\in\N$, is a set of the form
\[
\Lambda
=
\left\{ a_1 \mathbf{v}_1 + \ldots + a_n \mathbf{v}_n: a_1,\ldots, a_n\in\Z \right\},
\]
where $\mathbf{v}_1, \ldots, \mathbf{v}_n\in\R^n$ are $n$ given linearly independent row vectors. The determinant of $\Lambda$, denoted $\det(\Lambda)$, is the determinant of the matrix formed by $\mathbf{v}_1, \ldots, \mathbf{v}_n$.

%Given a lattice $\Lambda$, Minkowski's theorem on linear forms provides necessary conditions for the existence of a non-zero element of $\Lambda$ satisfying certain linear inequalities.

\begin{theorem}[Minkowski, 1896]\label{TheoremMinkowski}
    Let $N\in\N$, $\Lambda\subseteq \R^N$ a lattice of determinant $d$, and $A\in \R^{N\times N}$. If the positive real numbers $c_1,\ldots,c_N$ satisfy
    \[
    \det(\Lambda)|\det(A)| \leq c_1\cdots c_N,
    \]
    then there is some non-zero $\mathbf{u}\in \Lambda$ such that
    \begin{align*}
    \left| \mathbf{u}A_{j} \right| &< c_j \quad (1\leq j< N),\\
    \left| \mathbf{u}A_{N} \right| &\leq c_N\, .
    \end{align*}
\end{theorem}

A proof of Theorem \ref{TheoremMinkowski} can be found in \cite[Chapter III, Theorem III]{Cassels1957}. Throughout, we will refer to this as Minkowski's Theorem. Our formulation is slightly different from the one in the reference, but the proof can be easily adapted to our case. %{\color{red} added a bit here to relate to our setup a bit more.}
\par 
In relation to Diophantine approximation Minkowski's Theorem allows us to deduce that for vector $v=(v_{1}, \dots , v_{m})$ with $\sum_{i=1}^{m}v_{i}=m$, for any $A \in \R^{m\times n}$
\begin{equation*}
|\bq A_{i}-p_{i}|<\psi_{i}(\|\bq\|_{v}) \quad (1\leq i \leq n)\, ,
\end{equation*}
is solved for infinitely many integer vectors $(\bp, \bq) \in \Z^{n}\times(\Z^{m}\backslash\{0\})$ provided the $n$-tuple of approximation functions satisfy
\begin{equation*}
    \prod_{i=1}^{n}\psi_{i}(r)=r^{-1} \quad \text{ and } \quad \psi_{i}(r)<\tfrac{1}{2} \quad (1\leq i \leq n)\, \text{ monotonic decreasing}
\end{equation*}
for all $r\in\R_{+}$.

\subsection{Weighted Ubiquitous systems} \label{ubiquity} In this section we give the definition of local ubiquity for rectangles as given in \cite{KW23}. This definition is a generalisation of ubiquity for rectangles as found in \cite{WW19}. The notion of an ``Ubiquitous system'' for balls was introduced by Dodson, Rynne, and Vickers \cite{DRV}, which was then generalised to the abstract metric space settings in \cite{BDV06}.\par 
 Fix an integer $n \geq 1$, and for each $1 \leq i \leq n$ let $(\Omega_{i},|\cdot|_{i},\mu_i)$ be a bounded locally compact metric space, where $|\cdot|_i$ denotes a metric on $\Omega_i$, and $\mu_{i}$ denotes a Borel probability measure over $\Omega_i$. Further, we assume that the measure $\mu_i$ is $\delta_{i}$-Ahlfors regular probability measure. That is, there exist constants $0<c_{1}<c_{2}<\infty$ such that for any ball $B_{i}(x_{i},r)$ with centre $x_{i}\in\Omega_{i}$ and radius $0<r<r_{0}$ for some $r_{0}\in\R_{+}$ we have that
 \begin{equation*}
     c_{1}r^{\delta_{i}} \leq \mu_{i}\left(B_{i}(x_{i},r)\right) \leq c_{2}r^{\delta_{i}}\, .
 \end{equation*}
  Consider the product space $(\Omega,\|\cdot\|,\mu)$, where
\begin{equation*}
\Omega=\prod_{i=1}^{n}\Omega_{i}, \quad \mu=\prod_{i=1}^{n}\mu_{i}, \quad \|\cdot\|=\max_{1 \leq i \leq n}|\cdot|_{i}
\end{equation*}
are defined in the usual way.
For any $x \in \Omega$ and $r \in \R_{+}$ define the open ball
\begin{equation*}
B(x,r)=\left\{ y \in \Omega: \max_{1 \leq i \leq n}|x_{i}-y_{i}|_{i}< r \right\}=\prod_{i=1}^{n}B_{i}(x_{i},r),
\end{equation*}
where $B_{i}$ are the open $r$-balls associated with the $i^{\text{th}}$ metric space $\Omega_i$. Let $J$ be a countably infinite index set, and $\beta: J \to \R_{+}$, $\alpha \mapsto \beta_{\alpha}$ a positive function satisfying the condition that for any $N \in \N$
\begin{equation*}
\# \left\{ \alpha \in J: \beta_{\alpha} < N \right\} < \infty.
\end{equation*}
 Let $l_{k},u_{k}$ be two sequences in $\R_{+}$ such that $u_{k} \geq l_{k}$ with $l_{k} \to \infty$ as $k \to \infty$. Define
\begin{equation*}
J_{k}= \{ \alpha \in J: l_{k} \leq \beta_{\alpha} \leq u_{k} \}.
\end{equation*}
Let $\rho=(\rho_{1}, \dots , \rho_{n})$ be an $n$-tuple of non-increasing functions $\rho_{i}: \R_{+} \to \R_{+}$ such that each $\rho_{i}(x) \to 0$ as $x\to \infty$. For each $1 \leq i \leq n$, let $( R_{\alpha,i})_{\alpha \in J}$ be a sequence of subsets in $\Omega_{i}$.
The family of sets $( R_\alpha)_{\alpha\in J}$ where
\begin{equation*}
	R_{\alpha}=\prod_{i=1}^{n} R_{\alpha, i}, %_{ \alpha \in J}.
\end{equation*}
for each $ \alpha \in J$, are called \textit{resonant sets}.

Define
\begin{equation*}
\Delta(R_{\alpha},\rho(r))= \prod_{i=1}^{n}  \Delta_{i}(R_{\alpha,i},\rho_{i}(r)),
\end{equation*}
where for any set $A\subset \Omega_i$ and $b \in \R_{+}$
\begin{equation*}
\Delta_{i}(A,b)= \bigcup_{a \in A}B_{i}(a,b)
\end{equation*}
is the union of balls in $\Omega_i$ of radius $b$ centred at all possible points in $A$. Generally $\Delta(R_{\alpha},\rho(r))$ is the product of $\rho_{i}(r)$-neighbourhoods of $R_{a,i}$ for each coordinate $i\in\{1,\dots, n\}$. \par 
We also require the following definition that was introduced in \cite{AB19} as a generalisation of the intersection properties from \cite{BDV06}.

\begin{definition}[$\kappa$-scaling property]
Let $0 \leq \kappa_i <1$ and $1 \leq i \leq n$. The sequence $(R_{\alpha,i})_{\alpha \in J}$ has a {\em $\kappa_i$-scaling property} if for any $\alpha \in J$, any ball $B_{i}(x_{i},r) \subset \Omega_{i}$ with centre $x_{i} \in R_{\alpha,i}$, and $0 < \epsilon < r$ we have
\begin{equation*}
c_{2}r^{\delta_{i}\kappa_i}\epsilon^{\delta_{i}(1-\kappa_i)} \leq \mu_{i} \left( B_{i}(x_{i},r) \cap \Delta_{i}(R_{\alpha,i},\epsilon) \right) \leq c_{3} r^{\delta_{i}\kappa_i}\epsilon^{\delta_{i}(1-\kappa_i)},
\end{equation*}
for some constants $c_{2},c_{3}>0$.
\end{definition}
See \cite[Section 2]{AB19} for calculations of $\kappa_i$ for various resonant sets. Intuitively, one can think of $\kappa_i$ being the value such that $\delta_{i}\kappa_i$ is the box dimension of $R_{\alpha,i}$. As an example note that $\kappa=0$ when $R_{\alpha,i}$ is a finite collection of points and $\kappa=\frac{m-1}{m}$ for $R_{\alpha,i}$ being $(m-1)$-dimensional affine hyperplanes. Although the definition considers the $\kappa_i$-scaling property in each coordinate axis, for our purpose we take $\kappa_1=\cdots=\kappa_n=\kappa$, and refer to it as $\kappa$-scaling property. In particular, this is the $\kappa$-scaling property considered in \cite{WW19}. Note that in \cite{BDV06} the $\kappa$-scaling property is essentially the intersection conditions \cite[Section 2.3]{BDV06}. In particular $\gamma=\delta\kappa$.\par 
The following notion of ubiquity for rectangles can be found in \cite[Section 2.2]{KW23}.

\begin{definition}[Local ubiquitous system for rectangles]\rm
Call the pair $\big((R_{\alpha})_{\alpha \in J}, \beta\big)$ {\em a local ubiquitous system for rectangles with respect to $\rho$} if there exists a constant $c>0$ such that for any ball $B \subset \Omega$ and all sufficiently large $k \in \N$
\begin{equation*}
\mu \left( B \cap \bigcup_{\alpha \in J_{k}}\Delta(R_{\alpha}, \rho(u_{k})) \right) \geq c m(B).
\end{equation*}
\end{definition}
For an $n$-tuple of approximation functions $\Psi=(\psi_{1},\dots, \psi_{n})$ with each $\psi_{i}:\R_{+} \to \R_{+}$ define
\begin{align*}
    W^{\Omega}(\Psi)&=\left\{x \in \Omega : x \in \Delta\left(R_{\alpha},\Psi(\beta_{\alpha})\right) \, \text{ for infinitely many } \alpha \in J \right\}\\
    &=\limsup_{\alpha \in J} \Delta\left(R_{\alpha}, \Psi(\beta_{\alpha}) \right).
\end{align*}
For all the applications listed below, the corresponding sets will be described by the $\limsup$ set outlined above. %We now give the known results on the metric theory of $W^{\Omega}(\Psi)$.

\subsection{Ambient measure statements}

The following is the well-known Borel-Cantelli Lemma usually used to prove the convergence cases for Lebesgue dichotomy statements, see for example \cite[Theorem 4.18]{Kallenberg2021} for a proof.
\begin{lemma}[Borel-Cantelli lemma, convergence part] \label{Borel_Cantelli convergence}
Let $(\Omega,\mathcal{A},\mu)$ be a measure space and let $(A_k)_{k\geq 1}$ be a sequence of measurable sets. If 
\begin{equation*}
\sum_{k\geq 1} \mu(A_k)< \infty, \, \quad \text{then} \quad \mu\left(\limsup_{k\to \infty} A_k\right)=0\, .
\end{equation*}
\end{lemma}

The following theorem from \cite{KW23} provides the ambient measure theory for $W^{\Omega}(\Psi)$ in the divergence case. Prior to stating the result we need one more definition on functions. For constant $0<c<1$ a function $f$ is said to be $c$-regular with respect to a sequence $\{r_{i}\}_{i\in\N}$ if 
\begin{equation*} 
f(r_{i+1}) \leq c f(r_{i})
\end{equation*}
for all sufficiently large $i$.

\begin{theorem}[{\cite{KW23}}] \label{KW ambient measure}
     Let $W^{\Omega}(\Psi)$ be defined as above and assume that $\left( (R_{\alpha})_{\alpha \in J}, \beta\right)$ is a local ubiquitous system for rectangles with respect to $\rho$, and that the resonant sets $(R_{\alpha,i})$ have $\kappa_{i}$-scaling property and each measure $\mu_{i}$ is $\delta_{i}$-Ahlfors regular. Suppose that
    \begin{enumerate}
        \item[\rm (I)] each $\psi_{i}$ is decreasing,
        \item[\rm (II)] for each $1\leq i \leq n$,  $\psi_{i}(r) \leq \rho_{i}(r)$ for all $r\in\R_{+}$ and $\rho_{i}(r)\to 0$ as $r\to \infty$,
        \item[\rm (III)] either $\rho_{i}$ is $c$-regular on $( u_{k})_{k\geq 1}$ for all $1\leq i \leq n$ or $\psi_{i}$ is $c$-regular on $ ( u_{k})_{k\geq 1}$ for all $1\leq i \leq n$ for some $0<c<1$.
    \end{enumerate}
    Then,
    \begin{equation*}
        \mu(W(\Psi))=\mu(\Omega) \quad \text{\rm if }\quad  \quad \sum\limits_{k=1}^{\infty}\prod_{i=1}^{n}\left(\frac{\psi_{i}(u_{k})}{\rho_{i}(u_{k})}\right)^{\delta_{i}(1-\kappa_{i})}\, = \infty.
    \end{equation*}
\end{theorem}

It will often be convenient at times to multiply the approximating function $\Psi$ by some constant. The following lemma, see \cite[Lemma 5.7]{BDGW23}, shows that, in terms of ambient measure, the measure of the $\limsup$ set remains unchanged.

\begin{lemma}[{\cite{BDGW23}}] \label{CI_product}
Let $(\Omega, \|\cdot\|,\mu)$ be the product space defined above. Let $(S_i)_{i\in\N}$ be a sequence of subsets in the support of $\mu$ and $(\mathbf{\delta}_{i})_{i \in \N}$ be a sequence of positive $n$-tuples $\mathbf{\delta}_{i}=(\delta^{(1)}_{i},\dots,\delta^{(n)}_{i})$ such that $\delta^{(j)}_{i} \to 0$ as $i \to \infty$ for each $1 \leq j \leq n$. Then, for any $\mathbf{C}=(C_{1},\dots , C_{n})$ and $\mathbf{c}=(c_{1},\dots , c_{n})$ with $0<c_j\le C_{j}$ for each $1 \leq j \leq n$
\begin{equation}\label{vb25}
\mu\left( \limsup_{i \to \infty} \Delta(S_{i}, \mathbf{C}\mathbf{\delta}_{i}) \;\setminus\;\limsup_{i \to \infty} \Delta(S_{i}, \mathbf{c}\mathbf{\delta}_{i}) \right)=0\,,
\end{equation}
where $\mathbf{c}\mathbf{\delta}_i=(c_1\delta^{(1)}_i,\dots,c_n\delta^{(n)}_i)$ and similarly $\mathbf{C}\mathbf{\delta}_i=(C_1\delta^{(1)}_i,\dots,C_n\delta^{(n)}_i)$.
\end{lemma}

\subsection{Hausdorff measure and dimension statements}

For completeness, we give below a very brief introduction to Hausdorff measures and dimensions. For further details, see \cite{F14}.

Let $(\Omega,d)$ be a metric space and $F \subset \Omega$.
 Then for any $0 < \rho \leq \infty$, any finite or countable collection~$\{B_i\}$ of subsets of $\Omega$ such that
$F\subset \bigcup_i B_i$ and 
\begin{equation*}
    \diam (B_i)=\frac{1}{2}\inf\{ d(x,y) :  x,y \in B_{i}\}\le \rho
\end{equation*}
is called a \emph{$\rho$-cover} of $F$.

Let $f : \R_+ \to \R_+$ be a dimension function, that is $f(r)$ is a continuous, non-decreasing function defined on $\R_+$ such $\lim_{r\to\infty} f(r)=0$. Let

\[ 
\cH_{\rho}^{f}(F)=\inf \sum\limits_{i} f\left(\diam (B_i)\right),
\]
where the infimum is taken over all possible $\rho$-covers 
$\{B_i\}$ of $F$. The \textit{$f$-dimensional Hausdorff measure of $F$} is defined to be
\[
\cH^f(F)=\lim_{\rho\to 0}\cH_\rho^f(F).
\]
In the case that $f(r)=r^s \;\; (s\geq 0)$, the measure $\cH^f$ is denoted $\cH^s$ and is called \emph{$s$-dimensional Hausdorff measure}. For any set $F\subset \Omega$ one can easily verify that there exists a unique critical value of $s$ at which the function $s\mapsto\cH^s(F)$ ``jumps'' from infinity to zero. The value taken by $s$ at this discontinuity is referred to as the \textit{Hausdorff dimension} of $F$ and is denoted by $\dimh F$; i.e.
\[
\dimh F :=\inf\left\{s\geq 0\;:\; \cH^s(F)=0\right\}.
\]

Below is the Hausdorff measure analogue of the Borel-Cantelli lemma in the convergence case, which allows us to obtain the convergence Hausdorff measure statement of certain sets via calculating the Hausdorff $f$-sum of a fine cover. Recall a collection $\cB=\{B_{i}\}$ of subsets covering $F$ is called a fine cover if for any $x\in F$ and any $r>0$ there exists $B_{i} \in\cB$ such that $r(B_{i})<r$ and $x\in B_{i}$. Trivially this implies that for any $\rho>0$ there exists a subset of $\cB$ that is a $\rho$-cover of $F$.

\begin{lemma}[Hausdorff-Cantelli lemma]\label{bclem}
Let $\{B_i\}\subset \Omega$ be a fine cover of a set $F$ and let $f$ be a dimension function such that
\begin{equation}
\label{fdimcost}
\sum_i f\left(\diam(B_i)\right) \, < \, \infty.
\end{equation}
Then $\cH^f(F)=0$. In particular, if $f(r)=r^{s}$ and we have \eqref{fdimcost} then
\begin{equation*}
    \dimh F \leq s.
\end{equation*}
\end{lemma}

For the lower bound of the Hausdorff dimension and the divergent counterpart of the Hausdorff measure theory, we have the following theorem due to Wang and Wu \cite[Theorem 3.1-3.2]{WW19}, which appeals to the notion of weighted ubiquitous systems.

\begin{theorem}[{\cite{WW19}}] \label{MTPRR}
Let $W(\Psi)$ be defined as above and assume that $\left( (R_{\alpha})_{\alpha \in J}, \beta\right)$ is a local ubiquitous system for rectangles with respect to $\rho=(\rho^{a_{1}}, \dots ,\rho^{a_{n}})$ for some function $\rho:\R_{+} \to \R_{+}$ and $(a_{1},\dots, a_{n}) \in \R^{n}_{+}$, where $\rho^{a_j}(x):=\rho(x)^{a_j}$. Assume, for each $1\leq i \leq n$, the resonant sets $(R_{\alpha,i})$ have $\kappa$-scaling property. Assume each measure $\mu_{i}$ is $\delta_{i}$-Ahlfors regular. 
Then, if $\Psi=(\rho^{a_{1}+t_{1}},\dots, \rho^{a_{n}+t_{n}})$ for some $\textbf{t}=(t_{1}, \dots, t_{n}) \in \R^{n}_{+}$,
\begin{equation*}
\dimh W(\Psi) \geq \min_{A_{i} \in A} \left\{ \sum\limits_{j \in \cK_{1}} \delta_{j}+ \sum\limits_{j \in \cK_{2}} \delta_{j}+ \kappa \sum\limits_{j \in \cK_{3}} \delta_{j}+(1-\kappa) \frac{\sum\limits_{j \in \cK_{3}}a_{j}\delta_{j}-\sum\limits_{j \in \cK_{2}}t_{j}\delta_{j}}{A_{i}} \right\}=s,
\end{equation*}
where $A=\{ a_{i}, a_{i}+t_{i} : 1 \leq i \leq n \}$ and $\cK_{1},\cK_{2},\cK_{3}$ are a partition of $\{1, \dots, n\}$ defined as
\begin{equation*}
 \cK_{1}=\{ j:a_{j} \geq A_{i}\}, \quad \cK_{2}=\{j: a_{j}+t_{j} \leq A_{i} \} \backslash \cK_{1}, \quad \cK_{3}=\{1, \dots n\} \backslash (\cK_{1} \cup \cK_{2}).
 \end{equation*}
 Furthermore, for any ball $B \subset \Omega$ we have
  \begin{equation} \label{MTPRR_measure}
\cH^{s}(B \cap W(\Psi))=\cH^{s}(B).
\end{equation}
 \end{theorem}
The main motivation for this result came from the landmark paper of Beresnevich and Velani \cite{BeresnevichVelani} in which they developed the mass transference principle from balls to balls. This is surprising as the Hausdorff measure theory underpins the Lebesgue measure theory. This powerful tool has since been generalised to various settings, see for instance \cite{ AB19, AllBer, AllenDaviaud, HussainSimmons,  WW19, WWX15} and references therein. \par
We clarify some notations that will be used throughout. For real quantities $A,B$ and a parameter $t$, we write $A \ll_t B$ if $A \leq c(t) B$ for a constant $c(t) > 0$ that depends on $t$ only (while $A$ and $B$ may depend on other parameters). We write  $A\asymp_{t} B$ if $A\ll_{t} B\ll_{t} A$. If the constant $c>0$ is absolute, we simply write $A\ll B$ and $A\asymp B$.

\section{A motivating example}

As a motivating example, and a warm-up to later applications, let us consider the set
\begin{equation*}
    W^{\R}_{1,2}(\psi_{1},\psi_{2})=\left\{(x_{1},x_{2})\in I^{2}: |qx_{i}-p_{i}|<\psi_{i}(q) \quad (i=1,2) \quad \text{ for i.m.} \  (p_{1},p_{2},q)\in \Z^{2}\times\N \right\}\, ,
\end{equation*}
where $I^{2}=[0,1]^{2}$. Theorem~\ref{kleinbock-Wang-KG} allows us to deduce the following result.
\begin{corollary} \label{2 dim KG}
    For $\psi_{1},\psi_{2}:\N\to \R_{+}$ monotonic decreasing functions, we have that
    \begin{equation*}
        \mu^{\R}_{2}\left( W^{\R}_{1,2}(\psi_{1},\psi_{2})\right)= \begin{cases}
            0 \quad \text{\rm if } \quad \sum\limits_{q=1}^{\infty}\psi_{1}(q)\psi_{2}(q)<\infty\, , \\[2ex]
            1 \quad \text{\rm if } \quad \sum\limits_{q=1}^{\infty}\psi_{1}(q)\psi_{2}(q)=\infty\, .
        \end{cases}
    \end{equation*}
\end{corollary}
Furthermore, from Theorem~\ref{wang wu KG} we can deduce the following 
\begin{corollary} \label{2 dim JB}
    For 
    \begin{equation*}
        \psi_{1}(q)=q^{-\tau_{1}}\quad \text{\rm and } \quad \psi_{2}(q)=q^{-\tau_{2}}
    \end{equation*}
    with $\tau_{1}+\tau_{2}>1$ and $\tau_{1},\tau_{2}>0$, we have that
    \begin{equation*}
        \dimh W^{\R}_{1,2}(\psi_{1},\psi_{2}) = \min_{i=1,2} \left\{ \frac{3+(\tau_{i}-\min\{\tau_{1},\tau_{2}\})}{1+\tau_{i}} \right\}\, .
    \end{equation*}
\end{corollary}

We stress the results and methods of proof of this section are not new, in particular see \cite{KW23, WW19}. We now show how both of these results can be derived by constructing a suitable ubiquitous system of rectangles and then applying the theorems of the previous section. The weighted ubiquitous system is setup in much the same was as the classical ubiquitous system, with the key difference being that Minkowski's Theorem is used in place of Theorem~\ref{TheoremDirichlet}. See for example \cite[Theorem 1.1.4]{BRV16} for the statement in the classical one dimensional setting. However, the application of the weighted ubiquitous system is slightly more complex. In particular it requires a careful choice of the $\rho$ function.\par 
To begin, let
\begin{align*}
    \Omega &=[0,1]^{2}, \quad \mu=\mu^{\R}_{2}, \\
    J&=\N , \quad \beta_{\alpha}=\alpha, \\
    l_{k}&=M^{k-1}, \quad u_{k}=M^{k}, \\
    J_{k} &=\left\{q\in\N: M^{k-1}\leq q \leq M^{k} \right\}, \\
    R_{q} &=\left\{\left(\tfrac{p_{1}}{q},\tfrac{p_{2}}{q} \right) : 0\leq p_{1},p_{2}\leq q \right\}\, .
\end{align*}
Note that $R_{q}$ are collections of points and so $(R_{q})_{q\geq 1}$ has $\kappa$-scaling property $\kappa=0$. Since $\mu_{2}^{\R}=\mu_{1}^{\R}\times \mu_{1}^{\R}$ trivially $\delta_{1}=\delta_{2}=1$. Here $M$ in the definitions of $u_{k}$ and $l_{k}$ is some large constant (we can take any integer $M\geq 64$). Observe that
\begin{equation*}
    W^{\R}_{1,2}(\psi_{1},\psi_{2})= \limsup_{q\to \infty}\Delta\left( R_{q}, \left(\tfrac{\psi_{1}(q)}{q},\tfrac{\psi_{2}(q)}{q}\right) \right)\, .
\end{equation*}
We prove the following.

\begin{proposition} \label{2 dim ubiquity}
    Let $\rho_{1},\rho_{2}:\N\to \R_{+}$ be functions of the form
    \begin{equation} \label{2 dim rho conditions}
        \rho_{i}(q)=M\frac{\phi_{i}(q)}{q}\, \quad (i=1,2) \quad \text{  with } \quad \phi_{1}(q)\phi_{2}(q)=q^{-1}\, ,
    \end{equation}
    and each $\phi_{i}(q)\to 0$ as $q\to \infty$. Then for any ball $B\subset I^{2}$ there exists $k_{0}\in\N$ such that for all $k>k_{0}$ we have
    \begin{equation*}
        \mu^{\R}_{2}\left(B\cap \bigcup_{q\in J_{k}} \Delta\left(R_{q}, \left(\rho_{1}(M^{k}),\rho_{2}(M^{k})\right)\right)\right)\geq \frac{1}{2}\mu^{\R}_{2}(B)\, .
    \end{equation*}
\end{proposition}

\begin{proof}
    The following is a standard method of proving such a result. Initially, let us give some bounds on $k_{0}$ that will be used later.  Choose $k_{0}$ such that for all $k>k_{0}$ we have
    \begin{equation*}
           \phi_{i}(M^{k})<\frac{1}{2} \qquad (i=1,2) \quad\text{ and }\quad M^{k}> \frac{32}{3}\pi^{2}\mu^{\R}_{2}(B)^{-1}\, .
       \end{equation*}
       It will become clear later why such bounds are chosen.
       Now pick any $k>k_{0}$. Firstly, by Minkowski's Theorem for any $(x_{1},x_{2})\in[0,1]^{2}$ the system
    \begin{equation*}
        \begin{cases}
            |qx_{1}-p_{1}|<\phi_{1}(M^{k}),\\
            |qx_{2}-p_{2}|<\phi_{2}(M^{k}),\\
            |q|\leq M^{k}
        \end{cases}
    \end{equation*}
    has a non-zero integer solution $(p_{1},p_{2},q)\in\Z^{3}$. Note by our choice of $k_{0}$ that $q\neq 0$, since otherwise we would have $p_{1}$ and $p_{2}$ non-zero solving $|p_{i}|<\tfrac{1}{2}$ which is clearly false. So, multiplying the solution $(p_{1},p_{2},q)$ through by $-1$ if necessary, we have that $(p_{1},p_{2},q)\in\Z^{2}\times \N$. Dividing through by $q$ in the first two inequalities allows us to see that
    \begin{equation} \label{motivating example eq 1}
        \mu^{\R}_{2}\left(B\cap \bigcup_{1\leq q\leq M^{k}}\Delta\left(R_{q}, \left(\tfrac{\phi_{1}(M^{k})}{q},\tfrac{\phi_{2}(M^{k})}{q}\right)\right)\right)=\mu^{\R}_{2}(B)\, .
    \end{equation}
    Now observe that
    \begin{align*}
        &\mu^{\R}_{2}\left(B\cap \bigcup_{1\leq q\leq M^{k}}\Delta\left(R_{q}, \left(\tfrac{\phi_{1}(M^{k})}{q},\tfrac{\phi_{2}(M^{k})}{q}\right)\right)\right) \\
        &\leq \mu^{\R}_{2}\left(B\cap \bigcup_{1\leq q< M^{k-1}}\Delta\left(R_{q}, \left(\tfrac{\phi_{1}(M^{k})}{q},\tfrac{\phi_{2}(M^{k})}{q}\right)\right)\right)    +     \mu^{\R}_{2}\left(B\cap \bigcup_{M^{k-1}\leq q\leq M^{k}}\Delta\left(R_{q}, \left(\tfrac{\phi_{1}(M^{k})}{M^{k-1}},\tfrac{\phi_{2}(M^{k})}{M^{k-1}}\right)\right)\right).
    \end{align*}
    The second summation on the right-hand side of the above inequality is precisely what we are trying to calculate (since $\tfrac{\phi_{i}(M^{k})}{M^{k-1}}=\rho_{i}(M^{k})$ for $i=1,2.$). Note \eqref{motivating example eq 1} gives us that the left-hand side of the above inequality is $\mu^{\R}_{2}(B)$, so if we can show that
    \begin{equation*}
        \mu^{\R}_{2}\left(B\cap \bigcup_{1\leq q< M^{k-1}}\Delta\left(R_{q}, \left(\tfrac{\phi_{1}(M^{k})}{q},\tfrac{\phi_{2}(M^{k})}{q}\right)\right)\right)<\tfrac{1}{2}\mu^{\R}_{2}(B)
    \end{equation*}
    we are done. To see this, we compute that
    \begin{align*}
        \mu^{\R}_{2}\left(B\cap \bigcup_{1\leq q< M^{k-1}}\Delta\left(R_{q}, \left(\tfrac{\phi_{1}(M^{k})}{q},\tfrac{\phi_{2}(M^{k})}{q}\right)\right)\right) & \leq \sum_{1\leq q<M^{k-1}}\sum_{\left(\tfrac{p_{1}}{q},\tfrac{p_{2}}{q}\right)\in R_{q}\cap B} \mu^{\R}_{2}\left(\Delta\left( \left(\tfrac{p_{1}}{q},\tfrac{p_{2}}{q}\right), \left(\tfrac{\phi_{1}(M^{k})}{q},\tfrac{\phi_{2}(M^{k})}{q}\right)\right)\right) \\
        &\leq \sum_{1\leq q <M^{k-1}} \left( 2q r(B) + 1\right)^{2} 4 q^{-2}\phi_{1}(M^{k})\phi_{2}(M^{k}) \\
        &\overset{(*)}{\leq} \sum_{1\leq q<M^{k-1}} \left(4q^{2}\mu^{\R}_{2}(B) + 4\right) 4 q^{-2} M^{-k}\\
        & \leq 16 M^{-1}\mu^{\R}_{2}(B) + 16M^{-k}\sum_{1\leq q< M^{k-1}}q^{-2} \\
        &\overset{(M\geq 64)}{\leq} \tfrac{1}{4}\mu^{\R}_{2}(B) + M^{-k}16\tfrac{\pi^{2}}{6} \\
        &\leq \tfrac{1}{2} \mu^{\R}_{2}(B),
    \end{align*}
    where $(*)$ follows on using that $(a+b)^{2}\leq 2(a^{2}+b^{2})$ for all $a,b \geq R_{+}$. Hence, the proof is complete.
\end{proof}

\subsection{Proof of Corollary 1}
Begin with the convergence case of Corollary~\ref{2 dim KG}. By the Borel-Cantelli convergence lemma, and the above formulation of $W_{1,2}^{\R}(\psi_{1},\psi_{2})$ in terms of a $\limsup$ set, we have that
\begin{equation} \label{2 dim bc}
    \mu^{\R}_{2}\left(W^{\R}_{1,2}(\psi_{1},\psi_{2})\right)=0 \quad \text{ if} \quad \sum_{q=1}^{\infty}\mu^{\R}_{2}\left( \Delta\left(R_{q}, \left(\tfrac{\psi_{1}(q)}{q},\tfrac{\psi_{2}(q)}{q} \right)\right)\right)<\infty\, .
\end{equation}
A quick calculation of the Lebesgue measure of the collection of rectangles $\Delta\left(R_{q}, \left(\tfrac{\psi_{1}(q)}{q},\tfrac{\psi_{2}(q)}{q} \right)\right)$ yields
\begin{equation*}
    \mu^{\R}_{2}\left( \Delta\left(R_{q}, \left(\tfrac{\psi_{1}(q)}{q},\tfrac{\psi_{2}(q)}{q} \right)\right)\right) \leq 4 \psi_{1}(q)\psi_{2}(q)\, ,
\end{equation*}
and inputting this into \eqref{2 dim bc} completes the convergence case.\par 
In order to prove the divergence case of Corollary~\ref{2 dim KG} we now need to show conditions (I)-(III) of Theorem~\ref{KW ambient measure} are verified for some suitably chosen functions $\rho_{1},\rho_{2}$ satisfying \eqref{2 dim rho conditions}. Notice in this setting, our functions are of the form
\begin{equation*}
    \Psi(q)=\left(\tfrac{\psi_{1}(q)}{q},\tfrac{\psi_{2}(q)}{q}\right),
\end{equation*}
So the conditions (I)-(III) correspond to
\begin{itemize}
    \item[I($\R$)] $\frac{\psi_{1}(q)}{q},\tfrac{\psi_{2}(q)}{q}$ monotonic decreasing as $q\to \infty$, 
    \item[II($\R$)] $\rho_{1}(q)\geq \tfrac{\psi_{1}(q)}{q}$ and $\rho_{2}(q)\geq \tfrac{\psi_{2}(q)}{q}$ for all $q\in\N$, and $\rho_{1}(q),\rho_{2}(q)\to 0$ as $q\to \infty$.
    \item[III($\R$)] $\frac{\psi_{1}(q)}{q}$ and $\frac{\psi_{2}(q)}{q}$ are $c$-regular on the sequence $(M^{k})_{k\geq 1}$. 
\end{itemize}
Since $\psi_{1},\psi_{2}$ are decreasing I($\R$) is immediately satisfied. For III($\R$) note that for each $i=1,2$
\begin{equation*}
    \frac{\psi_{i}(M^{k+1})}{M^{k+1}}=M^{-1}\frac{\psi_{i}(M^{k+1})}{M^{k}} \leq M^{-1}\frac{\psi_{i}(M^{k})}{M^{k}}
\end{equation*}
where the last inequality follows due to the monotonic decreasing property of $\psi_{i}$. Thus it remains to choose functions $\rho_{1},\rho_{2}$ so that II($\R$) is satisfied. \par 
For each $q\in \N$, let $l_{1}(q),l_{2}(q) \in \{1,2\}$ be the ordering such that
\begin{equation*}
    \psi_{l_{1}(q)}(q)\geq \psi_{l_{2}(q)}(q)\, .
\end{equation*}
We can assume without loss of generality that 
\begin{equation} \label{minkpsi}
\psi_{1}(q)\psi_{2}(q)<q^{-1}
\end{equation}
for all sufficiently large $q\in\N$, say $q>q_{0}$. Otherwise by Minkowski's Theorem and the monotonicity of each $\psi_{i}$ we have that $W_{1,2}^{\R}(\psi_{1},\psi_{2})=[0,1]^{2}$.
For $q>q_{0}$ consider the function
\begin{equation*}
    \phi^{*}(q)= q^{-1}\psi_{l_{1}(q)}(q)^{-1}\, .
\end{equation*}
Observe that $\phi^{*}(q)>\psi_{l_{2}(q)}(q)$, since otherwise we would not have \eqref{minkpsi}. Now, for each $q>q_{0}$ choose functions $\rho_{1},\rho_{2}$ to be
\begin{align*}
    \phi_{l_{1}(q)}(q)& = \begin{cases}
        \psi_{l_{1}(q)}(q) \quad \text{ if } \quad \psi_{l_{1}(q)}(q)>q^{-\tfrac{1}{2}}\, , \\[2ex]
        q^{-\tfrac{1}{2}}\quad \quad \quad \text{ otherwise.}
    \end{cases} \\[2ex]
    \phi_{l_{2}(q)}(q)&=\begin{cases}
        \phi^{*}(q)\,\,\,  \quad \text{ if } \quad \psi_{l_{1}(q)}(q)>q^{-\tfrac{1}{2}}\, , \\[2ex]
        q^{-\tfrac{1}{2}} \quad \quad \text{ otherwise.}
    \end{cases} \\
\end{align*}
Note $\phi_{1},\phi_{2}$ satisfy condition \eqref{2 dim rho conditions}, and furthermore
\begin{equation*}
    \rho_{i}(q)=M\frac{\phi_{i}(q)}{q}> \frac{\psi_{i}(q)}{q} \quad (i=1,2)\, ,
\end{equation*}
where the inequality follows by our choice of each $\phi_{i}$. So II($\R$) is satisfied, thus Theorem~\ref{ambient measure} is applicable. It remains to see, via Cauchy condensation, that
\begin{equation*}
    \sum_{q=1}^{\infty}\psi_{1}(q)\psi_{2}(q) \asymp \sum_{k=1}^{\infty} M^{k}\psi_{1}(M^{k})\psi_{2}(M^{k}),
\end{equation*}
and so this completes the divergence case.

\subsection{Proof of Corollary~\ref{2 dim JB}}

For the upper bound consider the standard cover
\begin{equation*}
    \bigcup_{q>N} \Delta\left(R_{q},\left(q^{-1-\tau_{1}},q^{-1-\tau_{2}}\right) \right)
\end{equation*}
for any $N\in \N$. Letting $N$ tend to infinity naturally gives rise to a better cover of $W_{1,2}^{\R}(\psi_{1},\psi_{2})$. Each layer $\Delta\left(R_{q},(q^{-1-\tau_{1}}, q^{-1-\tau_{2}})\right)$ is composed of $q^{2}$ rectangles, each of which can be covered by either one large ball with the diameter equal to the longest sidelength, or by several smaller balls with diameters equal to the shortest sidelenght of the rectangle. For now take the balls with larger diameter. Say $\tau_{1}\geq \tau_{2}$ and so $2q^{-1-\tau_{2}}$ is the larger sidelenght. Then we have that
\begin{equation*}
    \cH^{s}\left(W_{1,2}^{\R}(\psi_{1},\psi_{2})\right) \leq \sum_{q>N} q^{2} 2^{s}q^{-(1+\tau_{2})s} \leq 2^{s}\sum_{q>N} q^{2-(1+\tau_{2})s} \to 0 
\end{equation*}
as $N\to \infty$ for any $s>\frac{3}{1+\tau_{2}}$. Hence
\begin{equation*}
\dimh W_{1,2}^{\R}(\psi_{1},\psi_{2}) \leq \frac{3}{1+\tau_{2}}
\end{equation*}
when $\tau_{1}\geq \tau_{2}$. Now consider the other case. Then each rectangle in $\Delta\left(R_{q},(q^{-1-\tau_{1}},q^{-1-\tau_{2}})\right)$ can be covered by
\begin{equation*}
    \frac{2q^{-(1+\tau_{2})}}{2q^{-(1+\tau_{1}})}=q^{(\tau_{1}-\tau_{2})}
\end{equation*}
balls of diameter $2q^{-(1+\tau_{1})}$. Hence
\begin{equation*}
    \cH^{s}\left(W_{1,2}^{\R}(\psi_{1},\psi_{2})\right) \leq \sum_{q>N} q^{2} q^{(\tau_{1}-\tau_{2})} 2^{s} q^{-(1+\tau_{1})s} \leq 2^{s}\sum_{q>N} q^{2+(\tau_{1}-\tau_{2})-(1+\tau_{1})s}\to 0
\end{equation*}
as $N\to \infty$ for any $s>\frac{3+(\tau_{1}-\tau_{2})}{1+\tau_{1}}$. Hence
\begin{equation*}
\dimh W_{1,2}^{\R}(\psi_{1},\psi_{2}) \leq \frac{3+(\tau_{1}-\tau_{2})}{1+\tau_{1}}\, .
\end{equation*}
Taking the minimum over the two cases completes the upper bound of Corollary~\ref{2 dim JB}. \par 
For the lower bound we use Theorem~\ref{MTPRR} combined with Proposition~\ref{2 dim ubiquity}. Within the notation of Theorem~\ref{MTPRR} set
\begin{equation*}
    a_{i}=a_{i}^{*}+1, \quad t_{i}=\tau_{i}-a_{i}^{*} \quad (i=1,2), 
\end{equation*}
and consider the function $q\mapsto q^{-1}$. Then provided $a_{1}^{*}+a_{2}^{*}=1$ and each $a_{i}^{*}>0$ the functions $\rho_{1}(q)=q^{-a_{1}}$, $\rho_{2}(q)=q^{-a_{2}}$ are applicable to Proposition~\ref{2 dim ubiquity}. The constant $M$ appearing in the conditions of $\rho$ in Proposition~\ref{2 dim ubiquity}  can clearly be omitted. \par 
Suppose that $\tau_{1}\geq \tau_{2}$, and let
\begin{equation*}
    A:=\{a_{1},a_{2}, a_{1}+t_{1}, a_{2}+t_{2}\}
\end{equation*}
If
\begin{enumerate}
    \item[(i)] $\tau_{2}>\tfrac{1}{2}$: Then set $a_{1}^{*}=a_{2}^{*}=\tfrac{1}{2}$. So $a_{1}^{*}+a_{2}^{*}=1$. Now the sets $\cK_{1},\cK_{2},\cK_{3}$ as defined in Theorem~\ref{MTPRR} become
$$
    \begin{cases}
   \cK_{1}=\{1,2\},  \quad\cK_{2}=\emptyset, \quad \cK_{3}=\emptyset &\text{if} \quad A_{i}=a_{1}=a_{2}, \\[2ex]
   \cK_{1}=\emptyset,\quad   \cK_{2}=\{2\}, \quad \cK_{3}=\{1\}\,  &\text{if}\quad   A_{i}=a_{2}+t_{2}=\tau_{2}+1, \\[2ex]
   \cK_{1}=\emptyset,  \quad \cK_{2}=\{1,2\}, \quad \cK_{3}=\emptyset &\text{if} \quad A_{i}=a_{1}+t_{1}=\tau_{1}+1.        
    \end{cases}
$$
Inputting each of these into the formula for $s$ as in Theorem~\ref{MTPRR} we get the values
    \begin{equation*}
        \dimh W_{1,2}^{\R}(\psi_{1},\psi_{2}) \geq \min \left\{2, \frac{3}{1+\tau_{2}},\frac{3+(\tau_{1}-\tau_{2})}{1+\tau_{1}} \right\}\, .
    \end{equation*}
    
    \item [(ii)]$\tau_{2}<\tfrac{1}{2}$: Then set $a_{1}^{*}=1-\tau_{2}$ and $a_{2}^{*}=\tau_{2}$. Again, note $a_{1}^{*}+a_{2}^{*}=1$. Now the sets $\cK_{1},\cK_{2},\cK_{3}$ as defined in Theorem~\ref{MTPRR} become
    
    \begin{equation*}
    \begin{cases}
        \cK_{1}=\{1,2\}\, , \quad \cK_{2}=\emptyset \, , \quad \ \cK_{3}=\emptyset\, , \quad \quad\, \, \, \, \text{\rm if } \quad  A_{i}=a_{2}=a_{2}+t_{2}=\tau_{2}+1, \\
        \cK_{1}=\{1\}, \ \ \ \quad \cK_{2}=\{2\}, \quad  \ \cK_{3}=\emptyset \, ,\quad\quad\  {\rm if  } \ \quad A_{i}=a_{1}=2-\tau_{2},\\
        %\cK_{1}=\{1,2\}, \quad \cK_{2}=\emptyset, \quad \ \ \cK_{3}=\emptyset\, , \quad\quad \  \    {\rm if } \quad A=a_{2}+t_{2}=\tau_{2}+1,\\
        \cK_{1}=\emptyset, \quad\quad \  \cK_{2}=\{1,2\}, \  \cK_{3}=\emptyset\, , \quad\quad \, \, \,  {\rm if } \quad A_{i}=a_{1}+t_{1}=\tau_{1}+1\, .
        \end{cases}
    \end{equation*}

Inputting each of these into the formula for $s$ as in Theorem~\ref{MTPRR} we get the values
    \begin{equation*}
        \dimh W_{1,2}^{\R}(\psi_{1},\psi_{2}) 
        \geq 
        \min \left\{2, 2, 2,\frac{3+(\tau_{1}-\tau_{2})}{1+\tau_{1}} \right\}\, 
        =
        \min \left\{2, \frac{3+(\tau_{1}-\tau_{2})}{1+\tau_{1}} \right\}.
    \end{equation*}
\end{enumerate}
Combining the two cases gives us the lower bound of Corollary~\ref{2 dim JB} completing the proof.

\section{$p$-adic approximation}
 Fix a prime $p$ and $n,m \in \N$. Let $|\cdot|_{p}$ be the $p$-adic norm and $\|\bfa\|_{p}=\max\{|a_{i}|_{p}\}$ over all coordinates of $\bfa$, let $\Qp^{m\times n}$ be the set of $m\times n$ dimensional matrices with entries from the $p$-adic numbers $\Qp$, and $\Zp^{m\times n}$ the $m\times n$ matrices with entries from the $p$-adic integers $\Zp:=\{x\in\Qp:|x|_{p}\leq 1\}$. For matrix $X\in \Zp^{m\times n}$ we denote by $X_{i}$ the $i$th column vector of $X$. Let $\mu_{m\times n}^{\Qp}$ denote the $m\times n$-dimensional Haar measure on $\Qp^{m\times n}$ normalised by $\mu^{\Qp}_{m\times n}(\Zp^{m\times n})=1$.\par 
Let $\Psi:\Z^{n+m}\to \R_{+}^{n}$ be an $n$-tuple of approximation functions of the form
\begin{equation*}
    \Psi(\bfa)=\left( \frac{\psi_{1}(\|\bfa\|_{v})}{\|\bfa\|_{v}},\dots , \frac{\psi_{n}(\|\bfa\|_{v})}{\|\bfa\|_{v}} \right)
\end{equation*}
with functions $\psi_{i}:\R_{+} \to \R_{+}$. Throughout this section let $\|\cdot\|_{v}$ (not to be confused with $\|\cdot\|_{p}$ the $p$-adic max norm) be the quasi-norm $\|\bfa\|_{v}=\max_{1\leq i \leq n+m}|a_{i}|^{1/v_{i}}$ for vector $v=(v_{1},\dots,v_{n+m})$ with 
\begin{equation} \label{vector v conditions p-adic}
    v_{i}>0 \quad  (1\leq i \leq m) , \quad \sum\limits_{i=1}^{m}v_{i}=m\, , \quad v_{i+m}= 1   \quad  (1\leq i \leq n) .
\end{equation}
Define
\begin{equation*}
    W_{n,m}^{\Zp}(\Psi):= \left\{ X \in \Zp^{m\times n} : \begin{array}{l} |\bfa_{0}X_{i}-a_{i}|_{p}< \frac{\psi_{i}(\|\bfa \|_{v})}{\|\bfa\|_{v}} \quad(1\leq i \leq n) \\[2ex]
    \text{\rm for infinitely many } \bfa=(\bfa_{0},a_{1},\dots , a_{n}) \in \Z^{m+n} \end{array} \right\}.
\end{equation*}
Note that the approximation function depends on $n+m$ values, rather than $m$ values in the real setting. This is because in the $p$-adic setting, for any $x\in\Zp$, one can make the value $|x-z|_{p}$ arbitrarily small for a sufficiently large choice of $z\in\Z$, that is, $\Z$ is dense in $\Zp$. This is also why we additionally require each approximation function to decrease faster than $\frac{1}{\|\bfa\|_{v}}$, and the condition that the last $n$ terms of the vector $v$ are each equal to $1$ is needed. \par 
% For instance consider the simplified case $n=2$, $m=1$, $\|\cdot\|_{v}$ the sup norm (i.e. $v=(1,1,1)$), and suppose that $\frac{\psi_{1}(r)}{r}>r^{-1}$ for all sufficiently large $r\in \N$. Then $\cW_{2,1}((\psi_{1},\psi_{2}))=\Zp^{2}$ for any choice of function $\psi_{2}$, since for any $X\in\Zp^{2}$ there are infinitely many integers vectors (of the form $(a_{0},a_{1},a_{2})=(0,p^{k},0)$, for $k\in\N$ sufficiently large) that $(\psi_{1},\psi_{2})$-approximate $X$. \par
%\subsection{Statement of results}
We prove the following $p$-adic analogue to the weighted Khintchine-Groshev Theorem \ref{kleinbock-Wang-KG}. 
%{\color{red} Can index $r$ be zero?}
%theorem of Kleinbock and Wang \cite{KW23}. 
\begin{theorem}\label{p-adic weighted KG}
Let $\Psi:\Z^{n+m}\to \R_{+}^{n}$ be an $n$-tuple of approximation functions of the form
\begin{equation*}
    \Psi(\bfa)=\left( \frac{\psi_{1}(\|\bfa\|_{v})}{\|\bfa\|_{v}},\dots , \frac{\psi_{n}(\|\bfa\|_{v})}{\|\bfa\|_{v}} \right)
\end{equation*}
for functions $\psi_{i}:\R_{+} \to \R_{+}$ with each $\psi_{i}(r)\to\0$ monotonically as $r\to\infty$. Suppose $\boldsymbol{v}\in \R^{m+n}_{+}$ satisfies \eqref{vector v conditions p-adic}. Then
\begin{equation*}
    \mu^{\Qp}_{m\times n}\left(W_{n,m}^{\Zp}(\Psi)\right)=\begin{cases}
    0 \quad \text{\rm if }\quad  \sum\limits_{r=1}^{\infty} r^{m-1}\prod\limits_{i=1}^{n}\psi_{i}(r) < \infty, \\[2ex]
    1 \quad \text{\rm if }\quad  \sum\limits_{r=1}^{\infty} r^{m-1}\prod\limits_{i=1}^{n}\psi_{i}(r) = \infty .
    \end{cases}
\end{equation*}
\end{theorem}
We again highlight previously known results, these include
\begin{itemize}
    \item $n=m=1$, $\psi$ monotonic Jarn\'ik \cite{J28}.
    \item $n=m=1$, $\psi$ non-monotonic proven by Haynes \cite{H10} via Maynard \& Koukoulopolous \cite{KM20}. See also the more recent work of Kristensen and Laursen \cite{KristensenLaursen2023}.
    %In fact, in this paper, Haynes proved that if the variance method from probability theory can be used to solve the Duffin-Schaeffer conjecture, then almost the entire (classical) Duffin-Schaeffer conjecture will follow. Conversely, if the variance method can be used to solve completely the classical Duffin-Schaeffer conjecture, then the corresponding conjecture is true in every field $\mathbb Q_p$.
    \item $nm\geq 1$, $\psi$ monotonic univariable  proven by Lutz \cite{L55}. See also Beresnevich, Dickinson, Velani \cite[Theorem 15]{BDV06} for proof via ubiquitous systems.
    \item $n>1$ $m=1$, $\Psi$ weighted monotonic univariable, proven by Beresnevich, Levesley \& Ward \cite[Theorem 2.1]{BLW21b}. See also \cite[Theorem 2.2]{BLW21b} for a restricted non-monotonic version.
    \item $n=1$ $m\geq 1$, $\psi$ monotonic univariable and for the inhomogeneous settings by Datta \& Ghosh \cite{DatGho22}. This result with property $P$ in place of the univariable condition is claimed true in \cite[Remark 1.4 (5)]{DatGho22}.
\end{itemize}
We also provide the complimentary Hausdorff dimension result in the slightly more restrictive setup of $\boldsymbol{v}=(1,\ldots,1)$.
\begin{theorem}\label{JB p-adic}
Let $\Psi$ be of the form
\begin{equation*}
    \Psi(\bfa)=\left( \|\bfa\|_{v}^{-\tau_{1}}, \dots , \|\bfa\|_{v}^{-\tau_{n}}\right)
\end{equation*}
for vectors $\boldsymbol{v}=(1,\dots,1)$ and $\boldsymbol{\tau}=(\tau_{1}, \dots , \tau_{n}) \in \R^{n}_{+}$ with $\sum\limits_{i}\tau_{i}>m+n$ and each $\tau_{i}>1$. Then
\begin{equation*}
    \dimh W_{n,m}^{\Zp}(\Psi) = \min_{1\leq j \leq n} \left\{n(m-1) + \frac{n+m-\sum\limits_{i:\tau_{i}<\tau_{j}} (\tau_{i}-\tau_{j})}{\tau_{j}} \right\}=s
\end{equation*}
and $$\cH^{s}\left(W_{n,m}^{\Zp}(\Psi)\right)=\infty.$$
\end{theorem}
The condition that $\sum\limits_{i} \tau_{i}>m+n$ is standard. For $\sum\limits_{i} \tau_{i} \leq m+n$ by the standard version of Dirichlet's approximation theorem in the $p$-adic setting, see for example Lemma~\ref{p-adic minkowski} below, we have that $W_{n,m}^{\Zp}(\Psi)=\Zp^{m\times n}$.\par 
Note that in the simultaneous setting, the dimension and Hausdorff measure was proven in \cite{A95, BDV06} (see \cite[Theorem 16]{BDV06}), and the weighted setting with $m=1$ and $n\geq 1$ was proven in \cite[Theorem 2.4]{BLW21b}. Hence the novelty of the above two theorems is the weighted approximation case of the dual linear forms.
\par  

\subsection{Ubiquity for $p$-adics} \label{ubiquity proof} 
The key to prove the above results is the following ubiquity statement in the $p$-adic setting (Proposition~\ref{ubiquity_p-adic}). Given the notation for a ubiquitous system for rectangles as in \S \ref{ubiquity} we now define our setup. Let $\Omega_{i}=\Zp^{m}$, so $\Omega=\Zp^{m\times n}$, let $\mu=\mu^{\Qp}_{m\times n}$ and $\boldsymbol{v}\in \R^{m+n}_{+}$ satisfy \eqref{vector v conditions p-adic}. Let
% \begin{align*}
%     &\alpha=\bfa =(\bfa_{0},a_{1},\dots , a_{n}) \in \Z^{m+n}=J, \\ &\beta:J \to \R_{+}, \, \alpha \mapsto \beta_{\bfa}=\|\bfa\|_{v}, \\
%     &l_{k+1}=u_{k}=M^{k+1} \text{ for some $M\in \N$}, \\
%     &J_{k}=\left\{\alpha \in J : M^{k} \leq \|\bfa\|_{v} \leq M^{k+1} \right\}, \\
%     &R_{\bfa,i}=\left\{X_{i} \in \Zp^{m}:\bfa_{0}X_{i}-a_{i}=0\right\}, \\ &  R_{\bfa}=\prod_{i=1}^{n}R_{\bfa,i}.
% \end{align*}

\begin{enumerate}[i.]
    \item $J=\Z^{m+n}$, 
    \item $\beta:J \to \R_{+}, \, \alpha=\bfa =(\bfa_{0},\bfa_{1})=(a_{0,1},\dots, a_{0,m},a_{1},\dots , a_{n}) \mapsto \beta_{\bfa}=\|\bfa\|_{v}$,
    \item $l_{k+1}=u_{k}=M^{k+1}$ for some $M\in \N$,
    \item $J_{k}=\left\{\alpha \in J : M^{k} \leq \|\bfa\|_{v} \leq M^{k+1} \right\}$,
    \item $R_{\bfa,i}=\left\{X_{i} \in \Zp^{m}:\bfa_{0}X_{i}-a_{i}=0\right\}$,
    \item $R_{\bfa}=\prod_{i=1}^{n}R_{\bfa,i}$,
    \item $\kappa_{i}=\frac{m-1}{m}$ and $\delta_{i}=m$ for each $i=1,\dots, n$.
\end{enumerate}
For justification as to why $\kappa_{i}=\frac{m-1}{m}$ see for example \cite[Section 12.6]{BDV06} (recall within our framework $n(m-1)=\gamma=nm\kappa$). We prove the following key statement and then prove Theorems~\ref{p-adic weighted KG} and \ref{JB p-adic} in subsections \ref{ambient measure} and \ref{hausdorff dimension} respectively.
\begin{proposition} \label{ubiquity_p-adic}
Consider any ball $B=B(X,r) \subset \Zp^{m\times n}$ with centre $X\in \Zp^{m\times n}$ and radius $r>0$. Let $\rho=(\rho_{1},\dots , \rho_{n})$ be an $n$-tuple of functions $\rho_{i}:\R_{+}\to\R_{+}$ satisfying
\begin{equation}\label{rho cond}
    \rho_{i}(h)=\frac{\phi_{i}(h)}{h}, \, \, \,\, (1\leq i\leq n) \,\,\,\text{ and } \quad \prod_{i=1}^{n}\rho_{i}(h) = p^{-n} h^{-(m+n)} \quad (h\in\R_{+}),
\end{equation}
for $\phi_{i}:\R_{+} \to \R_{+}$ with each $\phi_{i}(h)\to 0$ as $h \to \infty$. Suppose that $\lambda_{0}\in \N$ and $M\in \R_{+}$ are such that
\begin{equation}  \label{lambda 0 size}
    p^{\lambda_{0}}>2^{m+n+2}p^{n}\frac{p^{m-\frac{1}{2}}}{p^{m-\frac{1}{2}}-1} 
    \end{equation}
    and
    \begin{equation} \label{M size}
    M\geq \left(p^{n(\lambda_{0}-1)}(p+1)^{n}4\right)^{\frac{1}{m+n}} \, . 
\end{equation}
Then for all sufficiently large $ k\in \N$, we have that
\begin{equation*}
    \mu^{\Qp}_{m\times n}\left( B\cap \bigcup_{\alpha \in J_{k}}\Delta\left(R_{\alpha},\rho\left(M^{k+1}\right)p^{\lambda_{0}}\right)\right)\geq \frac{1}{2}\mu^{\Qp}_{m\times n}(B).
\end{equation*}
\end{proposition}

The following lemma, which can be seen as the $p$-adic version of Minkowski's Theorem is crucial in our ubiquitous system construction.

\begin{lemma}[{\cite[Lemma 6.2]{BDGW23}}] \label{p-adic minkowski}
    Let $n,m\in\N$ and $\Psi=(\psi)_{1\leq i\leq n}$ be a n-tuple of approximation functions. Let $H_{1},\dots, H_{m+n}\geq 1$ be positive integers and let $H^{n+m}=\prod_{i=1}^{n+m}H_{i}$. Suppose that
    \begin{equation*}
        \prod_{i=1}^{n}\psi_{i}(H) \geq H^{-(n+m)}p^{-n}
    \end{equation*}
    with $\psi_{i}(H)<p^{-1}$ for all $1\leq i \leq n$. Then for any $X=(x_{i,j})\in \Zp^{m\times n}$ there exists $H_{0}$ dependent only on $\Psi$, such that for all $H\geq H_{0}$,
    \begin{equation*}
        \left| a_{0,1}x_{i,1}+\dots + a_{0,m}x_{i,m}-a_{i}\right|_{p}<\psi_{i}(H) \quad (1\leq i \leq n),
    \end{equation*}
    has a non-zero solution in integers $a_{0,1}, \dots , a_{0,m},a_{1}, \dots , a_{n}$ satisfying
    \begin{equation*}
        |a_{0,i}|\leq H_{i} \,\, \, \, (1\leq i \leq m) \quad \text{and} \quad |a_{j}|\leq H_{j+m} \,\, \, \, (1\leq j\leq n).
    \end{equation*}
\end{lemma}

\begin{remark} \rm
Many variants of this result have appeared, see for example \cite{BLW21b, KT07}. This is a well-known result that can be obtained using the standard pigeonhole principle.
\end{remark}

We also need the following lemma which generally tells us how often, for each $\bfa_{0} \in \Z^{m}$, the thickened resonant sets $R_{(\bfa_{0},\bfa_{1})}$ intersect with a ball.

\begin{lemma}\label{lines in balls}
    Let $\lambda\in\N_{0}$ and fix some $\bfa_{0}\in\Z^{m}$ with $\|\bfa_{0}\|_{p}\leq p^{-\lambda}$. Let $B\subset \Zp^{m\times n}$ with $r(B)=r<p^{-1}$, and $U \in \N$. Then for all $\delta\in \R^{n}_{>0}$ with $\|\delta\|<r$ the cardinality of the set
    \begin{equation*}
        \#\left\{ \bfa_{1}\in\Z^{n} :\begin{cases}
            \|\bfa_{1}\|\leq U\\
            \|\bfa_{1}\|_{p}\leq p^{-\lambda}
        \end{cases}  \quad \text{ and } \quad  \Delta\left(R_{(\bfa_{0},\bfa_{1})},\delta\right)\cap B \neq \emptyset \right\} 
    \end{equation*}
    % \begin{equation*}
    %     \#\left\{ \bfa_{1}\in\Z^{n} :\begin{cases}
    %         \|\bfa_{1}\|\leq U\\
    %         \|\bfa_{1}\|_{p}\leq p^{-\lambda}
    %     \end{cases}  \quad \text{ and } \quad  \Delta\left(R_{(\bfa_{0},\bfa_{1})},\rho_{i}(V)p^{\lambda_{1}}\right)\cap B \neq \emptyset \right\} 
    % \end{equation*}
    is at most
    \begin{equation*}
        \left(1+\frac{U}{p^{\lambda-1}}r\right)^{n}\, .
    \end{equation*}
\end{lemma}
\begin{proof}

Let $B=B(X,r)=\prod_{i=1}^{n}B_{i}(X_{i},r)=\prod_{i=1}^{n}B_{i}$ for $B_{i}(X_{i},r)$ representing a $m$-dimensional $p$-adic ball with centre $X_{i}=(x_{i,1},\dots, x_{i,m})\in\Zp^{m}$ and radius $r>0$ in the $i$th coordinate of the $mn$-dimensional ball $B$ with centre $X \in \Zp^{m\times n}$ and radius $r(B)=r>0$. Consider one coordinate $1\leq i \leq n$ at a time. If there exists 
\begin{equation*}
Y_{i} \in \Delta_{i}\left(R_{\bfa,i},\delta_{i}\right)\cap B_{i} \subset \Zp^{m}\, ,
\end{equation*}
then by definition there exists some $Z_{i}=(z_{i,1},\dots, z_{i,m})\in R_{\bfa,i}$ such that $Y_{i} \in B_{i}\left(Z_{i},\delta_{i}\right)\cap B_{i}$.
% \begin{equation*}
%     Y_{i} \in B_{i}\left(Z_{i},\delta_{i}\right)\cap B_{i}\, .
% \end{equation*}
By the ultrametric property of $\Qp$ we have that 
\begin{equation*}
B_{i}(X_{i},r)\cup B_{i}(Z_{i},\delta_{i})=B_{i}\left(Z_{i},r \right) \quad \quad  \text{(since $\|\delta\|<r$)}\, ,
\end{equation*}
and so $\|Z_{i}-X_{i}\|_{p}\leq r$.
% \begin{equation*}
%     \|Z_{i}-X_{i}\|_{p}=\max_{1\leq j\leq m} |z_{i,j}-x_{i,j}|_{p}\leq r\, .
% \end{equation*}
By the strong triangle inequality we have that
\begin{align} \label{eq1.1}
\left| \sum\limits_{j=1}^{m}a_{0,j}(z_{i,j}-x_{i,j})\right|_{p} &\leq \|\bfa_{0}\|_{p}r\leq p^{-\lambda}r
% & \leq \max \left\{\rho_{i}(V)p^{\lambda_{1}-\lambda}, rp^{-\lambda}\right\} :=f_{i}(V,r,\lambda,\lambda_{1})\, , \nonumber\\
\end{align}
where the second inequality follows from the assumption $\|\bfa_{0}\|_{p}\leq p^{-\lambda}$. Since $Z_{i}\in R_{\bfa,i}$ we have that $\left(\sum\limits_{j=1}^{m}a_{0,j}z_{i,j}\right) +a_{i}=0$. Combining with \eqref{eq1.1} we have
\begin{equation} \label{bound 1}
    \left| \left(\sum\limits_{j=1}^{m}a_{0,j}x_{i,j}\right)+a_{i} \right|_{p}
    %\leq \|\bfa_{0}\|_{p}r
    \leq p^{-\lambda}r,
\end{equation}
Since we are counting $\|\bfa_{1}\|_{p}\leq p^{-\lambda}$ we must have that $p^{\lambda}|a_{i}$ for each $1\leq i \leq n$. Hence 
\begin{equation} \label{bound 2}
    a_{i} \equiv 0 \mod p^{\lambda}\, .
\end{equation}
Combining \eqref{bound 2}, noting that $\|\bfa_{1}\|\leq U$, and using the assumption $r<p^{-1}$, we have that each  $a_{i}$, written in base $p$-adic expansion, is of the form
\begin{equation*}
    a_{i}=\sum_{k=\lambda}^{\lambda-\log_{p}r}d_{k}p^{k}+\sum_{j=\lambda-\log_{p}r+1}^{1+\log_{p}U}d_{j}p^{j} \quad \quad \text{ with each digit }\, \,  d_{i} \in \{0,\dots p-1\}\, ,
\end{equation*}
where $d_{k}$ is fixed for $k\in\left\{\lambda,\dots, \lambda-\log_{p}r \right\}$ depending on $X_{i}$ and $\bfa_{0}$ by \eqref{bound 1}. Thus there are at most
\begin{equation*}
    1+p^{\left(1+\log_{p}U  - \lambda+\log_{p}r \right)}=1+p^{-\lambda+1}rU
\end{equation*}
possible values of $a_{i}$. Taking the product over each coordinate axis gives us our claimed upper bound.
%Hence we have that the cardinality of possible values of each $a_{i}$ is bounded above by
% \begin{equation*}
%    1+\min\left\{\frac{U}{p^{\lambda}}, \max\left\{ \frac{U}{p^{\lambda}}\left(\rho_{i}(V)p^{\lambda_{1}+1}\right), \frac{U}{p^{\lambda}}rp\right\} \right\}.
% \end{equation*}
% Since $\rho_{i}(V)p^{\lambda_{1}+1}<r$ we have that the cardinality is bounded from above by
% \begin{equation*}
%     1+\frac{U}{p^{\lambda}}rp \, .
% \end{equation*}
%Taking the product over each coordinate axis gives us our claimed upper bound.

\end{proof}

We now proceed with the proof of Proposition~\ref{ubiquity_p-adic}. 

Let $B=B(X,r)$ with $p^{-1}> r>0$ and $X=(X_{1},\dots, X_{n})=(x_{i,j}) \in \Zp^{m\times n}$. Choose $k_{r}$ sufficiently large such that for all $k>k_{r}$ 
\begin{equation} \label{k size}
    \max_{1\leq i\leq n} \phi_{i}(M^{k+1}) < r \, , \quad M^{-\frac{1}{2n}(k+1)}<r\, ,\quad \text{ and } \quad \max_{1\leq i \leq n} \rho_{i}(M^{k+1})p^{\lambda_{0}}<r\, .
    %\max_{I\subseteq\{m+1,\dots, m+n\}} M^{-k\frac{1}{\#I}\sum\limits_{i\in I}v_{i}} < r ,
\end{equation}
%where the maximum in the second case is over all possible non-empty subsets of the digits $\{m+1,\dots,m+n\}$.
Without loss of generality we can assume that $r\in \{p^{-t} : t\in \N_{0} \}$.
Consider linear forms
\begin{equation*}
\bfa_{0}X_{i}+a_{i}=\left(\sum\limits_{j=1}^{m}a_{0,j}x_{i,j}\right) +a_{i} \quad (1\leq i \leq n) 
\end{equation*}
for $\bfa=(\bfa_{0},a_{1},\dots, a_{n})=(a_{0,1},\dots, a_{0,m},a_{1}, \dots, a_{n})\in \Z^{m+n}$. Then, by Lemma~\ref{p-adic minkowski} (the $p$-adic analogue of Minkowski Theorem) and conditions \eqref{rho cond} on $\rho$, we have that the system
\begin{equation*}
    \begin{cases}
        |\bfa_{0}X_{i}+a_{i}|_{p}\leq \rho_{i}(H) \quad (1 \leq i \leq n), \\[2ex]
        |a_{0,i}|\leq H^{v_{i}} \hspace{2cm} (1\leq i \leq m), \\[2ex]
        |a_{i}|\leq H \hspace{1.9cm} (1\leq i \leq n),
    \end{cases}
\end{equation*}
has a non-zero rational integer solution $\bfa \in \Z^{m+n}$. Furthermore, since $|a_{i}|\leq H$ and $\rho_{i}(H)<H^{-1}$ for each $1\leq i \leq n$ we are forced to conclude that $\bfa_{0}\neq 0$. We claim that if $\|\bfa_{0}\|_{p}=p^{-\lambda}$ for some $\lambda\in\N_{0}$ then $\|\bfa_{1}\|_{p}\leq p^{-\lambda}$. To see this note that $\|\bfa_{0}\|_{p}=p^{-\lambda}$ implies $|a_{0,i}|_{p}\leq p^{-\lambda}$, so $p^{\lambda}|a_{0,i}$, which implies $|a_{0,i}|\geq p^{\lambda}$ for each $1\leq i \leq m$. Since $\sum_{i=1}^{m}v_{i}=m$ and each $v_{i}>0$ for at least one term, say the $j$th term, we have $v_{j}\leq 1$. Thus, by the second row in the above system of inequalities we have
\begin{equation} \label{equation:lambda bound}
p^{\lambda}\leq |a_{0,j}|\leq H^{v_{j}}\leq H\, ,
\end{equation}
and so, for any $1\leq i \leq n$ we have $\rho_{i}(H)<H^{-1}\leq p^{-\lambda}$. Now assume that $|a_{i}|_{p}>p^{-\lambda}$ for some $1\leq i\leq n$. Then by the strong triangle inequality
\begin{equation*}
    |\bfa_{0}X_{i}+a_{i}|_{p}=|a_{i}|_{p}\leq \rho_{i}(H)\leq p^{-\lambda}\, ,
\end{equation*}
which cannot be true, so we must that that $|a_{i}|_{p}\leq p^{-\lambda}$ for every $1\leq i \leq n$ and so \mbox{$\|\bfa_{1}\|_{p}\leq p^{-\lambda}$} as claimed.\par 

% By the ultrametric properties of $p$-adic space we can immediately conclude, as in the $p$-adic non-weighted setting \cite[\S 12.6 p.81]{BDV06}, that

The system of inequalities readily implies that
\begin{equation} \label{eq1}
    \mu^{\Qp}_{m\times n}\left( B \cap \bigcup_{\alpha\in J : \|\alpha\|_{v}\leq M^{k+1}}\Delta\left(R_{\alpha},\rho(M^{k+1})\|\bfa_{0}\|_{p}^{-1}\right) \right)=\mu^{\Qp}_{m\times n}(B).
\end{equation}
Since we want the sidelengths of the rectangles in the above set to be independent of $\bfa_{0}$ we want to bound the size of $\|\bfa_{0}\|_{p}$. We need $\lambda_{0}\in\N$ to be large enough to satisfy some technical condition later, hence we have assumed condition \eqref{lambda 0 size}. \par 
% such that
% \begin{equation} \label{lambda 0 size}
% p^{\lambda_{0}}>2^{m+n+2}p^{n}\frac{p^{m-\frac{1}{2}}}{p^{m-\frac{1}{2}}-1}\, .
% \end{equation}
For each $\lambda \in \N_{0}$ consider the sets
\begin{align*}
    \widehat{J}(k,\lambda)&:=\left\{\alpha \in J : \|\alpha\|_{v}\leq M^{k+1} \quad \text{ and }\quad  \|\bfa_{0}\|_{p}=p^{-\lambda} \right\}\, , \\
    \widetilde{J}(k,\lambda_{0})&:=\left\{\alpha \in J : \|\alpha\|_{v}< M^{k} \quad \text{ and } \quad  \|\bfa_{0}\|_{p}\geq p^{-\lambda_{0}} \right\}\, ,\\
    J_{k}(\lambda_{0})&:= \left\{\alpha \in J : M^{k}\leq \|\alpha\|_{v}\leq M^{k+1} \quad \text{ and }\quad  \|\bfa_{0}\|_{p}\geq p^{-\lambda_{0}} \right\} \, ,
\end{align*}
and note that
\begin{equation*}
    \left\{\alpha\in J: \|\alpha\|_{v}\leq M^{k+1}\right\}= J_{k}(\lambda_{0})\cup \widetilde{J}(k,\lambda_{0}) \cup \bigcup_{\lambda>\lambda_{0}}\widehat{J}(k,\lambda)\, .
\end{equation*}
% \begin{equation*}
%     J_{k}\supseteq J_{k}(\lambda_{0})\cup \widetilde{J}(k,\lambda_{0}) \cup \bigcup_{\lambda>\lambda_{0}}\widehat{J}(k,\lambda)\, .
% \end{equation*}
Hence, we have that
{%\footnotesize
\begin{align} \label{composition}
    \mu^{\Qp}_{m\times n}(B)=  \mu^{\Qp}_{m\times n}\left( B \cap \bigcup_{\alpha\in J : \|\alpha\|_{v}\leq M^{k+1}}\Delta\left(R_{\alpha},\rho\left(M^{k+1}\right)\|\bfa_{0}\|_{p}^{-1}\right) \right) & \leq  \nonumber\\
    & \hspace{-5cm} \underset{:=\mu^{\Qp}_{m\times n}(A_{1})}{\underbrace{\mu^{\Qp}_{m\times n} \left( B \cap \bigcup_{\lambda>\lambda_{0}} \bigcup_{\alpha \in \widehat{J}(k,\lambda)}\Delta\left(R_{\alpha},\rho\left(M^{k+1}\right)p^{\lambda}\right) \right)}}\nonumber\\
    &\hspace{-3cm}     + \,\,\,\,\underset{:=\mu^{\Qp}_{m\times n}(A_{2})}{\underbrace{\mu^{\Qp}_{m\times n}\left( B \cap \bigcup_{\alpha\in \widetilde{J}(k,\lambda_{0})}\Delta\left(R_{\alpha},\rho\left(M^{k+1}\right)p^{\lambda_{0}}\right) \right)}}\nonumber\\
    &\hspace{-1cm}
     + \,\,\,\,\underset{:=\mu^{\Qp}_{m\times n}(A_{3})}{\underbrace{\mu^{\Qp}_{m\times n}\left( B \cap \bigcup_{\alpha\in J_{k}(\lambda_{0})}\Delta\left(R_{\alpha},\rho\left(M^{k+1}\right)p^{\lambda_{0}}\right) \right)}}\, . \nonumber\\
\end{align}}
Considering our set of interest note that $J_{k}\supseteq J_{k}(\lambda_{0})$ and so we have that 
\begin{equation*}
    \mu^{\Qp}_{m\times n}\left(B \cap \bigcup_{\alpha \in J_{k}}\Delta\left(R_{\alpha},\rho\left(M^{k+1}\right)p^{\lambda_{0}}\right)\right) \, \geq\,  \mu^{\Qp}_{m\times n}(A_{3})\,  \geq\,  \mu^{\Qp}_{m\times n}(B)-\mu^{\Qp}_{m\times n}(A_{1})-\mu^{\Qp}_{m\times n}(A_{2})\, .
\end{equation*}
Thus showing that 
\begin{equation*}
    \mu^{\Qp}_{m\times n}(A_{1})<\tfrac{1}{4}\mu^{\Qp}_{m\times n}(B) \quad \text{ and } \quad \mu^{\Qp}_{m\times n}(A_{2})<\tfrac{1}{4}\mu^{\Qp}_{m\times n}(B)
\end{equation*}
completes the proof. \par 
Firstly, for each $\alpha \in J$ and $\lambda \in \N_{0}$ note that the thickened resonant sets with non-empty intersection with the ball $B$ has measure at most
\begin{equation*}
    \mu^{\Qp}_{m\times n}\left( B\cap \Delta\left(R_{\alpha},\rho\left(M^{k+1}\right)p^{\lambda}\right)\right) \leq  p^{n\lambda}r^{n(m-1)}\prod_{i=1}^{n}\rho_{i}\left(M^{k+1}\right)=p^{n\lambda-n}r^{n(m-1)}M^{-(k+1)(m+n)}\, .
\end{equation*}
Starting with the measure of $A_{1}$ observe that
\begin{align*}
    \mu^{\Qp}_{m\times n}(A_{1}) & \leq \sum_{\lambda>\lambda_{0}}\sum_{\alpha \in \widehat{J}(k,\lambda)} p^{n\lambda-n}r^{n(m-1)}M^{-(k+1)(m+n)} \\
    &\leq \sum_{\lambda>\lambda_{0}}\sum_{\|\bfa_{0}\|_{v}\leq M^{k+1} : \|\bfa_{0}\|_{p}=p^{-\lambda}} \left(1+\frac{M^{k+1}}{p^{\lambda-1}}r\right)^{n} p^{n\lambda-n}r^{n(m-1)}M^{-(k+1)(m+n)}\\
\end{align*}
by Lemma~\ref{lines in balls}. Observe Lemma~\ref{lines in balls} is applicable because $\|\bfa_{0}\|_{p}=p^{-\lambda}$ implies that $\|\bfa_{1}\|_{p}\leq p^{-\lambda}$ and, since $p^{\lambda}\leq M^{k+1}$ (see \eqref{equation:lambda bound} and replace $H$ by $M^{k+1}$), we have $\rho_{i}(M^{k+1})p^{\lambda}<r$ by \eqref{k size}. Then
\begin{align}
    \mu^{\Qp}_{m\times n}(A_{1}) & \leq \sum_{\lambda>\lambda_{0}} \left(2\frac{M^{k+1}}{p^{\lambda}}\right)^{m}\left(1+\frac{M^{k+1}}{p^{\lambda-1}}r\right)^{n} p^{n\lambda-n}r^{n(m-1)}M^{-(k+1)(m+n)}\nonumber\\
    &\leq 2^{m}p^{-n}\sum_{\lambda>\lambda_{0}}M^{-(k+1)n}r^{n(m-1)}\left(1+\frac{M^{k+1}}{p^{\lambda-1}}r\right)^{n}p^{(n-m)\lambda} \label{eq2 case}\\
    &= 2^{m}p^{-n}\sum_{\lambda>\lambda_{0}}\left(M^{-(k+1)}p^{\lambda(1-\frac{1}{2n})}\right)^{n}r^{n(m-1)}\left(1+\frac{M^{k+1}}{p^{\lambda-1}}r\right)^{n}p^{(-m+\frac{1}{2})\lambda} \label{eq1 case}
\end{align}
Now, suppose $\frac{M^{k+1}}{p^{\lambda}}r<1$. Then since $p^{\lambda}\leq M^{k+1}$, and by \eqref{k size} for all $k>k_{r}$, we have that
\begin{equation*}
    \left(M^{-(k+1)}p^{\lambda(1-\frac{1}{2n})}\right)^{n}<r^{n}\, .
\end{equation*}
If $\frac{M^{k+1}}{p^{\lambda}}r\geq 1$ then trivially
\begin{equation*}
    \left(1+\frac{M^{k+1}}{p^{\lambda-1}}r\right)^{n} \leq 2^{n}p^{n}\left( \frac{M^{k+1}}{p^{\lambda}}r\right)^{n}\, .
\end{equation*}
So
\begin{equation*}
    \left(M^{-(k+1)}p^{\lambda(1-\frac{1}{2n})}\right)^{n}r^{n(m-1)}\left(1+\frac{M^{k+1}}{p^{\lambda-1}}r\right)^{n}p^{(-m+\frac{1}{2})\lambda} \leq 2^{n}r^{nm}p^{(-m+\frac{1}{2})\lambda} \quad \text{ if } \quad \frac{M^{k+1}}{p^{\lambda}}r<1\, ,
\end{equation*}
and 
\begin{equation*}
    M^{-(k+1)n}r^{n(m-1)}\left(1+\frac{M^{k+1}}{p^{\lambda-1}}r\right)^{n}p^{(n-m)\lambda} \leq 2^{n}p^{n}r^{nm}p^{-m\lambda} \quad \text{ if } \quad  \frac{M^{k+1}}{p^{\lambda}}r\geq 1\, .
\end{equation*}
Hence
\begin{equation*}
    \mu^{\Qp}_{m\times n}(A_{1}) \leq 2^{m+n}p^{n}r^{nm}\sum_{\lambda>\lambda_{0}}p^{(-m+\frac{1}{2})\lambda} \leq \mu^{\Qp}_{m\times n}(B)2^{m+n}p^{n}\frac{p^{m-\frac{1}{2}}}{p^{m-\frac{1}{2}}-1}p^{-\lambda_{0}}
\end{equation*}
and so by \eqref{lambda 0 size} we have that
\begin{equation*}
    \mu^{\Qp}_{m\times n}(A_{1})\leq \frac{1}{4}\mu^{\Qp}_{m\times n}(B)
\end{equation*}
as required. \par 
To show that $\mu^{\Qp}_{m\times n}(A_{2})<\tfrac{1}{4}\mu^{\Qp}_{m\times n}(B)$ observe that
\begin{equation*}
    \mu^{\Qp}_{m\times n} (A_{2}) \leq \mu^{\Qp}_{m\times n}\left( B \cap \bigcup_{\alpha\in J : \|\alpha\|_{v}\leq M^{k}}\Delta\left(R_{\alpha},\rho\left(M^{k+1}\right)p^{\lambda_{0}}\right) \right)\,
\end{equation*}
and furthermore that,
\begin{align*}
    &\mu^{\Qp}_{m\times n}\left( B \cap \bigcup_{\alpha\in J : \|\alpha\|_{v}\leq M^{k}}\Delta\left(R_{\alpha},\rho\left(M^{k+1}\right)p^{\lambda_{0}}\right) \right) \\
    & \hspace{1cm} \leq \sum\limits_{\|\bfa_{0}\|_{(v_{1},\dots,v_{m})}\leq M^{k}} \, \underset{\Delta(R_{\bfa},\rho(M^{k+1})p^{\lambda_{0}}) \cap B \neq \emptyset}{\sum\limits_{\|(a_{1},\dots , a_{n})\|\leq M^{k}:}} \mu^{\Qp}_{m\times n}\left(B\cap \Delta(R_{\bfa},\rho(M^{k+1})p^{\lambda_{0}}\right).
\end{align*}

% Let $I$ denote the subset of digits $t \in \{m+1, \dots , m+n\}$ where
% \begin{equation*}
% M^{kv_{t}}r<1, \quad\text{and}\quad I^{c}=\{m+1, \dots , m+n\}\backslash I.
% \end{equation*}

Now,
\begin{equation*}
    \mu^{\Qp}_{m\times n}\left(A_{2} \right) \leq \left( \prod_{i=1}^{m}M^{kv_{i}} \right)  \left( pM^{k}r+1\right)^{n} \mu^{\Qp}_{m \times n}\left(B\cap \Delta\left(R_{\bfa},\rho(M^{k+1})p^{\lambda_{0}}\right)\right) \\
\end{equation*}
by Lemma~\ref{lines in balls} setting $U=M^{k}$, $\lambda=0$, and noting \eqref{k size} tells us that $\rho_{i}(M^{k+1})p^{\lambda_{0}}<r$. \par 
Then inputting the previous upper bound on the measure of such thickened resonant sets, we have
\begin{equation*}
    \mu^{\Qp}_{m\times n}(A_{2}) \leq M^{k\sum_{i=1}^{m}v_{i}} \left( pM^{k}r+1\right)^{n} p^{n\lambda_{0}-n}r^{n(m-1)}M^{-(k+1)(m+n)}. 
\end{equation*}
Since $k$ is chosen sufficiently large such that $M^{k}r>1$ we have that
\begin{equation*}
    \mu^{\Qp}_{m\times n}(A_{2}) \leq (p+1)^{n}r^{nm} p^{n\lambda_{0}-n} M^{-(m+n)} \leq (p+1)^{n}p^{n(\lambda_{0}-1)}M^{-(n+m)}\mu^{\Qp}_{m\times n}(B)
\end{equation*}
Thus, by \eqref{M size} we have that
\begin{equation*}
    \mu^{\Qp}_{m\times n}(A_{2})\leq \frac{1}{4}\mu^{\Qp}_{m\times n}(B)\, .
\end{equation*}
Hence, by considering \eqref{composition} and the above two calculations, we have completed the proof.

%&\hspace{1cm} \leq \left( \prod_{i=1}^{m}M^{kv_{i}} \right) \left( \prod_{i\in I^{c}}3M^{kv_{i}}r\right)2^{\#I} r^{n(m-1)}\prod_{i=1}^{n}\rho_{i}(M^{k+1}) \\
%     & \hspace{1cm} \overset{\eqref{rho cond}}{=} \left( \prod_{i=1}^{m}M^{kv_{i}} \right) \left( \prod_{i\in I^{c}}3M^{kv_{i}}r\right)2^{\#I} r^{n(m-1)} p^{-n}\prod_{i=1}^{n+m}M^{-(k+1)v_{i}} \\
%     & \hspace{1cm} \leq p^{-n}3^{\#I^{c}}2^{\#I} \left(\prod_{i\in I}M^{-kv_{i}} \right)\left( \prod_{i=1}^{n+m}M^{kv_{i}} \right)\left( \prod_{i=1}^{n+m}M^{-(k+1)v_{i}} \right)r^{nm-(n-\#I^{c})} \\
%     & \hspace{1cm} \leq\frac{p^{-n}3^{n}2^{n}}{M^{\tilde{v}}}r^{nm} \left( M^{-k\sum\limits_{i\in I}v_{i}}r^{-\#I}\right) \\
%     &\hspace{1cm} \overset{\eqref{M size}}{\leq}  \frac{1}{2}\mu^{\Qp}_{m\times n}(B) \left( M^{-k\sum\limits_{i\in I}v_{i}}r^{-\#I}\right) \\
%     & \hspace{1cm} \overset{\eqref{k size}}{\leq} \frac{1}{2}\mu^{\Qp}_{m\times n}(B).
% \end{align*}
% Combining this with \eqref{eq1} and \eqref{eq2} we obtain
% \begin{equation*}
%     \mu^{\Qp}_{m\times n}\left( B \cap \bigcup_{\bfa\in J_{k}}\Delta(R_{\bfa},\rho(M^{k+1})) \right)\geq \frac{1}{2}\mu^{\Qp}_{m\times n}(B)
% \end{equation*}
% as required.

\subsection{Proof of Theorem~\ref{p-adic weighted KG}} \label{ambient measure}
%We will, on occasion, use the following estimate
% \begin{equation} \label{counting integers}
% \#J_{k}=\# \left\{ \bfa \in \Z^{n+m} : M^{k}<\|\bfa\|_{v}\leq M^{k+1} \right\} \leq c M^{k(n+m)}
% \end{equation}
% for constant $c=2(n+m)3^{n+m-1}\left(M^{n+m}-M^{n+m-\min_{i} v_{i}}\right)$ independent of $k$. To see this, observe that
% \begin{align*}
%     \sum\limits_{M^{k}<\|\bfa\|_{v}\leq M^{k+1}} 1 & = \sum\limits_{j=1}^{n+m} \underset{|a_{t}|^{1/v_{t}}\leq M^{k+1} \, , \, 1\leq t \leq n+m \, \, t \neq j}{\sum\limits_{M^{k}<|a_{j}|^{1/v_{j}}\leq M^{k+1}}}1  \\
%     & \leq \sum\limits_{j=1}^{n+m}3^{n+m-1}M^{(k+1)\sum\limits_{t=1 \, \, t \neq j}^{n+m}v_{t}}2\left( M^{(k+1)v_{j}}-M^{k v_{j}}\right) \\
%     &= 3^{n+m-1}2\sum\limits_{j=1}^{n+m} M^{(k+1)\sum\limits_{t=1 \, \, t \neq j}^{n+m}v_{t}}M^{k v_{j}}\left( M^{v_{j}}-1\right) \\
%     & \leq 3^{n+m-1}2(n+m)\max_{1\leq j \leq n+m} M^{k(n+m)}M^{\sum\limits_{t=1 \, \, t \neq j}^{n+m}v_{t}}\left( M^{v_{j}}-1\right) \\
%     & = 3^{n+m-1}2(n+m)\left(M^{n+m}-M^{n+m-\min_{1\leq j \leq n+m}v_{j}}\right)M^{k(n+m)}.
% \end{align*}
Recall
\begin{align*}
     W_{n,m}^{\Zp}(\Psi)&:= \left\{ X \in \Zp^{m\times n} : \begin{array}{c} |\bfa_{0}X_{i}-a_{i}|_{p}< \frac{\psi_{i}(\|\bfa \|_{v})}{\|\bfa\|_{v}} \quad (1\leq i \leq n) \\
    \text{ for infinitely many  } \bfa=(\bfa_{0},a_{1},\dots , a_{n}) \in \Z^{m+n} \end{array} \right\}\\
    &=\limsup_{\bfa\in\Z^{m+n}}\left\{X\in \Zp^{m\times n}:|\bfa_{0}X_{i}-a_{i}|_{p}<\frac{\psi_{i}(\|\bfa\|_{v})}{\|\bfa\|_{v}} \quad (1\leq i \leq n) \right\}. \\
\end{align*}
For ease of notation, we will write for all $\mathbf{a}=(\mathbf{a}_{0},a_{1},\dots,a_{n})\in\Z^{m + n}$
\begin{equation*}
    \cA_{\bfa}(\Psi)=\left\{X\in \Zp^{m\times n}:|\bfa_{0}X_{i}-a_{i}|_{p}<\frac{\psi_{i}(\|\bfa\|_{v})}{\|\bfa\|_{v}} \quad  (1\leq i \leq n) \right\}.
\end{equation*}

\subsubsection{Convergence case of Theorem~\ref{p-adic weighted KG}}
By the Borel-Cantelli convergence Lemma (Lemma~\ref{Borel_Cantelli convergence}), we have that
\begin{equation*}
    \mu^{\Qp}_{m\times n}\left(W_{n,m}^{\Zp}(\Psi)\right)=0 \quad \text{\rm if }\quad  \sum\limits_{\bfa \in\Z^{n+m}}\mu^{\Qp}_{m\times n}\left(\cA_{\bfa}(\Psi)\right)<\infty.
\end{equation*}
Firstly, note for each $\bfa\in \Z^{n+m}$ by applying the strong triangle inequality in each coordinate axis we have that
\begin{equation*}
    \mu_{m\times n}^{\Qp}(\cA_{\bfa}(\Psi))\asymp_{p,n}\begin{cases}
        \prod_{i=1}^{n}\frac{\psi_{i}(\|\bfa\|_{v})}{\|\bfa\|_{v}}\|\bfa_{0}\|_{p}^{-1}\quad \text{ if } \quad |a_{i}|_{p}\leq \max \left\{\|\bfa_{0}\|_{p}, \frac{\psi_{i}(\|\bfa\|_{v})}{\|\bfa\|_{v}} \right\}\, \quad \text{ for every } \, 1\leq i \leq n\, ,\\
        0 \hspace{20ex} \text{otherwise.}
    \end{cases}
\end{equation*}
Let
\begin{equation*}
    \cC(k,t):=\left\{\bfa\in \Z^{m+n}: \begin{array}{c}
    M^{k}<\|\bfa\|_{v}\leq M^{k+1} \\
    \|\bfa_{0}\|_{p}=p^{-t}\\
    |a_{i}|_{p} \leq \max\left\{ p^{-t}, \frac{\psi_{i}(\|\bfa\|_{v})}{\|\bfa\|_{v}} \right\}\, , 1\leq i\leq n
    \end{array} \right\}
\end{equation*}
Observe that the set is empty for $t>\log_{p}M^{k+1}$ since $\|\bfa_{0}\|_{p}\geq M^{-(k+1)}$, and we have that 
\begin{equation*}
    \mu_{m\times n}^{\Qp}(\cA_{\bfa}(\Psi))=0\quad \text{ if } \quad \bfa\not\in \bigcup_{k=1}^{\infty}\bigcup_{t=0}^{\log_{p}M^{k+1}}\cC(k,t)\, .
\end{equation*}
Hence
\begin{align}
    \sum\limits_{\bfa \in \Z^{n+m}}\mu^{\Qp}_{m\times n}(\cA_{\bfa}(\Psi))&=\sum\limits_{k=1}^{\infty} 
  \sum_{\substack{\bfa\in\Z^{m+n}:\\M^{k}<\|\bfa\|_{v}\leq M^{k+1}}} \mu^{\Qp}_{m\times n}(\cA_{\bfa}(\Psi))\nonumber\\
  & \asymp_{p,n} \sum_{k=1}^{\infty}\sum_{t=0}^{\log_{p}M^{k+1}}\sum_{\bfa\in \cC(k,t)} \prod_{i=1}^{n}\frac{\psi_{i}(\|\bfa\|_{v})}{\|\bfa\|_{v}} \|\bfa_{0}\|_{p}^{-1} \nonumber\\
  &\leq \sum_{k=1}^{\infty}\left(\prod_{i=1}^{n}\frac{\psi_{i}(M^{k})}{M^{k}}\right)\sum_{t=0}^{\log_{p}M^{k+1}}\#\cC(k,t) p^{nt}\, . \label{eq:convergence case}
\end{align}
To bound the cardinality of $\cC(k,t)$ note that if $\|\bfa_{0}\|_{p}=p^{-t}$ then $p^{t}|a_{0,i}$ for each $1\leq i\leq m$ and so $\min_{i}|a_{0,i}|\geq p^{t}$. Furthermore, since $\sum_{i=1}^{m}v_{i}=m$ there exists $1\leq j\leq m$ such that $v_{j}\leq 1$ and so $|a_{0,j}|^{\frac{1}{v_{j}}}\geq p^{t}$. Hence $\|\bfa\|_{v}\geq p^{t}$. Since each $\psi_{i}(r)\to 0$ as $r\to \infty$ we can assume without loss of generality $\psi_{i}(r)<1$ for all $r\geq1$, and so if $\|\bfa_{0}\|_{p}=p^{-t}$ we must have
\begin{equation*}
    |a_{i}|_{p}\leq \max\{p^{-t}, p^{-t}\psi_{i}(\|\bfa\|_{v})\}\leq p^{-t} \quad \text{ for each } \, \, 1\leq i\leq n\, .
\end{equation*}
Thus,
\begin{align*}
    \#\cC(k,t) &\leq \#\left\{ \bfa_{0}\in \Z^{m}: \begin{array}{c}\|\bfa_{0}\|_{v}\leq M^{k+1}\\
    \|\bfa_{0}\|_{p}=p^{-t}
    \end{array} \right\} \times  \#\left\{ \bfa_{1}\in \Z^{n}: \begin{array}{c} \|\bfa_{1}\|\leq M^{k+1} \\
    |a_{i}|_{p}\leq p^{-t}
    \end{array}\right\}\\
    & \leq \#\{\bfa_{0}\in p^{t}\Z^{m}: \|\bfa_{0}\|_{v}\leq M^{k+1}\}\times \#\{\bfa_{1}\in p^{t}\Z^{n}: \|\bfa_{1}\|\leq M^{k+1}\}\\
    & \leq \left(\frac{3 M^{k+1}}{p^{t}}\right)^{m+n}
\end{align*}
Applying this to \eqref{eq:convergence case} we have that
\begin{align*}
     \sum\limits_{\bfa \in \Z^{n+m}}\mu^{\Qp}_{m\times n}(\cA_{\bfa}(\Psi))& \leq \sum_{k=1}^{\infty}\left(\prod_{i=1}^{n}\frac{\psi_{i}(M^{k})}{M^{k}}\right)\sum_{t=0}^{\log_{p}M^{k+1}}\left(\frac{3 M^{k+1}}{p^{t}}\right)^{m+n} p^{nt}\\
     &\leq (3M)^{m+n}\sum_{k=1}^{\infty}M^{k(m+n)}\left(\prod_{i=1}^{n}\frac{\psi_{i}(M^{k})}{M^{k}}\right)\sum_{t=0}^{\log_{p}M^{k+1}}p^{-mt}\\
     &\leq (3M)^{m+n}\frac{p^{m}}{p^{m}-1} \sum_{k=1}^{\infty}M^{k(m+n)}\left(\prod_{i=1}^{n}\frac{\psi_{i}(M^{k})}{M^{k}}\right)\\
     &\leq (3M)^{m+n}\frac{p^{m}}{p^{m}-1} \sum_{k=1}^{\infty}M^{k}\left(M^{k(m-1)}\prod_{i=1}^{n}\psi_{i}(M^{k})\right) 
\end{align*}

% Observe that
% \begin{align*}
%     \sum\limits_{\bfa \in \Z^{n+m}}\mu^{\Qp}_{m\times n}(\cA_{\bfa}(\Psi))&=\sum\limits_{k=1}^{\infty} 
%   \sum_{\substack{\bfa\in\Z^{m+n}:\\M^{k}<\|\bfa\|_{v}\leq M^{k+1}}} \mu^{\Qp}_{m\times n}(\cA_{\bfa}(\Psi)) \\
%     &\asymp_{p,n}\sum\limits_{k=1}^{\infty}\sum_{\substack{\bfa\in\Z^{m+n}:\\M^{k}<\|\bfa\|_{v}\leq M^{k+1}}} \prod_{i=1}^{n}\frac{\psi_{i}(\|\bfa\|_{v})}{\|\bfa\|_{v}} \\
%     &\leq \sum\limits_{k=1}^{\infty}\sum_{\substack{\bfa\in\Z^{m+n}:\\M^{k}<\|\bfa\|_{v}\leq M^{k+1}}} \prod_{i=1}^{n}\frac{\psi_{i}(M^{k})}{M^{k}}\\
%     &\leq \sum\limits_{k=1}^{\infty}cM^{k(n+m)} \prod_{i=1}^{n}\frac{\psi_{i}(M^{k})}{M^{k}}\\
%     &\leq c\sum\limits_{k=1}^{\infty}M^{k}\left(M^{k(m-1)} \prod_{i=1}^{n}\psi_{i}(M^{k})\right).
% \end{align*}
Note that we can trivially assume 
\begin{equation*}
    \prod_{i=1}^{n}\psi_{i}(r)<p^{-n}r^{-m}
\end{equation*}
for all $r\in\N$, otherwise by Lemma~\ref{p-adic minkowski} we have $W_{n,m}^{\Zp}(\Psi)=\Zp^{m\times n}$. Hence we can assume $$M^{k(m-1)} \prod_{i=1}^{n}\psi_{i}(M^{k})$$ is decreasing in $k\in\N$, and so by Cauchy condensation we have that
\begin{align*}
 \sum\limits_{\bfa \in \Z^{n+m}}\mu^{\Qp}_{m\times n}(\cA_{\bfa}(\Psi)) \ll_{M,p,n,m} \sum\limits_{r=1}^{\infty}r^{m-1}\prod_{i=1}^{n}\psi_{i}(r)\, .
\end{align*}
Hence the convergence case follows by the Borel-Cantelli Lemma and convergence assumption on $\sum\limits_{r=1}^{\infty}r^{m-1}\prod_{i=1}^{n}\psi_{i}(r)$. \par

\subsubsection{Divergence case of Theorem~\ref{p-adic weighted KG}}

To prove the divergence case of Theorem~\ref{p-adic weighted KG}, we use Proposition~\ref{ubiquity_p-adic} and Theorem~\ref{KW ambient measure}. We need to find an $n$-tuple of functions $\rho=(\rho_{1},\dots, \rho_{n})$ such that \eqref{rho cond} and the conditions (I)-(III) are satisfied for any chosen $\Psi$. Importantly observe that our $n$-tuple of functions $\Psi$ are of the form $\frac{\psi_{i}(r)}{r}$ in each coordinate axis $1\leq i \leq n$, and so we are asking that:
\begin{enumerate}
    \item[I($\Qp$).] For each $1\leq i \leq n$, the function $\frac{\psi_{i}(r)}{r}$ is montonic decreasing as $r\to \infty$, \\
    \item[II($\Qp$).] For each $1\leq i \leq n$, we have $\frac{\psi_{i}(r)}{r}\leq \rho_{i}(r)$ and $\rho_{i}(r)\to 0$ as $r\to \infty$,\\
    \item[III($\Qp$).] For each $1\leq i \leq n$, the function $\frac{\psi_{i}(r)}{r}$ is $c$-regular on $(M^{k})_{k\geq 1}$ for some fixed constant $0<c<1$.
\end{enumerate} 
Firstly, by assumption, each $\psi_{i}$ is monotonic so I($\Qp$) is satisfied. Secondly, note that $\Psi$ is $M^{-1}$-regular on $(M^{k})_{k\geq 1}$ since in each coordinate
\begin{equation*}
    \frac{\psi_{i}(M^{k+1})}{M^{k+1}} \leq \frac{\psi_{i}(M^{k})}{M^{k+1}} \leq M^{-1}\frac{\psi_{i}(M^{k})}{M^{k}},
\end{equation*}
where the first inequality follows due to the monotonicity of each $\psi_{i}$, so III($\Qp$) is satisfied. Thus, it remains to choose functions $\rho_{i}$ such that II($\Qp$) and \eqref{rho cond} are satisfied. 

At each $u\in\N$ let $l(u)_{1}, \dots , l(u)_{n}$ be an ordering of the digits $1, \dots , n$ such that
\begin{equation*}
    \psi_{l(u)_{1}}(u)\geq \psi_{l(u)_{2}}(u) \geq \dots \geq \psi_{l(u)_{n}}(u).
\end{equation*}
For each $u\in\N$ there exists unique $0\leq j \leq n-1$ such that, for $\nu(u,j)$ defined by
\begin{equation*}
    \nu(u,j)=\left(p^{-n}u^{-m}\left(\prod_{i=1}^{j}\psi_{l(u)_{i}}(u)\right)^{-1}\right)^{\frac{1}{n-j}},
\end{equation*}
we have that
\begin{equation*}
    \psi_{l(u)_{1}}(u)\geq \dots \geq \psi_{l(u)_{j}}(u)\geq \nu(u,j) > \psi_{l(u)_{j+1}}(u) \geq \dots \geq \psi_{l(u)_{n}}(u).
\end{equation*}
We can assume without loss of generality that $\nu(u,n-1)>\psi_{l(u)_{n}}(u)$ for all sufficiently large $u\in\N$, and so such $0\leq j\leq n-1$ exists. Otherwise, we would have that for infinitely many $u\in\N$
\begin{equation*}
    p^{-n}u^{-(n+m)}<\prod_{i=1}^{n}\frac{\psi_{l(u)_{i}}(u)}{u}
\end{equation*}
and so, by applying Lemma~\ref{p-adic minkowski}, we could conclude that $W_{n,m}^{\Zp}(\Psi)=I_{p}^{m\times n}$, thus completing the proof.\par 
For each $u \in\N$ set
\begin{equation*}
\begin{cases}
    \phi_{l(u)_{i}}(u)=\psi_{l(u)_{i}}(u) \quad \quad  (1 \leq i \leq j), \\
    \phi_{l(u)_{i}}(u)=\nu(u,j) \quad \quad \,\,\,\,\,\,\,\,\, (j+1\leq i \leq n).
\end{cases}
\end{equation*}
Note that, by definition, $\phi_{i}(u)\geq \psi_{i}(u)$, and that $\phi_{i}(u)\to 0$ as $u\to \infty$ for each $1\leq i \leq n$ since either
\begin{equation*}
\begin{cases}
    \phi_{i}(u)=\psi_{i}(u)\to 0, \\
    \phi_{i}(u)=\nu(u,j)\leq \psi_{j}(u) \to 0 \quad \quad \quad  (1\leq j \leq n-1),\\
    \phi_{i}(u)=\nu(u,0)=p^{-1}u^{-\frac{m}{n}} \to 0
\end{cases}
\end{equation*}
as $u \to \infty$. For each $1\leq i \leq n$ define
\begin{equation*}
    \rho_{i}(u)=\frac{\phi_{i}(u)}{u}\geq \frac{\psi_{i}(u)}{u}\, ,
\end{equation*}
thus II($\Qp$) is satisfied. Lastly, observe that for each $u\in\N$
\begin{align*}
    \prod_{i=1}^{n}\rho_{i}(u)&=\prod_{i=1}^{n}\frac{\phi_{l(u)_{i}}(u)}{u}\\
    &=u^{-(n-j)}\left(p^{-n}u^{-m}\left(\prod_{i=1}^{j}\psi_{l(u)_{i}}(u)\right)^{-1}\right) \times\prod_{i=1}^{j}\frac{\psi_{l(u)_{i}}(u)}{u} \\
    &=p^{-n}u^{-(n+m)},
\end{align*}
so \eqref{rho cond} is satisfied.

Hence, Theorem~\ref{KW ambient measure} and Proposition~\ref{ubiquity} are applicable. Thus $\mu^{\Qp}_{m\times n}\left(W_{n,m}^{\Zp}(\Psi)\right)=1$ if
\begin{equation*}
    \sum\limits_{k=1}^{\infty}\prod_{i=1}^{n}\left(\frac{\psi_{i}(u_{k})}{u_{k}\rho_{i}(u_{k})}\right)^{m\left(1-\tfrac{m-1}{m}\right)}=\infty.
\end{equation*}

Inputting our chosen $\rho$ we have that
\begin{equation*}
     \sum\limits_{k=1}^{\infty} u_{k}^{n+m}\prod_{i=1}^{n}\frac{\psi_{i}(u_{k})}{u_{k}}  \asymp \sum\limits_{k=1}^{\infty} u_{k}^{m}\prod_{i=1}^{n}\psi_{i}(u_{k})  \asymp \sum\limits_{r=1}^{\infty} r^{m-1}\prod_{i=1}^{n}\psi_{i}(r),
\end{equation*}
where the last line follows by Cauchy condensation with our choice of $(u_{k})_{k\geq 1}=(M^{k})_{k\geq 1}$. This completes the proof of Theorem~\ref{p-adic weighted KG}.

\subsection{Proof of Theorem~\ref{JB p-adic}} \label{hausdorff dimension}
Recall that
\begin{equation*}
    \Psi(r)=(r^{-\tau_{1}}, \dots , r^{-\tau_{n}}),
\end{equation*}
and for ease of notation, let
\begin{equation*}
    \cA_{\bfa}(\tau)=\left\{X\in\Zp^{m\times n}:|\bfa_{0}X_{i}-a_{i}|_{p}<\|\bfa\|_{v}^{-\tau_{i}} \quad (1\leq i \leq n) \right\},
\end{equation*}
so that
\begin{equation*}
    W_{n,m}^{\Zp}(\Psi)=\limsup_{\bfa\in\Z^{n+m}}\cA_{\bfa}(\tau).
\end{equation*}
The upper bound uses a standard covering argument, but for completeness, we include it here. Note, for any $N\in\N$
\begin{equation*}
    W_{n,m}^{\Zp}(\tau) \subset \bigcup_{r\geq N} \bigcup_{\|\bfa\|_{v}=r}\cA_{\bfa}(\tau)
\end{equation*}
is a cover of $W_{n,m}^{\Zp}(\tau)$ by rectangles. Since the Hausdorff dimension is determined by covers of balls we consider each $1\leq j\leq n$ and cover the collections of rectangles $\cA_{\bfa}(\tau)$ in the layer $\|\bfa\|_{v}=r$ by balls of radius $\|\bfa\|_{v}^{-\tau_{j}}$ to obtain our upper bound. Observe that in each coordinate axis
\begin{equation*}
    \cA_{\bfa,i}(\tau)=\left\{X_{i}\in\Zp^{m}:|\bfa_{0}X_{i}-a_{i}|<\|\bfa\|_{v}^{-\tau_{i}} \right\}
\end{equation*}
can be covered by
\begin{equation}
    \asymp \max\left\{1, \frac{\|\bfa\|_{v}^{-\tau_{i}}}{\|\bfa\|_{v}^{-\tau_{j}}}\right\}\times \left(\frac{1}{\|\bfa\|_{v}^{-\tau_{j}}}\right)^{m-1}
\end{equation}
balls of size $\|\bfa\|_{v}^{-\tau_{j}}$. So, for any $s>0$,
\begin{align*}
    \cH^{s}\left(W_{n,m}^{\Zp}(\Psi)\right) & \ll \sum\limits_{r\geq N} r^{-s\tau_{j}}r^{n+m-1}r^{n(m-1)\tau_{j}}\prod_{i=1}^{n}\max\{1,r^{\tau_{j}-\tau_{i}}\} \\
    & \ll \sum\limits_{r\geq N} r^{-\tau_{j}s+ n+m-1+n(m-1)\tau_{j}+\sum\limits_{i:\tau_{j}>\tau_{i}}(\tau_{j}-\tau_{i})},
\end{align*}
which converges for all
\begin{equation*}
    s>s_{j}:= n(m-1)+\frac{n+m-\sum\limits_{i:\tau_{j}>\tau_{i}}(\tau_{i}-\tau_{j})}{\tau_{j}}.
\end{equation*}
Hence by definition of Hausdorff dimension $\dimh W_{n,m}^{\Zp}(\Psi) \leq s_{j}$. Since we can do the same calculation as above for each $1\leq j \leq n$, the upper bound of Theorem~\ref{JB p-adic} follows. \par 
For the lower bound of Theorem~\ref{JB p-adic}, we use Theorem~\ref{MTPRR} combined with Proposition~\ref{ubiquity}. Assume without loss of generality that
\begin{equation*}
    \tau_{1}\geq \tau_{2} \geq \dots \geq \tau_{n} >1.
\end{equation*}
In Proposition~\ref{ubiquity_p-adic} we pick each $\rho_{i}$ to be of the form
\begin{equation*}
    \rho_{i}(h)=\rho(h)^{\ell_{i}}=\left(p^{-1}h^{-1}\right)^{\ell_{i}}
\end{equation*}
for
\begin{equation} \label{dirivector}
\boldsymbol{\ell}=(\ell_{1},\dots, \ell_{n}) \in (1,m-1)^{n} \quad \text{ with } \, \, \sum\limits_{i=1}^{n}\ell_{i}=n+m.
\end{equation}
Observing such a choice of exponent satisfies the requirements of Proposition~\ref{ubiquity}. As in the setup of Theorem~\ref{MTPRR}, set each 
\begin{equation*}
t_{i}=\tau_{i}-\ell_{i} \quad (1\leq i \leq n)\, .    
\end{equation*}
This choice of $\rho$ function includes the constant $"p^{-1}"$ appearing in Proposition~\ref{ubiquity_p-adic}, but note this can safely be removed to obtain our result by observing that for any choice of $\varepsilon>0$ there exists sufficiently large $h\in\R_{+}$ such that
\begin{equation*}
     h^{-\tau_{i}-\varepsilon} \leq \rho(h)^{\ell_{i}+t_{i}}=\left(p^{-1}h^{-1}\right)^{\tau_{i}}\leq h^{-\tau_{i}}\, \quad (1\leq i \leq n).
\end{equation*}
Consider the following cases:
\begin{enumerate}
    \item (Ball to rectangle) $\tau_{i}\geq \frac{n+m}{n}$ for all $1\leq i \leq n$. Then let each $\ell_{i}=\frac{n+m}{n}$. Note that such choice satisfies \eqref{dirivector}. The set $A$ from Theorem~\ref{MTPRR} takes the following order:
    \begin{equation*}
        \ell_{1}=\dots=\ell_{n} \leq \tau_{n} \leq \dots \leq \tau_{1}.
    \end{equation*}
    So for any $\ell_{i}$ we have
    \begin{equation*}
        \cK_{1}=\{1,\dots, n \} , \quad \cK_{2}=\emptyset , \quad \cK_{3}=\emptyset ,
    \end{equation*}
    which trivially leads to a full dimension lower bound in the case $A=\ell_{i}$.\par 
    For each $\tau_{j}$ we have
    \begin{equation*}
        \cK_{1}=\emptyset , \quad \cK_{2}=\{j, \dots, n\}, \quad \cK_{3}=\{1,\dots , j-1\}.
    \end{equation*}
    Hence
    \begin{align}
        \dimh W_{n,m}^{\Zp}(\Psi) & \geq \min_{1\leq j \leq n} \left\{ (n-j+1)m + (m-1)(j-1) + \frac{\frac{n+m}{n}(j-1)-\sum\limits_{j\leq i \leq n}(\tau_{i}-\frac{n+m}{n})}{\tau_{j}} \right\} \nonumber \\ 
        & \geq \min_{1\leq j \leq n} \left\{ n(m-1) + \frac{(n+m) -\sum\limits_{j\leq i \leq n}(\tau_{i}-\tau_{j})}{\tau_{j}} \right\}. \label{eq4}
    \end{align}
    \item (Rectangle to rectangle) $\tau_{j}<\frac{n+m}{n}$ for some $1\leq j \leq n$. The idea now is to form a "Dirichlet-exponent" rectangle that contains the $\tau$-rectangle. To do this, we are trying to find $1\leq u \leq n$ that solves
    \begin{equation*}
        u \times \widetilde{D} + \sum\limits_{u < i \leq n}\tau_{i}=n+m,
    \end{equation*}
    for some $\widetilde{D}>0$ with $\tau_{i}>\widetilde{D}$ for all $1\leq i \leq u$. That is, pick $u$ such that
    \begin{equation*}
        \tau_{1} \geq \dots \geq \tau_{u} > \frac{n+m -\sum\limits_{u\leq i \leq n} \tau_{i}}{u}:=\widetilde{D} \geq \tau_{u+1} \geq \dots \geq \tau_{n}.
    \end{equation*}
    Set 
    \begin{align*}
        \ell_{i}&=\widetilde{D} \quad (1\leq i \leq u), \\[2ex]
        \ell_{i}&=\tau_{i} \quad (u+1\leq i \leq n),
    \end{align*}
    and observe, by definition of $\widetilde{D}$ and the fact that each $\tau_{i}>1$, that $\boldsymbol{\ell}$ satisfies \eqref{dirivector}. For $\ell_{i}$ with $1\leq i \leq u$ we have that
    \begin{equation*}
        \cK_{1}=\{ 1, \dots , u \}, \quad \cK_{2}=\{u+1, \dots , n\}, \quad \cK_{3}=\emptyset,
    \end{equation*}
    and for $\ell_{i}$ with $u+1\leq i \leq n$
    \begin{equation*}
        \cK_{1}=\{1, \dots , i \}, \quad \cK_{2}= \{i+1, \dots , n\}, \quad \cK_{3}=\emptyset.
    \end{equation*}
    Thus, for $A=\ell_{i}$ we obtain the trivial full dimension as a lower bound. \par 
    For each $\tau_{j}$ with $1\leq j \leq u$ we have that
    \begin{equation*}
        \cK_{1}=\emptyset , \quad \cK_{2}=\{ j, \dots , n\}, \quad \cK_{3}=\{1, \dots , j-1\},
    \end{equation*}
    thus the same calculation of \eqref{eq4} follows to complete the proof.    
\end{enumerate}

\section{Complex approximation}
In this section, we obtain a Lebesgue measure dichotomy on certain sets of systems of complex linear forms. Under an additional assumption, we calculate the Hausdorff dimension of such sets when they are null.

For each $z\in \C$, let $\llbracket z \rrbracket$ be the distance from $z$ to its nearest Gaussian integer; that is
\[
\llbracket z \rrbracket
:=
\min\left\{ |z - p|: p\in \Z[i]\right\}.
\]
where $|\cdot|$ is the Euclidean norm in the complex plane.
Let $I_{\C}$ be the compact square
\[
I_{\C}
:=
\left\{ z\in \C: - \frac{1}{2}\leq \Re(z)\leq  \frac{1}{2} \;\text{ and }\;  -\frac{1}{2}\leq  \Im(z) \leq  \frac{1}{2} \right\}.
\]
Fix $m,n\in\N$ for the rest of this section. For any $n$-tuple of non-increasing positive functions $\varphi=(\varphi_1,\ldots, \varphi_n):\N\to\R_{+}^m$ with
\[
\lim_{q\to\infty} \varphi_j(q) = 0  \quad (1\leq j\leq n)\, ,
\]
and any $m$-tuple of non-decreasing positive functions $\Phi=(\Phi_1,\ldots, \Phi_m):\N \to \R_{+}^{m}$ such that
\[
\lim_{q\to\infty} \Phi_k(q)=\infty
\quad (1\leq k\leq m),
\]
let $W_{n,m}^{\C}(\varphi,\Phi)\subseteq I_{\C}^{m\times n}$ be the collection of $m\times n$ matrices $A$ with entries in $I_{\C}$ which verify the following property:
\begin{center}
for infinitely many integers $u\geq 1$  there is a non-zero $\q\in \Z[i]^{1\times m}$ such that
\begin{align*}
\left\llbracket \q \,A_j  \right\rrbracket &< \varphi_j(u) \quad (1\leq j\leq n), \\[2ex]
|q_k| &\leq \Phi_k(u) \quad (1\leq k\leq m). 
\end{align*}
\end{center}
%Following \cite{DodsonEveritt2014}, we assume that for every $j\in\{1,\ldots, n\}$, $k\in\{1,\ldots,m\}$, and $u\in\R_{\geq 1}$, we have
%\[\varphi_j(u)=\varphi_j(\lfloor u \rfloor) \;\text{ and }\; \Phi_k(u)=\Phi_k(\lfloor u \rfloor), \]
%where $\lfloor u \rfloor$ is the floor function of $u$. 
For $k\in\{1,\ldots, m\}$ and $u\in \N$, write
\[
F_k(u)
:=
\min\{ v\in\N: \Phi_k(v)\geq u\}.
\]
The set $W_{n,m}^{\C}(\varphi,\Phi)$ is then %consists of precisely those matrices $A\in I_{\C}^{m\times n}$ for which there are infinitely many $\q\in \Z[i]^{1\times m}$ verifying
%\[\forall j\in \{1,\ldots, n\} \quad \left\llbracket  \q \,A_j  \right\rrbracket < \varphi_j\left( \max\left\{ \Phi_1^{-1}(|q_1|), \ldots, \Phi_m^{-1}(|q_m|) \right\}\right). \] That is,
\begin{equation*}
    W_{n,m}^{\C}(\varphi,\Phi) =\left\{ A \in I_{\C}^{m\times n } : \begin{array}{l} \left\llbracket  \q \,A_j  \right\rrbracket 
<
\varphi_j\left( \max\left\{ F_1(|q_1|), \ldots, F_m(|q_m|) \right\}\right)\, \, \, (1\leq j \leq n) \\[2ex]
\text{ for infinitely many } \q \in \Z[i]^{1\times m} \end{array} \right\}. 
\end{equation*}
We denote the Lebesgue measure on $\C^{m\times n}$ by $\mu_{m\times n}^{\C}$.
\begin{theorem}\label{TEO:MEASURE}
If there are some constants $N_0,M\in \N$ with $M\geq 2$, and $c_1, c_2>1$ such that for every $j\in\N_{\geq N_0}$ we have
\begin{equation}\label{EQ:LinearForms:C}
c_1\Phi_k(M^j) \leq \Phi_k(M^{j+1}) \leq c_2\Phi_k(M^j)
\quad (k=1,\ldots,m),
\end{equation}
then
\[
\mu_{m\times n}^{\C}\left(W_{n,m}^{\C}(\varphi, \Phi)\right)
=
\begin{cases}
0 &\text{\rm if }\quad  \displaystyle\sum\limits_{q=1}^{\infty} \frac{1}{q}\left( \displaystyle\prod_{j=1}^n \varphi_j(q) \displaystyle\prod_{k=1}^m \Phi_k(q)\right)^2   < \infty,  \\[2ex]
1 &\text{\rm if }\quad  \displaystyle\sum\limits_{q=1}^{\infty} \frac{1}{q} \left( \displaystyle\prod_{j=1}^n \varphi_j(q) \displaystyle\prod_{k=1}^m \Phi_k(q)\right)^2   = \infty.
\end{cases}
\]
\end{theorem}
The problem for complex numbers is less studied than the previous contexts.
Most of the work has been done in the case $m=n=1$.
\begin{itemize}
    \item $n=m=1$. LeVeque used continued fractions to get Theorem \ref{TEO:MEASURE} \cite{Leveque}. 
    \item $n=m=1$. Sullivan proved Khintchine's theorem for real numbers and imaginary quadratic fields \cite{Sullivan}.
    %Sullivan's result holds in the context of approximation of complex numbers by quotients of elements in the ring of integers of a given imaginary quadratic fields $\mathbb R(i\sqrt d)$, where $d$ is a square-free natural number. Sullivan used the theory of Bianchi groups. The case $d = 1$ corresponds to the Picard group and approximation by Gaussian rationals. 
    \item $n=m=1$. Beresnevich, Dickinson, and Velani derived Sullivan's result using ubiquitous systems \cite[Theorem 7]{BDV06}.
    \item  $m,n\in\N$. Hussain, non-weighted small linear forms \cite{HussainNZJM}.
\end{itemize}
In $\mathbb{Q}(i)$, the problems can be attacked using Hurwitz continued fractions. To this end, the best-known results are in \cite{BGH, GRGRH}.

For any vector $\boldsymbol{\tau} = (\tau_1, \ldots, \tau_n)\in \R^n$ such that
\begin{equation}\label{EQ:CondOnTau}
\min_{1\leq j\leq n} \tau_j > 1
\;\text{ and }\;
\sum\limits_{j=1}^n \tau_j \geq m+n,
\end{equation}
define the numbers
\[
s_j(\boldsymbol{\tau})
:= 2n(m-1) 
+ 
2\, \frac{m + n - \sum\limits_{r\,:\, \tau_{r}<\tau_{j}} ( \tau_r-\tau_j)}{\tau_j}
\quad (1\leq j\leq n)
\]
and the set
%\[
%W(\boldsymbol{\tau})
%:=
%\left\{ A\in I_{\C}^{m\times n}: \left\llbracket \mathbf{q}  A_j \right\rrbracket_{\infty} < \frac{\|\mathbf{q}\|_{\infty}^{1}}{\|\mathbf{q}\|^{\tau_j}} \text{ for all } j\in\{1,\ldots ,n\} \text{ for infinitely many } \mathbf{q}\in \Z[i]^m\right\}
%\]
\[
W^{\C}_{n,m}(\boldsymbol{\tau})
:=
\left\{ A\in I_{\C}^{m\times n}:\begin{array}{l} \left\llbracket \mathbf{q}  A_j \right\rrbracket  < \frac{1}{\|\mathbf{q}\|^{\tau_j-1}} \; (1\leq j \leq n)\;\\[2ex] \text{ for infinitely many } \mathbf{q}\in \Z[i]^m \end{array}\right\}.
\]

\begin{theorem}\label{TEO:HDIM}
For any vector $\boldsymbol{\tau}$ satisfying \eqref{EQ:CondOnTau}, we have 
\[
\dimh W^{\C}_{n,m}(\boldsymbol{\tau}) = \min\{ s_1(\boldsymbol{\tau}), \ldots, s_n(\boldsymbol{\tau})\}.
\]   
\end{theorem}

%DO NOT ERASE
%\begin{theorem}
%Let $\bfa=(a_1,\ldots, a_n)\in (1,+\infty)^n$ be such that $a_1 + \cdots + a_n = n + m$. Let $\varphi=(\varphi_1,\ldots, \varphi_n):\N\to\R_{+}$ be such that each $u\mapsto \varphi_j(u)/u^{a_j}$ is decreasing. If $\mathcal{U}$ is bounded, then
%\[
%\dimh W_{n,m}^{\C}(\psi)
%=
%\sup\{ s(\tau): \tau\in \mathcal{U}\}.
%\]
%Moreover, $\cH^{s}\left( W_{n,m}^{\C}(\psi)\right)= \infty$.
%\end{theorem}
The only previously known Hausdorff dimension results for $\Psi$ approximable complex numbers was for $m=n=1$ in \cite{DodsonKristensen} by Dodson and Kristensen.
They proved $\dimh W^{\C}_{1, 1}( \tau )=\frac4{\tau+1}$.
Later, He and Xiong \cite{HeXiong2021} extended the result to an arbitrary approximation function.

We develop the proofs in a slightly different context. For each complex number $z$, define
\[
|z|_{\infty}:= \max\{|\Re(z)|, |\Im(z)|\}
\;\text{ and }\;
\llbracket z \rrbracket_{\infty} := \min\{|z - p|_{\infty}: p\in \Z[i]\}.
\]
First, we show Theorem \ref{TEO:MEASURE} and Theorem \ref{TEO:HDIM} when $|\,\cdot\,|$ and $\llbracket \,\cdot\,\rrbracket$ are replaced by $ |\,\cdot\, |_{\infty}$ and $\llbracket \,\cdot\,\rrbracket_{\infty}$, respectively. Afterwards, we use the equivalence between $|\,\cdot\, |_{\infty}$ and $|\,\cdot\,|$ to conclude the results in the original setting.

\subsection{Ubiquity for complex numbers}
We need the following complex version of Minkowski's theorem for linear forms (cfr. \cite[Theorem 95]{Hecke1981}).
\begin{lemma}\label{Teo:Mink:Cx:LF}
Let $\gamma_1,\ldots, \gamma_n,\theta_1,\ldots, \theta_m$ be positive numbers satisfying $\prod_{k=1}^n \gamma_k \prod_{j=1}^m \theta_j\geq 1$. 
For every matrix $A\in \C^{m\times n}$ with columns $A_1, \ldots, A_n$ there exists a vector $(\q,\p)\in \Z[i]^{1\times (m + n)}$, $\q\neq 0$, such that
\begin{align}\label{Eq:SysMink}
\left|\, \q \; A_k - p_k\right|_{\infty} &< \gamma_k  \quad (1\leq k\leq n),\nonumber \\
\left| q_j \right|_{\infty} &\leq  \theta_j \quad (1\leq j\leq m).
\end{align}
\end{lemma}
\begin{proof}
Take any matrix $A=(a_{j,k})_{1\leq j\leq m, 1\leq k\leq n} \in \C^{m\times n}$. For $j\in\{1,\ldots, m\}$ and $k\in \{1,\ldots, n\}$, write
\[
a_{j,k}^1 := \Re(a_{j,k})
\;\text{ and }\;
a_{j,k}^2 := \Im(a_{j,k}).
\]
For each $j\in\{1,\ldots, n\}$, define the column vectors $B_{2j-1}$, $B_{2j}$ by
\[
B_{2j-1}
:= 
\begin{pmatrix}
    a_{1,j}^1 \\
    -a_{1,j}^2 \\
    \vdots \\
    a_{m,j}^1 \\
    -a_{m,j}^2
\end{pmatrix}
\;\text{ and }\;
B_{2j}
:=
\begin{pmatrix}
   a_{1,j}^2 \\ 
   a_{1,j}^1 \\
   \vdots \\
   a_{m,j}^2 \\
   a_{m,j}^1.
\end{pmatrix}.
\]
Call $B\in \Z^{2m\times 2n}$ the matrix whose $r$-th column is $B_r$. 
Then, a vector $(\q,\p)\in \Z[i]^{1\times (m + n)}$ solves \eqref{Eq:SysMink} if and only if the vector $( \mathbf{Q} , \mathbf{P})=(Q_1,\ldots, Q_{2m}, P_1, \ldots, P_{2n})\in \Z^{1\times (2m+2n)}$ given by
\[
\mathbf{Q} =(\Re(q_1), \Im(q_1), \ldots, \Re(q_m), \Im(q_m))
\;\text{ and }\;
\mathbf{P} =(\Re(p_1), \Im(p_1), \ldots, \Re(p_n), \Im(p_n))
\]
solves the system
\begin{align}\label{Eq:SysMink02}
\left|\, \mathbf{Q} \, B_{2k-1} - P_{2k-1}\right| &< \gamma_{k} \quad (1\leq k\leq n),  \nonumber\\
\left|\, \mathbf{Q} \, B_{2k} - P_{2k}\right|  &< \gamma_{k} \quad (1\leq k\leq n),  \nonumber\\
\left| Q_{2j-1} \right| &\leq  \theta_{j}  \quad (1\leq j\leq m), \nonumber\\
\left| Q_{2j} \right|  &\leq  \theta_{j}  \quad (1\leq j\leq m).
\end{align}
Call $I_{2n}$ (resp. $I_{2m}$) the identity matrix of size $2n\times 2n$ (resp. $2m\times 2m$) and let $O_{2m\times 2n}$ be the matrix of size $2m\times 2n$ whose entries are all equal to $0$. Since 
\[
\prod_{k=1}^n \gamma_k \prod_{j=1}^m \theta_j\geq 1
\]
and the determinant of the matrix
\[
\begin{pmatrix}
B & I_{2m} \\
-I_{2n} & O_{2m\times 2n}
\end{pmatrix}
\]
is $1$, the system \eqref{Eq:SysMink} has a non-trivial solution by Minkowski's Theorem.
\end{proof}

%\subsection{Proof of Theorem \ref{TEO:MEASURE}}
Given any $k\in \N$ and $\mathbf{z}=(z_1,\ldots, z_k)\in \C^k$, irrespective of interpreting it as a row or a column, define
\[
\| \mathbf{z}\|_{\infty} 
:=
\max\left\{ |z_1|_{\infty}, \ldots, |z_k|_{\infty}\right\}.
\]
This way, if $\|\mathbf{z}\|_2=\sqrt{|z_1|^2 + \cdots + |z_k|^2}$ is the Euclidean norm on $\C^k$, we have
\begin{equation}\label{Eq:EquivNorm}
\frac{1}{\sqrt{2k}} \| \mathbf{z}\|_{2}
\leq 
\| \mathbf{z}\|_{\infty}
\leq
\| \mathbf{z}\|_{2}.
\end{equation}
%Certainly,
%\[
%\frac{1}{\sqrt{2k}} \|\mathbf{z}\|_2
%\leq 
%\frac{1}{\sqrt{2}}  \, \max_{1\leq j \leq k} |z_j|
%\leq 
%\max_{1\leq j \leq k} |z_j|_{\infty}
%=
%\|\mathbf{z}\|_{\infty}
%\leq
%\max_{1\leq j\leq k} |z_j| 
%\leq 
%\|\mathbf{z}\|_2.
%\]
%\eqref{EQ:LinearForms:C}
Let $M \in \N$ be such that $M>2^{5n}3^m m^{n+1}$. 
We work with the following objects:
\begin{enumerate}[i.]
\item $J:=\{\alpha=(\q,\p)\in \Z[i]^{m+n}: \|\p\|_{\infty} \leq 2m \|\q\|_{\infty}\}$,
\item $\beta:J\to\R_{+}$, $\alpha=(\q,\p) \mapsto \beta_{\alpha}=\max\left\{ F_1\left( |q_1|_{\infty}\right), \ldots, F_m\left( |q_m|_{\infty} \right)\right\}$,
\item $(u_j)_{j\geq 1}$ given by $u_j=M^j$ for all $j\in \N$,
\item For each $\alpha = (\q,\p) \in J$, define $R_{\alpha,j} :=\left\{ A_j \in I_{\C}^{m\times 1}: \q \,A_j = p_j\right\}$ for $j\in\{ 1,\ldots, n\}$ and the resonant set $R_{\alpha}$ is $R_{\alpha}:=\prod_{j=1}^n R_{\alpha,j}$,
\item In this setting, we have $\kappa_j=1  - \frac{1}{m} = \frac{2m-2}{2m}$ and $\delta_j=2m$ for $j\in\{1,\ldots,n\} $.
\end{enumerate}
Let us expand on the $\kappa$-scaling property. If $V$ is a finite-dimensional complex vector space and if $W$ is a subspace of $V$, then both $V$ and $W$ are real vector spaces, and their dimension as real vector spaces is twice their dimension as complex vector spaces. 
The value of $\kappa_i$ is then obtained from \cite[Section 2.2]{AB19}.

Define the $n$-tuples of positive functions $\rho:(\rho_1,\ldots, \rho_n):\N\to \R_{+}^n$ by
\[
\rho_j(u)
:=
\sqrt{2}\,M\frac{\varphi_j(u)}{\|\Phi(u)\|_{\infty}} \left( \prod_{s=1}^n \varphi_s(u)\prod_{k=1}^m \Phi_k(u)\right)^{-1/n}
\quad (1\leq j\leq n, \, u\in \N)
\]
and $\Psi=(\psi_1,\ldots, \psi_n): \N\to \R_{+}^n$ by
\[
\psi_j(u)
:=
\frac{\varphi_j(u )}{2m \|\Phi(u )\|_{\infty} }
\quad (1\leq j\leq n, \, u\in \N).
\]
Note that
\begin{equation}\label{Eq:Quotpsirho}
\prod_{j=1}^n \frac{\psi_j(u)}{\rho_j(u)}
=
\left(2^{3/2} m M\right)^{-n}  \prod_{j=1}^n \varphi_j(u) \prod_{k=1}^m \Phi_k(u)
\quad(u\in \N).
\end{equation}

When there is a strictly increasing sequence of natural numbers $(r_s)_{s\geq 1}$ satisfying
\[
\prod_{j=1}^n \varphi_s(r_s)\prod_{k=1}^m \Phi_k(r_s)> 1
\quad(s\in \N),
\]
then, by Lemma \ref{Teo:Mink:Cx:LF}, we have $\mu_{m\times n}^{\C}\left(W_{n,m}^{\C}(\varphi,\Phi)\right)=\mu_{m\times n}^{\C}\left(I_{\C}^{m\times n}\right)$. Thus, we assume 
\begin{equation}\label{EQ:LF:C:ExtAss}
\prod_{j=1}^n \varphi_j(u)\prod_{k=1}^m \Phi_k(u) \leq 1
\quad(u\in \N).
\end{equation}
%The rest of this subsection is devoted to proving Lemma \ref{Lem:UbiqSyst:CxLF} below. Although the proof amounts to checking the definition of a ubiquitous system, its length persuades us to divide it into several steps.
\begin{lemma}\label{Lem:UbiqSyst:CxLF}
The system $( ( R_{\alpha} )_{\alpha\in J}, \beta)$ is a weighted ubiquitous system with respect to the function $\rho$.
\end{lemma}
The proof of Lemma \ref{Lem:UbiqSyst:CxLF} will require three preliminary propositions and some additional discussion.
In what follows, given a vector $\mathbf{z}\in \C^{m\times 1}$ (resp. $\mathbf{z}\in \C^{1\times m }$) , we denote by $\overline{\mathbf{z}}$ the vector in $\C^{m\times 1}$ (resp. in $\C^{1\times m }$) whose $r$-th coordinate is the complex conjugate of the $r$-th coordinate of $\mathbf{z}$.
\begin{proposition}\label{Prop:Cx:Measure:Aux}
Let $f_j:\N\to\R_{+}$, $1\leq j\leq n$, be positive functions such that $f_j(u)\to 0$ as $u\to \infty$ and
\[
\prod_{j=1}^n f_j(u)\prod_{k=1}^m \Phi_k(u) = 1
\quad(u\in\N).
\]
For all large $u\in \N$ and all $A$ in the interior of $I_{\C}^{m\times n}$, there is some $\alpha=(\q,\p) \in J$ satisfying $\beta_{\alpha} \leq u$ and 
\begin{equation}\label{Eq:Propo:LemUbiqSys}
A\in \prod_{j=1}^n \Delta\left( R_{\alpha,j}; \sqrt{2} \, \frac{f_j(u)}{\|\q\|_{\infty} } \right).
\end{equation}
\end{proposition}
\begin{proof}
Take $A$ in the interior of $I_{\C}^{m\times n}$. 
By Lemma \ref{Teo:Mink:Cx:LF}, there is some non-zero $(\q,\p)\in \Z[i]^{1\times(m+n)}$ such that 
\begin{align*}
\left| \q\, A_j - p_j\right|_{\infty} &< f_j(u) \quad (1\leq j\leq n), \\[2ex]
\left|q_k\right|_{\infty} &\leq \Phi_k(u) \quad (1\leq k\leq m).
\end{align*}
The vector $\alpha=(\q,\p)$ satisfies $\beta_{\alpha}\leq u$ and
\begin{align*}
\left| \q\,  A_j - p_j\right| &< \sqrt{2} \, f_j(u)  \,\quad\, (1\leq j\leq n), \\[2ex]
\left|q_k\right| &\leq \sqrt{2}\, \Phi_k(u)\quad (1\leq k\leq m).
\end{align*}
Let us verify \eqref{Eq:Propo:LemUbiqSys}. 
Let $\mathbf{q}^{\top}$ be the transpose of $\mathbf{q}$. 
Define the vector
\[
\bfa
=
\left(
\frac{\overline{p_j}}{\|q_j\|_2^2} - \frac{\overline{\q}\,\overline{A_j}}{\|\q\|_2^2}
\right)\q^{\top}
+
\overline{A_j}.
\]
Direct computations yield 
\[
\bfa \in R_{\alpha,j}
\quad\text{ and }\quad
\|A_j - \bfa\| = \frac{|\q A_j - p|}{\|\q\|^2}.
\]
Pick any $\bfb \in R_{\alpha,j}$. 
The vectors 
\[
\bfw_{\bfb} = \overline{\bfb} - \frac{\overline{p_j}}{\| \q \|_2^2} \q^{\top}
\;\text{ and }\;
\mathbf{v} = \overline{A_j} - \frac{ \overline{\q} \, \overline{A_j} }{\|\q\|^2} \q^{\top}
\]
are orthogonal to $\q^{\top}$ with respect to the usual inner product,
so
\begin{align*}
\left\| A_j - \bfb \right\|_2^2
&=
\left\| \overline{A_j} - \overline{\bfb} \right\|_2^2 \\
&=
\left\| 
 \frac{ \overline{\q} \, \overline{A_j} }{\|\q\|_2^2} \q^{\top}
 -
 \frac{\overline{p_j}}{\| \q \|_2^2} \q^{\top}
 \right\|_2^2 
 +
 \|\mathbf{v} - \bfw_{\bfb}\|_2^2 \\
 &\geq 
 \frac{|\q A_j - p|^2}{\| \q\|_2^2} \\
 &=
 \|A_j - \bfa\|_2^2.
\end{align*}
Therefore, 
\[
\min\left\{ \|A_j - \mathbf{a}\|_{\infty}:\mathbf{a}\in R_{\alpha,j} \right\}
\leq 
\min\left\{ \|A_j - \mathbf{a}\|_2:\mathbf{a}\in R_{\alpha,j} \right\}
=
\frac{ |\q  \, A_j -  p_j | }{\|\q\|_2}
\leq 
\frac{ \sqrt{2}\, f_j(u) }{\|\q\|_{\infty}}.
\]
\end{proof}
Define the functions $f_j:\N \to \R$, $j\in\{1,\ldots,n\}$, by
\[
f_j(u)
:=
\varphi_j(u) \left( \prod_{s=1}^n \varphi_s(u)  \prod_{k=1}^m \Phi_k(u) \right)^{-1/n}
\quad\quad(u\in \N).
\]
%Let $(u_j)_{j\geq 1}$ be the sequence of real numbers given by
%\[\forall j\in\N \quad u_j := M^{j}.\]
For each $s\in \N$, define the set
\[
\widetilde{J}_s
:=
\left\{ \alpha=(\q,\p)\in J : \frac{\Phi_k(u_s)}{M}\leq |q_k|_{\infty} \leq \Phi_k(u_s) \quad (1\leq k\leq m) \right\}.
\]
Observe that $\beta_{\alpha}\leq u_s$ when $\alpha=(\q,\p)\in \widetilde{J}_{s}$. 
Then, since $\beta_{\alpha}\to \infty$ as $\alpha\to\infty$, we may choose an adequate $l_s$ ensuring $\widetilde{J}_s\subseteq J_s$. For $k\in \{1,\ldots, m\}$, write 
\[
J_{s,k}
:=
\left\{ \alpha\in J: \left| q_k\right|_{\infty} \leq \frac{\Phi_k(u_s)}{M} \text{ and } |q_j|_{\infty}\leq \Phi_j(u_s) \quad( j\in \{1,\ldots m\}\setminus \{k\}) \right\}.
\]
Let $B=\prod_{k=1}^n B(X_k;r)$ be an arbitrary ball whose closure is contained in the interior of $I_{\C}^{m\times n}$. In view of Proposition \ref{Prop:Cx:Measure:Aux}, for large $s\in\N$ we have 
\begin{align*}
B 
&= 
B\cap \bigcup_{ \substack{\alpha\in J\\ \beta_{\alpha}\leq u_s} } \prod_{j=1}^n\Delta\left( R_{\alpha,j}, \sqrt{2}\, \frac{f_j(u_s)}{\beta_{\alpha} } \right) \nonumber\\
&=
\left(B\cap \bigcup_{\alpha \in \widetilde{J}_s} \prod_{j=1}^n \Delta\left( R_{\alpha,j}, \sqrt{2}\, \frac{f_j(u_s)}{\beta_{\alpha}  } \right)\right)
\,\cup\,
\left(B\cap \bigcup_{h=1}^m \bigcup_{\alpha \in J_{s,h}} \prod_{j=1}^n \Delta\left( R_{\alpha,j}, \sqrt{2}\, \frac{f_j(u_s)}{\beta_{\alpha} } \right)\right).
\end{align*}
\begin{proposition}\label{Prop:EstMeasBall}
There is some $N(r)\in\N$ such that for all $j\in\{1,\ldots, n\}$ and $\q \in \Z[i]^{1\times m}$, every $s\in\N_{\geq N(r)}$ satisfies
\[
\#\left\{ p \in \Z[i]:
B(X_j,r)
\cap 
\Delta\left( R_{\alpha,j}, \sqrt{2}\, \frac{f_j(u_s)}{\|\q\|_{\infty}} \right) 
\neq 
\varnothing
\right\}
\leq 
(8rm \|\q\|_{\infty} + 2)^{2}.
\]
\end{proposition}
\begin{proof}
Let $j\in\{1,\ldots, n\}$ and $0\neq \q\in\Z[i]^m$ be arbitrary. Take $p\in\Z[i]$ be such that for some $A_j, Y_j \in I_{\C}^{m\times 1}$ we have
\[
\|X_j - Y_j\|_{\infty}< r, 
\quad
\q A_j = p, 
\;\text{ and } \;
\|A_j - Y_j\|_{\infty}< \sqrt{2}\, \frac{f_j(u_s)}{\|\q\|_{\infty} }.
\]
That is, the $\frac{f_j(u_s)}{\|\q\|_{\infty} }$-thickened hyperplane $\q A_j = p$ has non-empty intersection with $B(X_j,r)$. By Cauchy's inequality and \eqref{Eq:EquivNorm}, we have
\[
| \q  X_j - p|_{\infty}
\leq 
| \q  (X_j - A_j)|
\leq 
\|\q\|_2 \|X_j - A_j\|_2
\leq 
2m \|\q\|_{\infty} \left(\|X_j - Y_j\|_{\infty} + \|Y_j - A_j \|_{\infty}\right)
\]
and hence
\[
| \q  X_j - p|_{\infty}
<
2m \|\q\|_{\infty} \left( r + \sqrt{2} \, \frac{f_j(u_s)}{\|\q\|_{\infty}} \right).
\]
As a consequence, for every large $s\in \N$ (depending on $r$) we have
\[
| \q  X_j - p|_{\infty}
\leq 
4mr\|\q\|_{\infty}.
\]
The proposition now follows from the next elementary estimate: 
\begin{equation}\label{Eq:ElemEst}
\#\{z\in \Z[i] : |z|_{\infty} \leq R\} \leq (2R+2)^2
\quad (R>0).
\end{equation}
\end{proof}

\begin{proposition}\label{complexprop4}
If $M>2^{5n}3^m m^{n+1}$ and writing $\alpha=(\q,\p)$, every large $s\in\N$ satisfies
\begin{equation}\label{Eq:LemUbiqPrel}
\mu_{m\times n}^{\C} \left( B\cap \bigcup_{k=1}^m \bigcup_{\alpha \in J_{s,k}} \prod_{j=1}^n\Delta\left( R_{\alpha,j}, \sqrt{2}\, \frac{f_j(u_s)}{\|\q\|_{\infty}} \right)\right)
\leq 
\frac{1}{2}\mu_{m\times n}^{\C}(B).
\end{equation}
\end{proposition}
\begin{proof}
Take $s\in\N$. For each $k\in\{1,\ldots, m\}$, write 
%\[ J_{s,k}' :=
%\left\{ \q \in \Z[i]^m: \left| q_k\right|_{\infty} \leq \frac{\Phi_k(u_s)}{M} \,\text{ and }\, |q_j|_{\infty}\leq \Phi_j(u_s) 
%\quad ( j\in \{1,\ldots m\}\setminus \{k\}) \right\},
%\]
\[
J_{s,k}':= \left\{ \q \in \Z[i]^{1 \times m} : \alpha=(\q,\p) \in J_{s,k} \right\}\, ,
\]
so, by \eqref{Eq:ElemEst}, every large $s\in \N$ verifies
\[
\# J_{n,k}'
\leq
3^{2m}\left(\frac{1}{M}\prod_{j=1}^m \Phi_j(u_n)\right)^2.
\]
Let $G_{\C}$ be the intersection in \eqref{Eq:LemUbiqPrel}. By the $\kappa$-scaling property and Proposition \ref{Prop:EstMeasBall}, 
\[
\mu_{m\times n}^{\C}(G_{\C})
\leq 
\sum\limits_{k=1}^m \sum\limits_{\q\in J_{s,k}'} 
(8m r\|\q\|_{\infty} + 2)^{2n} \left(\prod_{j=1}^n \sqrt{2}\,\frac{f_j(u_s)r^{m-1}}{\|\q\|_{\infty}} \right)^2.
\]
Recall that $(a+b)^{2n}\leq (2a)^{2n}+ (2b)^{2n}$ for all $a,b\geq 0$, then 
\begin{equation}\label{Eq:Cx:Sums}
\mu_{m\times n}^{\C}(G_{\C})
\leq 
2^{9n}m^{2n}r^{2mn}
\sum\limits_{k=1}^m \sum\limits_{\q\in J_{s,k}'} 
\prod_{j=1}^n f_j^2(u_s)  
+ 
2^{5n}r^{2n(m-1)}
\sum\limits_{k=1}^m \sum\limits_{\q\in J_{s,k}'} 
\prod_{j=1}^n \frac{f_j^2(u_s) }{\|\q\|_{\infty}^2}.
\end{equation}
We can bound the the first sum in \eqref{Eq:Cx:Sums} using $M>2^{5n}3^m m^{n+1}$ as follows:
\begin{align*}
2^{9n}m^{2n}r^{2mn}
\sum\limits_{k=1}^m \sum\limits_{\q\in J_{s,j}'} 
\prod_{k=1}^n f_k^2(u_s)  
&=
\frac{2^{9n}m^{2n}r^{2mn}}{ \Phi_1^2(u_s)\cdots \Phi_m^2(u_s) }
\sum\limits_{k=1}^m \#J_{s,j}' \\
&\leq
\frac{2^{9n}3^{2m}m^{2n+1}}{M^2} r^{2nm} \\
&<  \frac{1}{4} \mu_{m\times n}^{\C}(B).   
\end{align*}

The second sum in \eqref{Eq:Cx:Sums} tends to $0$ as $s$ tends to $\infty$. Indeed, since $\#\{q\in \Z[i]: |q|_{\infty}=s\}=8s$, we have
\begin{align*}
\sum\limits_{k=1}^m \sum\limits_{\q\in J_{s,k}'} 
\prod_{j=1}^n \frac{f_j^2(u_s) }{\|\q\|^2}
&=
\frac{1}{\Phi_1^2(u_s)\cdots \Phi_m^2(u_s)}
\sum\limits_{k=1}^m \sum\limits_{\q\in J_{s,k}'} 
\frac{1}{\|\q\|^{2n}}\nonumber\\
&\leq 
\frac{1}{\Phi_1^2(u_s)\cdots \Phi_m^2(u_s)}
\sum\limits_{k=1}^m \sum\limits_{\q\in J_{s,k}'} 
\frac{1}{\|\q\|^{2}}\nonumber\\
&\leq 
\frac{1}{\Phi_1^2(u_s)\cdots \Phi_m^2(u_s)}
\sum\limits_{k=1}^m \sum\limits_{q_k=1}^{\Phi_k(u_s)/M} 
\frac{ 8q_k }{q_k^2} \left(\frac{\Phi_1(u_s)\cdots \Phi_m(u_s)}{\Phi_k(u_s)}\right)^2 \nonumber\\
&\leq
\frac{8}{\Phi_1^2(u_s)\cdots \Phi_m^2(u_s)}
\sum\limits_{k=1}^m \left(\frac{\Phi_1(u_s)\cdots \Phi_m(u_s)}{\Phi_k(u_s)}\right)^2
\left( \log \Phi_k(u_s) - \log M \right) \nonumber\\
&\leq
8
\sum\limits_{k=1}^m \frac{\log \Phi_k(u_s)}{\Phi_k(u_s)^2} \to 0 \text{ as } s\to\infty.
\end{align*}
\end{proof}
\begin{proof}[Proof of Lemma \ref{Lem:UbiqSyst:CxLF}]
It only remains to show the ubiquity condition with respect to $\rho$. For any $s\in \N$, the definition of $\widetilde{J}_s$ tells us that
\[
\|\q\|_{\infty}
\geq 
\frac{1}{M} \left\| \left(\Phi_1(u_s), \ldots, \Phi_m(u_s)\right) \right\|_{\infty}
\quad
\left(\alpha=(\q,\p)\in \widetilde{J}_s\right),
\]
then
\begin{align*}
\mu_{m\times n}^{\C} \left( B \cap \bigcup_{\alpha\in J_s} \Delta\left( R_{\alpha},\rho\right) \right) 
&=
\mu_{m\times n}^{\C} \left( B\cap  \bigcup_{\alpha \in J_{s}} \prod_{k=1}^m \Delta\left( R_{\alpha,k}, \sqrt{2} M\, \frac{f_k(u_s)}{\|\Phi(u_s)\|_{\infty}} \right)\right)\\
&\geq 
\mu_{m\times n}^{\C}\left( B\cap  \bigcup_{\alpha \in \widetilde{J}_{s}} \prod_{k=1}^m \Delta\left( R_{\alpha,k}, \sqrt{2} M\, \frac{f_k(u_s)}{\|\Phi(u_s)\|_{\infty}} \right)\right) \\
&\geq 
\mu_{m\times n}^{\C}\left( B\cap  \bigcup_{\alpha \in \widetilde{J}_{s}} \prod_{k=1}^m \Delta\left( R_{\alpha,k}, \sqrt{2} \, \frac{f_k(u_s)}{\|\q \|_{\infty}} \right)\right) \\
&\geq 
\frac{1}{2}\mu_{m\times n}^{\C}(B).
\end{align*}
\end{proof}
Approximating by finitely many balls whose closure is contained in the interior of $I_{\C}^{m\times n}$, we conclude that previous estimates also holds for any ball $B$ in $I_{\C}^{m\times n}$.

\subsection{Proof of Theorem~\ref{TEO:MEASURE}}

Firstly we introduce Lemma \ref{LF:C:AuxLemma} below, which allows us to impose an additional assumption on $\psi$ and $\Phi$ without losing any generality. We omit its proof, for it is shown as Lemma 6.1 in \cite{KW23}. The only difference is that in the last step we deal with the series $\sum_t (c_1^{-2(d-\varepsilon)})^t <\infty$. Call $W_{\infty}(\varphi,\Phi)$ the set of matrices $A\in I_{\C}^{m\times n}$ such that
\begin{align*}
    \left\llbracket \q A_j \right\rrbracket_{\infty} &< \phi_j(u) \quad (1\leq j \leq n),\\
    |q_k|_{\infty} &\leq \Phi_k(u) \quad (1\leq k \leq m).
\end{align*}
%satisfying the conditions defining $W(\varphi,\Phi)$ with $|\,\cdot\,|_{\infty}$ and $\llbracket \,\cdot\,\rrbracket_{\infty}$ in place of $|\,\cdot\,|$ and $\llbracket \,\cdot\,\rrbracket$, respectively.
\begin{lemma}\label{LF:C:AuxLemma}
Under the assumptions of Theorem \ref{TEO:MEASURE}, there are functions $\widetilde{\varphi}=(\widetilde{\varphi}_1,\ldots, \widetilde{\varphi}_n)$ such that $W_{n,m}^{\C}(\varphi,\Phi)\subseteq W_{n,m}^{\C}(\widetilde{\varphi},\Phi)$, $\mu_{m\times n}^{\C}\left(W_{n,m}^{\C}(\varphi,\Phi)\right)=\mu_{m\times n}^{\C}\left( W_{n,m}^{\C}(\tilde{\varphi},\Phi)\right)$, and 
\[
\lim_{u\to \infty} \|\Phi(u)\|_{\infty}^n \prod_{j=1}^m \Phi_k(u) \prod_{j=1}^n \varphi_j(u)= \infty
\quad (1\leq j \leq n).
\]
\end{lemma}
\begin{lemma}\label{LF:C:AuxRep}
Theorem \ref{TEO:MEASURE} holds if we replace $|\,\cdot\,|$ and $\llbracket \,\cdot\,\rrbracket $ with $|\,\cdot\,|_{\infty}$ and $\llbracket \,\cdot\,\rrbracket_{\infty}$, respectively.
\end{lemma}
\begin{proof}
First, we verify the hypotheses of Theorem \ref{KW ambient measure}.
\begin{enumerate}[1.]
\item The system $(( R_{\alpha})_{\alpha\in J}, \beta)$ is ubiquitous with respect to $\rho$, $(l_j)_{j\geq 1}$, and $(u_j)_{j\geq 1}$ by the discussion above.

\item The function $\Psi=(\psi_1,\ldots, \psi_n)$ is $c$-regular. Indeed, by \eqref{EQ:LinearForms:C} and the monotonicity of each $\varphi_j$ and each $\Phi_k$, we have 
\[
\psi_j(u_{s+1})
=
\frac{\varphi_j(u_{s+1})}{2m\|\Phi(u_{s+1})\|_{\infty}} 
\leq
\frac{\varphi_j(u_{s})}{c_12m\|\Phi(u_{s})\|_{\infty}} 
=
\frac{1}{c_1} \psi_j(u_{s})
\quad
(s\in \N).
\]
\item By \eqref{EQ:LF:C:ExtAss}, for $j\in\{1,\ldots, n\}$ we have $\rho_j(u)\geq \psi_j(u)$ for all $u\in\N$ and Lemma \ref{LF:C:AuxLemma} implies
\[
\lim_{u\to\infty} \rho_j(u) = 0.
\]
\end{enumerate}
Combining \eqref{Eq:Quotpsirho}, $\delta_j(1-\kappa_j)=2$ for all $j\in\{1,\ldots,n\}$, and Theorem \ref{KW ambient measure}, the set $W_{\infty}(\varphi,\Phi)$ is either of full measure or null depending on the divergence or convergence of the series
\begin{equation}\label{Eq:SerDetMeas:01}
\sum\limits_{s=0}^{\infty} \left( \prod_{j=1}^n \varphi_j(u_s)\prod_{k=1}^m \Phi_k(u_s)\right)^2.
\end{equation}
Take any $s\in\N$. Since each $\varphi_j$ is non-increasing, it follows that
\begin{align*}
\sum\limits_{u_{s-1}\leq q < u_s} \, \frac{1}{q} \left( \prod_{j=1}^n \varphi_j(q) \prod_{k=1}^m \Phi_k(q) \right)^2
&\leq 
\sum\limits_{u_{s-1}\leq q < u_s} \, \frac{1}{u_{s-1}} \left(\prod_{j=1}^n \varphi_j(u_{s-1}) \prod_{k=1}^m c_2\Phi_k(u_{s-1}) \right)^2  &&\text{ (by }\eqref{EQ:LinearForms:C})\\
&=
\left( \frac{u_s}{u_{s-1}}-1\right) c_2^{2m} \,  \left(\prod_{j=1}^n \varphi_j(u_{s-1}) \prod_{k=1}^m  \Phi_k(u_{s-1}) \right)^2 \\
&=
\left(M-1\right) c_2^{2m} \,  \left(\prod_{j=1}^n \varphi_j(u_{s-1}) \prod_{k=1}^m  \Phi_k(u_{s-1}) \right)^2 ;
\end{align*}
as a consequence, 
\begin{align*}
\sum\limits_{q=1}^{\infty} \frac{1}{q} \left( \prod_{j=1}^n \varphi_j(q) \prod_{k=1}^m \Phi_k(q) \right)^2
&=
\sum\limits_{s=1}^{\infty} \sum\limits_{u_{s-1}\leq q < u_n} \, \frac{1}{q} \left(\prod_{j=1}^n \varphi_j(q) \prod_{k=1}^m \Phi_k(q) \right)^2\\
&\leq 
\left(M-1\right) c_2^{2m} \,
\sum\limits_{s=1}^{\infty} \left(\prod_{j=1}^n \varphi_j(u_{s-1}) \prod_{k=1}^m \Phi_k(u_{s-1})\right)^2.
\end{align*}
Similarly, we may see that
\begin{align*}
\sum\limits_{u_{s-1}\leq q < u_{s} }\frac{1}{q} \left( \prod_{j=1}^n \varphi_j(q)\prod_{k=1}^n \Phi_k(q)\right)^2 
&\geq 
\sum\limits_{u_{s-1}\leq q < u_{s} } \frac{1}{u_s} \left( \prod_{j=1}^n \varphi_j(u_s)\prod_{k=1}^n \Phi_k(u_{s-1})\right)^2  \\
&\geq 
\frac{1}{c_2^{2m}} \left( 1 - \frac{1}{M}\right)  \left( \prod_{j=1}^n \varphi_j(u_s)\prod_{k=1}^n \Phi_k(u_{s})\right)^2.
\end{align*}
Therefore, the convergence of the series in \eqref{Eq:SerDetMeas:01} is equivalent to that of 
\begin{equation}\label{Eq:SerDetMeas:02}
    \sum\limits_{q=1}^{\infty} \frac{1}{q} \left( \prod_{j=1}^n \varphi_j(q)\prod_{k=1}^m \Phi_k(q)\right)^2  .
\end{equation}
The divergence part of Theorem \ref{TEO:MEASURE} follows from Theorem \ref{KW ambient measure}. In order to prove the convergence part, for each $u\in \N$, let $\cA(u;\varphi,\Phi)$ be the set of matrices $A\in I_{\C}^{m\times n}$ for which there is some non-zero $\q\in\Z[i]^m$ satisfying
\begin{align*}
    \llbracket \q A_j\rrbracket_{\infty}< \varphi_j(u) &\quad (1\leq j\leq n), \\
    |q_k|_{\infty} \leq \Phi_k(u) &\quad (1\leq k\leq m).
\end{align*}
If $\q\in \Z[i]^m$, $\q\neq 0$, verifies $|q_k|_{\infty} \leq \Phi_k(u)$ for $k\in \{1,\ldots,m\}$, define
\[
\cA_{\q}(u;\varphi,\Phi)
:=
\left\{ A\in I_{\C}^{m\times n}: \llbracket \q A_j\rrbracket_{\infty}< \varphi_j(u) \quad (1\leq j\leq n) \right\}
\]
and, for all $\p\in\Z[i]^n$,
\[
\cA_{\q,\p}(u;\varphi,\Phi)
:=
\left\{ A\in I_{\C}^{m\times n}: |\q A_j-p_j|_{\infty}< \varphi_j(u) \quad (1\leq j\leq n)  \right\}.
\]
The following estimates hold for some constants depending on $m$ and $n$:
\begin{align*}
    \#\left\{\q\in\Z[i]^m:  |q_k|_{\infty}\leq \Phi_k(u) \quad (1\leq k\leq m) \right\} &\ll_{m,n} \left(\prod_{k=1}^m\Phi_k(u)\right)^2,\\
    \#\left\{\p\in\Z[i]^n: \cA_{\q,\p}(u;\varphi,\Phi)\neq \varnothing \right\} &\ll_{m,n} \|\q\|_{\infty}^{2n}\\
    \mu_{m\times n}^{\C}(\cA_{\q,\p}(u;\varphi,\Phi)) &\ll_{m,n} \frac{1}{\| \q\|_{\infty}^{2n}} \left(\prod_{j=1}^n \varphi_j(u)\right)^2.
\end{align*}
Therefore, we have
\[
\mu_{m\times n}^{\C} ( \cA(u;\varphi,\Phi))
\ll_{m,n}
\left(\prod_{j=1}^n \varphi_j(u) \prod_{k=1}^m\Phi_k(u) \right)^2.
\]
Choose $s\in \N$ such that $u_{s-1}\leq u< u_s$, so $\varphi(u_{s})\leq \varphi(u)\leq \varphi(u_{s-1})$ and $\Phi(u_{s-1})\leq \Phi(u)\leq \Phi(u_{s})$. Hence, by \eqref{EQ:LinearForms:C}, 
\[
\cA(u;\varphi,\Phi)
\subseteq
\cA(u_{s-1};\varphi,c_2\Phi),
\]
which implies
\[
\limsup_{u\to\infty} \cA(u;\varphi,\Phi)
\subseteq 
\limsup_{s\to\infty} \cA(u_s;\varphi,c_2 \Phi).
\]
Finally, since the series in \eqref{Eq:SerDetMeas:01} converges if and only if the series in \eqref{Eq:SerDetMeas:02} converges, the result follows from the Borel-Cantelli lemma. 
\end{proof}
\begin{proof}[Proof of Theorem \ref{TEO:MEASURE}]
Since $|z|_{\infty}\leq |z|\leq \sqrt{2}|z|_{\infty}$ for all $z\in \C$, we have
\[
W_{\infty}(\varphi,\Phi)
\subseteq
W(\varphi,\Phi)
\subseteq
W_{\infty}(\sqrt{2}\,\varphi,\sqrt{2}\,\Phi).
\]
When the series in \eqref{Eq:SerDetMeas:02} converges, the set $W_{\infty}(\sqrt{2}\,\varphi,\sqrt{2}\,\Phi)$ is null and $W(\varphi,\Phi)$ is also null. The divergence of the series implies the full measure of $W_{\infty}(\varphi,\Phi)$ and, hence, the full measure of $W(\varphi,\Phi)$.
\end{proof}
\subsection{Proof of Theorem \ref{TEO:HDIM}}
Consider again the norm $\| \cdot\|_{\infty}$. Pick $\boldsymbol{\tau} =(\tau_1, \ldots, \tau_n)\in \R^n$ satisfying \eqref{EQ:CondOnTau}. Let $\Phi=(\Phi_1, \ldots, \Phi_m)$ be determined by 
\[
\Phi_k(q)=q
\quad (1\leq k \leq m, \; q\in \N)
\]
and let $\varphi=(\varphi_1, \ldots, \varphi_n)$ be given by
\[
\varphi_j(q)= \frac{1}{q^{\tau_j-1}}
\quad (1\leq j \leq n, \; q\in \N).
\]
If $\eta= \frac{1}{2}(\tau_1 + \cdots + \tau_n - n - m)>0$, we have for all $q\in\N$ 
\[
\frac{1}{q} \left( \prod_{j=1}^n \varphi_j(q) \prod_{k=1}^m \Phi_k(q)\right)^2
=
\frac{1}{q^{1+\eta}},
\]
and Theorem \ref{TEO:MEASURE} implies $\mu_{m\times n}^{\C}(W(\varphi,\Phi))=0$. It is trivial to choose $\boldsymbol{\tau}=(\tau_1,\ldots, \tau_n)$ such that for some $j\in \{1,\ldots, n\}$ we have
\[
\tau_1 + \cdots + \tau_n>  n\tau_j + n + m;
\]
(for example, $\tau_1= \cdots = \tau_{n-1}= 3(n+m)$ and $\tau_n=2$). For such $\boldsymbol{\tau}$ and $j$, the function $\rho_j$ defined as in the previous section does not converge to $0$ when its argument tends to $\infty$. However, we may find a suitable $n$-tuple of positive functions $\tilde{\rho}=(\tilde{\rho}_1,\ldots,\tilde{\rho}_n)$ and a sequence $(\tilde{l}_s)_{s\geq 1}$ such that, for a sufficiently large $M$, the system $( (R_{\alpha})_{\alpha\in J}, \beta)$ is ubiquitous with respect to $\tilde{\rho}$ and $(\tilde{l}_s)_{s\geq 1}$, $(u_s)_{s\geq 1}$.

Let $M\in \N$ large (we determine how large $M$ should be below) and take $J$, $R_{\alpha}$ for $\alpha\in J$, $\beta$, and $(u_s)_{s\geq 1}$ as in Lemma \ref{Lem:UbiqSyst:CxLF}. Call $(\tilde{l}_s)_{s\geq 1}$ the sequence given by $l_s=M^{s-1}$ for all $s\in \N$. In this context, we have 
\[
\beta_{\alpha}=\|\mathbf{q}\|_{\infty}
\quad \left(\alpha=(\mathbf{q},\mathbf{p}) \in J\right).
\]
Given any $\bfa=(a_1,\ldots, a_n)\in \R^n$ such that
\begin{equation}\label{Eq:CondOnA}
\min_{1\leq j\leq n} a_j \geq 1
\;\text{ and }\;
\sum\limits_{j=1}^n a_j = m+n,
\end{equation}
define $\tilde{\rho} =(\tilde{\rho}_1 ,\ldots,\tilde{\rho}_n):\N\to\R^n$ by
\[
\tilde{\rho}_j(u) = \sqrt{2} \,\frac{1}{u^{a_j-1}}
\quad
(1\leq j\leq n,\; u\in\N).
\]
\begin{lemma}\label{Lem:UbiqSyst:CxLF:HD}
For any $\mathbf{a} =(a_1,\ldots, a_n) \in \R^n$ satisfying \eqref{Eq:CondOnA}, the system $(\{R_{\alpha}\}_{\alpha\in J}, \beta)$ is ubiquitous with respect to $\tilde{\rho}$ and $(\tilde{l}_s)_{s\geq 1}$, $(u_s)_{s\geq 1}$.
\end{lemma}
We omit several steps in the forthcoming proof, for it resembles that of Lemma \ref{Lem:UbiqSyst:CxLF}.
\begin{proof}
We may show--as we did with \eqref{Eq:Propo:LemUbiqSys}--that for every $A$ in the interior of $I_{\C}^{m\times n}$ and any large $s\in \N$ there is some $\alpha=(\mathbf{q},\mathbf{p})\in J$ such that $1\leq \|\mathbf{q}\|_{\infty} \leq u_s$ and 
\[
A
\in 
\prod_{j=1}^n \Delta\left( R_{\alpha,j}, \frac{\sqrt{2}}{\|\mathbf{q}\|_{\infty} u_s^{a_j-1}} \right).
\]
Similar to \eqref{Eq:LemUbiqPrel}, we may choose a sufficiently large $M$ for which any ball $B=\prod_{j=1}^n B_j\subseteq X$ and any large $s\in \N$ (depending on the radius of $B$) satisfy
\begin{align*}
\mu_{m\times n}^{\C}\left( B\cap \bigcup_{\substack{\alpha =(\mathbf{q},\mathbf{p})\in J \\ \beta_{\alpha}< \tilde{l}_s}} \prod_{j=1}^n \Delta\left( R_{\alpha,j}, \frac{ \sqrt{2} }{\|\mathbf{q}\|_{\infty} u_s^{a_j-1}} \right) \right)
&=
\mu_{m\times n}^{\C}\left( \bigcup_{q=1}^{\tilde{l}_s-1} \bigcup_{\substack{\alpha\in J\\ \beta_{\alpha}=q}} \prod_{j=1}^n   B_j\cap \Delta\left( R_{\alpha,j}, \frac{ \sqrt{2} }{\|\mathbf{q}\|_{\infty} u_s^{a_j-1}} \right)\right) \\
&\leq 
\frac{1}{2}\mu_{m\times n}^{\C}(B),
\end{align*}
which implies
\[
\mu_{m\times n}^{\C}\left( B\cap \bigcup_{\substack{\alpha =(\mathbf{q},\mathbf{p})\in J \\ \tilde{l}_s\leq \beta_{\alpha} \leq u_s}} \prod_{j=1}^n \Delta\left( R_{\alpha,j}, \frac{\sqrt{2}}{ u_s^{a_j-1}}  \right) \right)
\geq 
\frac{1}{2} \mu_{m\times n}^{\C}(B).
\]
\end{proof}
For any $\boldsymbol{\tau}=(\tau_1,\ldots, \tau_n)\in\R^n$ satisfying \eqref{EQ:CondOnTau}, define the sets
\[
W^{\C}_{\infty}(\boldsymbol{\tau})
:=
\left\{ A\in I_{\C}^{m\times n}: \left\llbracket \mathbf{q}  A_j \right\rrbracket_{\infty} < \frac{1}{\|\mathbf{q}\|_{\infty}^{\tau_j-1}} \quad(1\leq j \leq n) \text{ for i. m. } \mathbf{q}\in \Z[i]^m\right\}
\]
and
\[
\widetilde{W}^{\C}_{\infty}(\gamma,\boldsymbol{\tau})
:=
\bigcap_{Q\in \N}
\bigcup_{\substack{\mathbf{q}\in \Z[i]^m\\ \|\mathbf{q}\|_{\infty}=Q}}
\bigcup_{\substack{\mathbf{p}\in \Z[i]^n\\ \alpha=(\mathbf{q},\mathbf{p})\in J}} 
\prod_{j=1}^n \Delta\left( R_{\alpha,j}, \frac{\gamma}{\|\mathbf{q}\|_{\infty}^{\tau_j}} \right)
\quad (\gamma>0).
\]
Clearly, when we consider $\|\,\cdot\,\|_{\infty}$ and $|\,\cdot\, |_{\infty}$ and the function $\Psi=(\psi_1,\ldots, \psi_n)$ given by $\psi_j(u)= \gamma u^{-\tau_j}$, the set $\widetilde{W}^{\C}_{\infty}(\gamma,\boldsymbol{\tau})$ is precisely $W^{I_{\C}^{m\times n} }(\Psi)$ as defined in Section \ref{SEC:TOOLBOX}.
\begin{lemma}\label{Lem:HD:Wgamma}
If $\boldsymbol{\tau}=(\tau_1,\ldots,\tau_n)\in\R^n_{+}$ satisfies \eqref{EQ:CondOnTau} and $\gamma>0$, then
\[
\dimh \widetilde{W}_{\infty}(\gamma,\boldsymbol{\tau}) 
= 
\min\{s_1(\boldsymbol{\tau}),\ldots, s_n(\boldsymbol{\tau})\}.
\]
\end{lemma}
\begin{proof}
Let us further assume that $\gamma=1$. \par

\textbf{Upper bound.} First, for all $Q\in \N$ we have 
\[
\#\{ \mathbf{q}\in \Z[i]^m: \|\mathbf{q}\|_{\infty}=Q\}
\asymp_m
Q^{2m-1},
\]
because
\begin{align*}
\#\left\{ \mathbf{q}\in \Z[i]^m:\| \mathbf{q}\|_{\infty}=Q\right\} 
&= \#\left\{ \mathbf{q}\in \Z^{2m}:\| \mathbf{q}\|_{\infty}=Q\right\}  \\
&= \#\left\{ \mathbf{q}\in \Z^{2m}:\| \mathbf{q}\|_{\infty}\leq Q\right\} - \#\left\{ \mathbf{q}\in \Z^{2m}:\| \mathbf{q}\|_{\infty}\leq Q -1 \right\}  \\
&=(2Q+1)^{2m} - (2Q-1)^{2m} \\
&= 2\sum\limits_{j=0}^{2m-1}(2Q + 1)^{2m-1-j}(2Q - 1)^{j},
\end{align*}
which is a polynomial of degree $2m-1$ in $Q$. %As in the proof of Lemma \ref{LF:C:AuxLemma}, we have 
%\[
%\widetilde{W}_{\infty}(1;\boldsymbol{\tau})
%\subseteq
%\bigcap_{Q\in \N} 
%\bigcup_{\substack{\mathbf{q}\in\Z[i]^m \\ \|\mathbf{q}\|=Q}} 
%\bigcup_{\substack{ \mathbf{p} \\ (\mathbf{q},\mathbf{p})\in J}} 
%\prod_{j=1}^n \triangle\left( R_{\alpha,j},  \frac{\sqrt{2}}{\|\mathbf{q}\|_{\infty}^{\tau_j}} \right).
%\]
For all $j,l\in \{1,\ldots, m\}$ the number of balls of radius $\|\mathbf{q}\|^{-\tau_j}$ required to cover $\triangle(R_{\alpha,l},\|\mathbf{q}\|_{\infty}^{- \tau_l})$ is asymptotically equivalent to
\[
\left(\max\left\{ 1, \frac{\|\mathbf{q}\|_{\infty}^{-\tau_l} }{\|\mathbf{q}\|_{\infty}^{-\tau_j}}\right\} \|\mathbf{q}\|_{\infty}^{\tau_j(m-1)}\right)^2.
\]
Hence, for all $s>0$, we have
\begin{align*}
\cH^s\left( \widetilde{W}_{\infty}(1,\boldsymbol{\tau}) \right)
&\ll
\liminf_{N\to\infty} \sum\limits_{Q\geq N} Q^{2m-1}Q^{2n} Q^{-s\tau_j} Q^{2\tau_j n(m-1)} Q^{2 \sum\limits_{\tau_l<\tau_j} \tau_j - \tau_l} \\
&=
\liminf_{N\to\infty} \sum\limits_{Q\geq N} Q^{2m+2n-1 + 2\tau_j n(m-1) + 2 \sum\limits_{\tau_l<\tau_j} \tau_j - \tau_l}  Q^{-s\tau_j}.
\end{align*}
The last series above converges if and only if the exponent of $Q$ is strictly less than $-1$, or equivalently if $s> s_j(\boldsymbol{\tau})$, so $\dimh W_{\infty}(1,\boldsymbol{\tau})\leq s_j$.\\

\textbf{Lower bound.} Let $J$, $(R_{\alpha})_{\alpha\in J}$, $\beta$, $(\tilde{l}_s)_{s\geq 1}$ and  $(u_s)_{s\geq 1}$ as in Lemma \ref{Lem:UbiqSyst:CxLF:HD}. We only establish the setup, the computations are quite similar to the $p$-adic case. Suppose without loss of generality that $\tau_1 \geq  \tau_2 \geq \ldots \geq \tau_n > 1$ and recall that $\delta_k=2m$ for $k\in\{1,\ldots, n\}$. 
First, assume that $\tau_n \geq \frac{n+m}{n}$. For $j\in\{1,\ldots, n\}$ define
\[
a_j:= \frac{n+m}{n} 
\;\text{ and }\; 
t_j:= \tau_j - a_j.
\]
Then, the order of $\cA$ is 
\[
a_1 + t_1
\geq 
a_2 + t_2
\geq
\ldots
\geq
a_n + t_n
\geq 
a_1 = \ldots = a_n. 
\]
Suppose that $\tau_n < \frac{n+m}{n}$ and let $K\in\{1,\ldots, n\}$ be the largest integer such that
\[
\tau_K
>
\frac{ m + n - (\tau_{K+1} + \cdots + \tau_n) }{K}.
\]
For each $j\in\{1,\ldots, n\}$, write
\[
a_j
:=
\begin{cases}
\tau_j, &\quad  \quad (K+1\leq j \leq n), \\[2ex]
\displaystyle\frac{m+n - (\tau_{K+1} + \cdots + \tau_d) }{K}, &\quad \quad (1\leq j\leq K).
\end{cases}
\]
Then, $\cA$ is ordered as follows:
\[
a_1 + t_1 
\geq 
a_2 + t_2
\geq 
\ldots
\geq
a_K + t_K
> 
a_1= \ldots = a_K 
>
a_{K+1}  = a_{K+1}  + t_{K+1}  
\geq 
\ldots
\geq 
a_{n}  = a_{n}  + t_{n}.
\]
By Theorem \ref{MTPRR}, $\dimh \widetilde{W}_{\infty}(1,\boldsymbol{\tau})=\min\{s_1(\boldsymbol{\tau}), \ldots, s_n(\boldsymbol{\tau})\}$. 

Now assume that $\gamma>0$ is arbitrary and take $0<\varepsilon<\max_{1\leq j \leq n} \tau_j -1$. Since $\|\mathbf{q}\|_{\infty}^{\varepsilon}>\gamma$ for all but finitely many $\mathbf{q}\in \Z[i]^m$, we have
\[
\widetilde{W}_{\infty} (1, \tau_1 + \varepsilon, \ldots, \tau_d + \varepsilon)
\subseteq
\widetilde{W}_{\infty} (\gamma,\boldsymbol{\tau})
\subseteq
\widetilde{W}_{\infty} (1, \tau_1 - \varepsilon, \ldots, \tau_d - \varepsilon).
\]
The lemma follows by letting $\varepsilon\to 0$, because $\tau\mapsto \dimh \widetilde{W}_{\infty}(1,\boldsymbol{\tau})$ is continuous.
\end{proof}
\begin{lemma}
If $\boldsymbol{\tau}$ satisfies \eqref{EQ:CondOnTau}, then
\[
\dimh W_{\infty} (\boldsymbol{\tau})
=
\min\{s_1(\boldsymbol{\tau}), \ldots, s_n(\boldsymbol{\tau})\}.
\]
\end{lemma}
\begin{proof}
Arguing as in the proof of \eqref{Eq:Propo:LemUbiqSys} and using \eqref{Eq:EquivNorm}, we obtain
\[
\widetilde{W}_{\infty}\left((2m)^{-1},\tau\right)
\subseteq 
W_{\infty} (\boldsymbol{\tau})
\subseteq 
\widetilde{W}_{\infty}\left(\sqrt{2},\tau\right).
\] 
The result now follows from Lemma \ref{Lem:HD:Wgamma}.
\end{proof}
In a similar fashion, we may show that the corresponding set obtained from the usual complex absolute value $|\,\cdot\,|$ has the same Hausdorff dimension:
\begin{equation}\label{Eq:HD:Wlambda:NormaUsual}
\dimh W^{\C}_{n,m}(\boldsymbol{\tau})
=
\min\{s_1(\boldsymbol{\tau}), \ldots, s_n(\boldsymbol{\tau})\}.
\end{equation}
%From this point, Theorem \ref{TEO:HDIM} is obtained from \eqref{Eq:HD:Wlambda:NormaUsual} by arguing just as in \cite[Section 10.2]{WW19}.

\section{Quaternion approximation}

In this section, we study sets of linear forms in quaternion space. We refer the reader to \cite[Section 20]{HardyWright1979} for the elementary algebraic aspects of quaternions and Hurwitz integers, to \cite{DodsonEveritt2014} for a beautiful account on metrical Diophantine approximation aspects of quaternions, and to \cite{ConwaySmith2003} for an overview of the theory. It is worth stressing that no metrical results whatsoever are known in the higher (quaternion) dimensions, other than what is presented below.

Fix two natural numbers $m$ and $n$. Let $i,j$ be two symbols and define the operations 
\begin{equation}\label{Eq:Quat:01}
i^2=-1, 
\;
j^2=-1, 
\;
ij=-ji
\end{equation}
We write $k:=ij$. The \textit{ring of quaternions} $\mathbb{H}$ is the skew field whose elements are the objects $a+bi+cj+dk$ with $a,b,c,d\in\R$ along with the sum defined coordinate-wise and the product obtained by following the usual rules and \eqref{Eq:Quat:01}. Given any $\xi=a+bi+cj+dk\in \mathbb{H}$, its \textit{conjugate} $\overline{\xi}$ is
\[
\overline{\xi}=a - bi -cj  - dk
\]
and its \textit{norm} $|\xi|$ is
\[
|\xi|=\sqrt{a^2+b^2+c^2+d^2},
\]
so $|\xi|=\sqrt{\xi\overline{\xi}}$. The definition of the product implies that for any $\xi,\zeta\in \mathbb{H}$ we have $\overline{\xi\zeta}=\overline{\zeta} \,\overline{\xi}$ and, hence, $|\xi\zeta|=|\xi||\zeta|$. 

When regarded as a real vector space, $\mathbb{H}$ is isomorphic to $\R^4$ and an isomorphism is determined by 
\[
1\mapsto(1,0,0,0), 
\;
i\mapsto(0,1,0,0), 
\;
j\mapsto(0,0,1,0), 
\;
k\mapsto(0,0,0,1).
\]
Under this identification between $\mathbb{H}$ and $\mathbb{R}^4$, the real bi-linear map $\langle \,\cdot\, , \, \cdot\,\rangle: \mathbb{H}\times \mathbb{H} \to \mathbb{R}$ given by 
\[
\langle \zeta, \xi \rangle:= \frac{1}{2}\left( \overline{\zeta}\xi - \zeta\overline{\xi}\right)
\quad
(\zeta, \xi\in \mathbb{H})
\]
is the usual inner product. As a consequence, the function from $\mathbb{H}^n\times \mathbb{H}^n$ to $\mathbb{R}$ given by
\[
\left\langle (\xi_1,\ldots,\xi_n), (\zeta_1,\ldots, \zeta_n)\right\rangle_{\mathbb{H}^n}
\mapsto
\sum_{j=1}^n \langle \xi_j,\zeta_j\rangle_{\mathbb{H}}
\]
for all $(\xi_1,\ldots,\xi_n)$, $(\zeta_1,\ldots, \zeta_n)\in \mathbb{H}^n$ is the usual inner product on $\mathbb{R}^{4n}$.

The set of \textit{Lipschitz integers}, given by 
\[
\{a + ib + cj +dk: a,b,c,d\in \Z\},
\]
is the obvious choice for integers in $\mathbb{H}$. However, it is customary to work instead with the \textit{Hurwitz integers} $\Z_{\mathbb{H}}$, defined as
\[
\Z_{\mathbb{H}}
:= 
\left\{ \frac{a}{2} + \frac{b}{2} i + \frac{c}{2} j + \frac{d}{2} k: a,b,c,d\in \Z \text{ and } a\equiv b\equiv c \equiv d \pmod 2\right\}.
\]
Clearly, the Lipschitz integers are contained in $\Z_{\mathbb{H}}$.
Hurwitz integers are preferred over Lipschitz integers because they are a Euclidean domain while Lipschitz integers are not \cite[Chapter 5]{ConwaySmith2003}. 
Hurwitz integers have $24$ invertible elements or \textit{units}: 
\[
\pm 1, \; \pm i , \; \pm j , \;\pm k, \;
\frac{1}{2}\left( \pm 1 +  \pm i + \pm j + \pm k\right).
\]

As noted in \cite{DodsonEveritt2014}, $\Z_{\mathbb{H}}$ is the additive subgroup of $\mathbb{H}$ generated by $i$, $j$, $k$, $\frac{1+i+j+k}{2}$ and, as a sub-lattice of $\R^4$, the determinant of $\Z_{\mathbb{H}}$ is $\frac{1}{2}$. Let $I_{\mathbb{H} }$ be the Voronoi region for $\Z_{\mathbb{H}}$ containing $0$, that is
\[
I_{\mathbb{H} }
:=
\{ \xi\in \mathbb{H}:  |\xi|\leq |\xi - \zeta| \text{ for all } \zeta\in \Z_{\mathbb{H}}\}.
\]
If $\llbracket \xi \rrbracket_{\mathbb{H}}$ denotes the distance between $\xi\in\mathbb{H}$ and its nearest Hurwitz integer, then $I_{\mathbb{H} }$ is precisely the set of quaternions $\xi$ for which $\llbracket \xi\rrbracket_{\mathbb{H}} = |\xi|$. Observe that $I_{\mathbb{H} }$ is the convex hull of $\{\pm \frac{1}{2}, \pm \frac{1}{2}i, \pm \frac{1}{2}j, \pm \frac{1}{2}k\}$ and that its Lebesgue measure is $\frac{1}{2}$. Call $I_{\mathbb{H}}^{m\times n}$ the set of $m\times n$ matrices $A$ with entries in $I_{\mathbb{H}}$.

Given an $n$-tuple of non-increasing positive functions $\varphi=(\varphi_1,\ldots, \varphi_n):\N \to\R^n$ such that
\[
\lim_{q\to\infty} \phi_j(q)=0
\quad (1\leq j\leq n)
\]
and an $m$-tuple of non-decreasing positive functions $\Phi=(\Phi_1,\ldots, \Phi_m):\N \to \R^m$ such that
\[
\lim_{q\to\infty} \Phi_k(q)=\infty
\quad (1\leq k\leq m),
\]
we call $W_{n,m}^{\mathbb{H}}(\varphi,\Phi)\subseteq I_{\mathbb{H}}^{m\times n}$ the set of $m\times n$ matrices $A$ with the following property: 
\begin{center}
    there are infinitely many integers $u\geq 1$  such that the system 
\begin{align*}
\left\llbracket \q \,A_j  \right\rrbracket &< \varphi_j(u) \quad (1\leq j\leq n)\\
|q_k| &\leq \Phi_k(u)  \quad (1\leq k\leq m)
\end{align*}
has a non-zero solution $\q\in \Z_{\mathbb{H}}^{1\times m}$.
\end{center}
We denote by $\mu_{m\times n}^{\mathbb{H}}$ the Lebesgue measure on $\mathbb{H}^{m\times n}$. 
\begin{theorem}\label{TEO:Q:MEASURE}
If there are some constants $N_0,M\in \N$, $M\geq 2$, and $c_1, c_2>1$ such that for every $j\in\N_{\geq N_0}$ we have
\begin{equation}\label{EQ:Q:LinearForms}
c_1\Phi_k(M^j) \leq \Phi_k(M^{j+1}) \leq c_2\Phi_k(M^j)
\quad (1 \leq k \leq m),
\end{equation}
then, 
\[
\mu_{m\times n}^{\mathbb{H}} \left(W_{n,m}^{\mathbb{H}}(\varphi, \Phi)\right)
=
\begin{cases}
0, &\text{\rm if }\quad \displaystyle\sum_{q=1}^{\infty} \frac{1}{q}\left( \displaystyle\prod_{j=1}^n \varphi_j(q) \displaystyle\prod_{k=1}^m \Phi_k(q)\right)^4   < \infty, \\
\mu_{m\times n}^{\mathbb{H}}(I_{\mathbb{H}}^{m\times n} ), &\text{\rm if }\quad \displaystyle\sum_{q=1}^{\infty} \frac{1}{q} \left( \displaystyle\prod_{j=1}^n \varphi_j(q) \displaystyle\prod_{k=1}^m \Phi_k(q)\right)^4   = \infty.
\end{cases}
\]
\end{theorem}
In \cite{DodsonEveritt2014}, Dodson and Everitt solved the one-dimensional case using the theory of ubiquitous systems of balls as introduced in \cite{BDV06}. More precisely, given a non-increasing function $\psi:\R_{+}\to\R_{+}$ such that $\psi(t)=\psi(\lfloor t \rfloor)$ for all $t>0$, they proved that the set
\[
V(\psi)
:=
\left\{ \xi \in I_{\mathbb{H}}: |\xi q - p | < |q|\psi(|q|) \text{ for i.m. } p,q\in \Z_{\mathbb{H}}\right\}
\]
is of either zero or full measure according to the convergence or divergence of
\[
\sum_{q=1}^{\infty} q^7 \psi(q)^4.
\]
When $m=n=1$, the function $\Phi_1$ is the identity map and $\varphi$ is decreasing, Theorem \ref{TEO:Q:MEASURE} tells us that the set 
\[
\left\{ \xi\in I_{\mathbb{H}}: |q\xi - p| < \varphi(\lfloor |q|\rfloor ) \text{ for i.m. } p,q\in \Z_{\mathbb{H}}\right\}
\]
is of zero or full measure according to the convergence or divergence of
\[
\sum_{q=1}^{\infty} q^3 \varphi(q)^4.
\]
Although quaternions are not commutative, the proof of Theorem \ref{TEO:Q:MEASURE} also allows us to conclude the 0-1 dichotomy for the set 
\[
\left\{ \xi\in I_{\mathbb{H}}: |\xi q - p| < \varphi(\lfloor |q|\rfloor ) \text{ for i.m. } p,q\in \Z_{\mathbb{H}}\right\}.
\]
The result of Dodson and Everitt now follows.

%In particular, it was proved that the Lebesgue measure of the set is either zero or full if the series $\sum_{r=1}^\infty r^7\varphi(r)^4$ converges or diverges respectively. The starting point of Dodson and Everitt was the inequality $\left| \xi - pq^{-1}\right| < \varphi(|q|)$, while ours is $|q\xi - p|< \varphi\left(\Phi^{-1}(|q|)\right)$. In particular, if each $\Phi_{k}$ is the identity, that is $\Phi_{k}(r)=r$, then our summation becomes 
%\begin{equation*}     \sum_{q=1}^{\infty} q^{4m-1}\prod_{j=1}^{n}\varphi_{j}(q)\, . \end{equation*}

%We obtain the Hausdorff dimension of the null sets when $\Phi_k(u)=\lfloor u\rfloor$ for all $u\geq 1$ and all $k$. In this case, write $W(\varphi):= W(\varphi, \Phi)$.
%\begin{theorem}\label{TEO:Q:HDIM}
%Let $\mathbf{a}=(a_1,\ldots, a_n)\in (1,+\infty)^n$ be such that $a_1 + \ldots + a_n = n + m$. Let $\varphi=(\varphi_1,\ldots, \varphi_n):\N\to\R_{+}$ be such that each $u\mapsto \varphi_j(u)/u^{a_j}$ is decreasing. If $\cU$ is bounded, then
%\[
%\dimh W(\psi)
%=
%\sup\{ s(\boldsymbol{\tau}): \tau\in \cU\}.
%\]
%\end{theorem}{\color{red} what is $\psi$ here? this section needs more elaboration. Not sure what is the precise dimension result here and where is its proof}

Let $\boldsymbol{\tau}=(\tau_1,\ldots, \tau_n)\in\R^n$ be any vector for which \eqref{EQ:CondOnTau} holds. Define 
 \[
s_j(\boldsymbol{\tau})
:= 
4n(m-1) 
+ 
4\, \frac{m + n - \sum\limits_{r\,:\, \tau_{r}<\tau_{j}} (\tau_r-\tau_j)}{\tau_j}
\quad
(1\leq j\leq n)
\]
and 
\[
W_{n,m}^{\mathbb{H}}(\boldsymbol{\tau})
:=
\left\{ A\in I_{\mathbb{H}}^{m\times n} : \left\llbracket \mathbf{q}  A_j \right\rrbracket < \frac{1}{\|\mathbf{q}\|^{\tau_j-1}} \; (1\leq j \leq n) \text{ for i. m. } \mathbf{q}\in \Z_{\mathbb{H}}^m\right\}.
\]
%Our Hausdorff dimension result on quaternions reads as follows.
\begin{theorem}\label{TEO:Q:HDIM}
If $\tau\in\R^n$ satisfies \eqref{EQ:CondOnTau}, then
\[
\dimh W_{n,m}^{\mathbb{H}}(\boldsymbol{\tau}) = \min\{ s_1(\boldsymbol{\tau}), \ldots, s_n(\boldsymbol{\tau})\}.
\]
Moreover, $\cH^{s}\left( W_{n,m}^{\mathbb{H}}(\boldsymbol{\tau})\right)= \infty$.
\end{theorem}
Theorems \ref{TEO:Q:MEASURE} and \ref{TEO:Q:HDIM} can be shown as their complex counterpart. We thus omit a considerable amount of detail in their proofs. First, as before, we replace $|\,\cdot\,|$ with the more manageable norm $|\,\cdot\,|_{\infty}$ given by
\[
|\xi|_{\infty}:= \max\{|a|, |b|, |c|, |d|\}
\quad
\left( \xi=a+bi+cj+dk \in \mathbb{H}\right)
\]
and $\llbracket \,\cdot\,\rrbracket$ with the function $\llbracket \,\cdot\,\rrbracket_{\infty}:\mathbb{H}\to\R_{+}$ given by
\[
\llbracket \xi \rrbracket_{\infty} := \min\{|\xi - \zeta|_{\infty}: \zeta\in \Z_{\mathbb{H}} \}
\quad (\xi\in \mathbb{H}).
\]
The quaternion version of Minkowski's theorem for linear forms, Lemma \ref{Teo:Mink:Q:LF} below, follows from Minkowski's Theorem and $\det(\Z_{\mathbb{H}})=\frac{1}{2}$.

\begin{lemma}\label{Teo:Mink:Q:LF}
Let $\gamma_1,\ldots, \gamma_n,\theta_1,\ldots, \theta_m$ be positive numbers satisfying $\prod_{j=1}^n \gamma_j \prod_{k=1}^m \theta_k\geq \frac{1}{2^{m+n}}$. For every matrix $A\in \mathbb{H}^{m\times n}$ there exists a vector $(\q,\p)\in \Z_{\mathbb{H}}^{1\times (m + n)}$ with $\q\neq 0$ such that
\begin{align}
\left|\, \q \; A_t - p_t\right|_{\infty} &< \gamma_t \quad (1\leq t \leq n),   \nonumber \\
\left| q_r \right|_{\infty} &\leq  \theta_r \quad (1\leq r \leq m).
\end{align}
\end{lemma}
\subsection{Ubiquity for quaternions}
Given $k\in \N$, for any $\xi=(\xi_1,\ldots, \xi_k)\in \mathbb{H}^k$ define 
\[
\| \xi\|_{\infty} 
:=
\max\left\{ |z_1|_{\infty}, \ldots, |z_k|_{\infty}\right\}.
\]
Hence, if $\|\xi\|_2=\sqrt{|\xi_1|^2 + \cdots + |\xi_k|^2}$, we have
\begin{equation}
\frac{1}{2\sqrt{k}} \| \xi\|_{2}
\leq 
\| \xi \|_{\infty}
\leq
\| \xi \|_{2}.
\end{equation}
Consider a sufficiently large $M$ and the next objects:
\begin{enumerate}[i.]
\item $J:= \left\{\alpha=(\q,\p)\in \Z_{\mathbb{H}}^{m+n}: \|\p\|_{\infty} \leq 8m \|\q\|_{\infty}\right\}$,
\item $\beta:J\to\R_{+}$, $\alpha=(\q,\p)\mapsto \beta_{\alpha} := \max\left\{ \Phi_1^{-1}\left( |q_1|_{\infty}\right), \ldots, \Phi_m^{-1}\left( |q_m|_{\infty} \right)\right\}$,
\item $(u_s)_{s\geq 1}$ given by $u_s=M^s$ for all $s\in \N$,
\item For each $\alpha\in J$, write $R_{\alpha,t} :=\left\{ A_j\in I_{\mathbb{H} }^{m\times 1}: \q \,A_t = p_t\right\}$ for $t\in\{1,\ldots, n\}$ and the resonant set $R_{\alpha}$ is $R_{\alpha}:=\prod_{t=1}^n R_{\alpha,t}$,
\item In this context, $\kappa_t=1 - \frac{1}{m}$ and  $\delta_t=4m$ for $t\in\{1,\ldots,n\}$.
\end{enumerate}
Let $\rho:(\rho_1,\ldots, \rho_n):\N\to \R_{+}^n$ be given by
\[
\rho_j(u)
:=
2\,M\frac{\varphi_j(u)}{\|\Phi(u)\|_{\infty}} \left( \prod_{s=1}^n \varphi_s(u)\prod_{k=1}^m \Phi_k(u)\right)^{-1/n}
\qquad(1\leq j \leq n, \; u\in\N)
\]
and $\Psi=(\psi_1,\ldots, \psi_n): \N\to \R_{+}^n$ by
\[
\psi_j(u)
:=
\frac{\varphi_j(u )}{\|\Phi(u )\|_{\infty} }
\quad(1\leq j \leq n, \; u\in\N).
\]
Recall that, by Lemma~\ref{CI_product}, we can replace $\Psi$ with any of its multiples without altering the measure. In view of Lemma \ref{Teo:Mink:Q:LF}, we may assume
\[
\prod_{t=1}^n \varphi_t(u)\prod_{r=1}^m \Phi_r(u) \leq \frac{1}{2^{m+n}}
\quad (u\in\N).
\]
In order to prove Theorem \ref{TEO:Q:HDIM}, take $J$, $R_{\alpha}$ for $\alpha\in J$, $\beta$, and $(u_s)_{s\geq 1}$ as above, pick a sufficiently large $M$ and put $l_s=M^{s-1}$ for all $s\in \N$. Given $\bfa=(a_1,\ldots, a_n)\in \R^n$ satisfying \eqref{Eq:CondOnA}, define $\tilde{\rho} =(\tilde{\rho}_1 ,\ldots,\tilde{\rho}_n):\N\to\R^n$ by 
\[
\tilde{\rho}_j(u) =  \frac{2}{u^{a_t -1 }} 
\quad
(u\in\N)
\]
for each $t=1,\ldots, N$.
%Theorems \ref{TEO:Q:MEASURE} and \ref{Lem:UbiqSyst:CxLF:HD} are a consequence of Lemma \ref{Lem:UbiqSyst:CxLF:HD1} below. 
\begin{lemma}\label{Lem:UbiqSyst:CxLF:HD1}
The system $((R_{\alpha})_{\alpha\in J}, \beta)$ is ubiquitous with respect to $\rho$ and $(l_s)_{s\geq 1}$, $(u_s)_{s\geq 1}$. The same system is also ubiquitous with respect to $\tilde{\rho}$ and $(\tilde{l}_s)_{s\geq 1}$, $(u_s)_{s\geq 1}$ for a suitable $(\tilde{l}_s)_{s\geq 1}$.
\end{lemma}
The proof of Lemma \ref{Lem:UbiqSyst:CxLF:HD1} follows closely the proofs of lemmas \ref{Lem:UbiqSyst:CxLF} and \ref{Lem:UbiqSyst:CxLF:HD}. We leave the details to the reader.
\subsection{Proof of Theorem \ref{TEO:Q:MEASURE}}
%\begin{proof}[Proof of Theorem \ref{TEO:Q:MEASURE}]
    Condition \eqref{EQ:Q:LinearForms} and the Cauchy condensation imply that the next series are either both convergent or both divergent:
    \[
    \sum_{q=1}^{\infty} \frac{1}{q}\left(\prod_{t=1}^n  \varphi_t(q) \prod_{r=1}^m \Phi_r(q)\right)^4
    \quad\text{ and }\quad
    \sum_{s=0}^{\infty} \left( \prod_{t=1}^n \varphi_t(u_s) \prod_{r=1}^m \Phi_r(u_s) \right)^4.
    \]
    Hence, the divergence case is a consequence of Lemma \ref{Lem:UbiqSyst:CxLF:HD1} and Theorem \ref{KW ambient measure}. In order to prove the convergence case, for each $u\in \N$, let $\cA^{\mathbb{H}}(u;\varphi,\Phi)$ be the collection of matrices $A\in I_{\mathbb{H}}^{m\times n}$ such that for some non-zero $\q\in\Z[i]^{1\times m}$ we have
\begin{align*}
    \llbracket \q A_t\rrbracket_{\infty}< \varphi_t(u), &\quad (1\leq t \leq n), \\
    |q_r|_{\infty} \leq \Phi_r(u), &\quad  (1\leq r \leq m).
\end{align*}
    Therefore,
    \[
    \mu_{m\times n}^{\mathbb{H}}\left( \cA^{\mathbb{H}}(u;\varphi,\Phi)\right)\ll_{m,n} \left(\prod_{t=1}^n  \varphi_t(u) \prod_{r=1}^m \Phi_r(u)\right)^4
    \]
    and, as in the proof of Theorem \ref{TEO:MEASURE},
    \[
    \limsup_{u\to\infty} \cA^{\mathbb{H}}(u;\varphi,\Phi)
    \subseteq 
    \limsup_{u\to\infty} \cA^{\mathbb{H}}(u;\varphi,c_2\Phi).
    \]
    Finally, the Borel-Cantelli lemma gives the theorem.
%\end{proof}
\subsection{Proof of Theorem \ref{TEO:Q:HDIM} }
%\begin{proof}[Proof of Theorem \ref{TEO:Q:HDIM}]
As in the complex setting, for any $\boldsymbol{\tau}=(\tau_1,\ldots, \tau_n)\in\R^n$ such that \eqref{EQ:CondOnTau} holds, define the sets
\[
W_{\infty}^{\mathbb{H}}(\boldsymbol{\tau})
:=
\left\{ A\in I_{\mathbb{H}}^{m\times n}  : \left\llbracket \mathbf{q}  A_t \right\rrbracket_{\infty} < \frac{1}{\|\mathbf{q}\|_{\infty}^{\tau_t-1}} 
\quad (1\leq t\leq n),\,\, \text{ for i. m. } \mathbf{q}\in \Z_{\mathbb{H}}^m\right\}
\]
and, for any $\gamma>0$,
\[
\widetilde{W}_{\infty}^{\mathbb{H} }(\gamma,\boldsymbol{\tau})
:=
\bigcap_{Q\in \tfrac{1}{2}\N}
\bigcup_{\substack{\mathbf{q}\in \Z_{\mathbb{H}}^m\\ \|\mathbf{q}\|_{\infty}=Q}}
\bigcup_{\substack{\mathbf{p}\in \Z_{\mathbb{H}}^n\\ \alpha=(\mathbf{q},\mathbf{p})\in J}} 
\prod_{t=1}^n \Delta\left( R_{\alpha,t}, \frac{\gamma}{\|\mathbf{q}\|_{\infty}^{\tau_t}} \right).
\]
The core of the argument is done under the assumption $\gamma=1$. For the upper bound, note that
\[
\#\left\{ \mathbf{q}\in \Z_{\mathbb{H}}^m: \|\mathbf{q}\|_{\infty}=Q\right\}
\asymp_m
Q^{4m-1}
\quad \left(Q\in \tfrac{1}{2}\N\right),
\]
Indeed, for any $\mathbf{q}\in \Z_{\mathbb{H}}$, the quaternion $2\mathbf{q}$ is a Lipschitz integer and, for all $Q\in \tfrac{1}{2}\N$, we have $\|\mathbf{q}\|_{\infty}=Q$ if and only if $\|2\mathbf{q}\|_{\infty}=2Q$.

Also, for $t,l\in \{1,\ldots, m\}$, we need 
\[
\asymp
\left(\max\left\{ 1, \frac{\|\mathbf{q}\|_{\infty}^{-\tau_l} }{\|\mathbf{q}\|_{\infty}^{-\tau_t}}\right\} \|\mathbf{q}\|_{\infty}^{\tau_j(m-1)}\right)^4
\]
balls of radius $\|\mathbf{q}\|^{-\tau_t}$ to cover $\triangle(R_{\alpha,l},\|\mathbf{q}\|_{\infty}^{- \tau_l})$. Hence, if $s>0$, we have
\begin{align*}
\cH^s\left(\widetilde{W}^{\mathbb{H}}_{\infty}(1,\boldsymbol{\tau})\right)
&\ll
\liminf_{N\to\infty} \sum\limits_{Q\geq N} Q^{4m-1}Q^{2n} Q^{-s\tau_t} Q^{4\tau_t n(m-1)} Q^{4 \sum\limits_{\tau_l<\tau_t} \tau_t - \tau_l} \\
&=
\liminf_{N\to\infty} \sum\limits_{Q\geq N} Q^{4m+4n-1 + 4\tau_t n(m-1) + 4 \sum\limits_{\tau_l<\tau_j} \tau_t - \tau_l}  Q^{-s\tau_t}.
\end{align*}
The series converges if and only if $s> s_t(\boldsymbol{\tau})$. \\

From this point, the argument follows verbatim that of the complex case. That is, we reorder the coefficients $\tau$ to have $\tau_1\geq \ldots \geq \tau_2 \geq \ldots \geq \tau_n >1$ and we consider two cases: $\tau_n \geq \frac{n+m}{m}$ and $\frac{n+m}{m}>\tau_n$. Afterwards, we apply Theorem \ref{wang wu KG} to conclude 
\[
\dim_H \widetilde{W}^{\mathbb{H}}_{\infty}(1,\boldsymbol{\tau})
=
\min\left\{ s_1(\boldsymbol{\tau}), \ldots, s_n(\boldsymbol{\tau})\right\}.
\]
The theorem for an arbitrary $\gamma$ follows from the continuity of the dimension as a function of $\tau$. Finally, we conclude the theorem for $W_{\infty}^{\mathbb{H}}(\boldsymbol{\tau})$ by appealing to the equivalence of any two norms in a finite dimensional vector space.
%\end{proof}

\section{Formal power series approximation}
In this section we study sets of linear forms over the field of formal power series. Let $\F$ be the finite field with $t=p^r$ elements for some prime $p$ and $r\in\N$. We define \textit{the field of Laurent series with coefficients from} $\mathbb{F}$ or \textit{the field of formal power series with coefficients from} $\mathbb{F}$  to be
\begin{equation}
\LL = \left\{ \sum\limits_{i=-n}^\infty a_{-i}X^{-i}\, : \, n\in\Z,\ a_i\in\F,\ a_n\neq0\right\}\cup\{0\}.
\end{equation}
 An absolute value $\|\cdot\|$ on $\LL$ can be defined as
$$ \left\|\sum\limits_{i=-n}^\infty a_{-i}X^{-i} \right\|=t^n,\quad\,\left\|0\right\|=0.$$
 For any $\x=(x_1,\ldots, x_h)\in \LL^h$, we define the \textit{height of} $\x$ to be
$$
\|\x\|_{\infty} = \max\{\|x_1\|,\ldots,\|x_h\|\}.
$$
Note that for both $\|\cdot \|$ and $\|\cdot\|_\infty$, we have
$$
\|x+y\| \leq \max ( \|x\|, \|y\|) \quad\text{ and }\quad \|\x+\y\|_\infty \leq \max ( \|\x\|_\infty, \|\y\|_\infty).
$$
In $\LL$, the polynomial ring $\F[X]$ plays a role analogous to the one played by the integers in the field of real numbers. We define \textit{the polynomial part} of a non-zero element by
$$
\left[ \sum\limits_{i=-n}^\infty a_{-i}X^{-i} \right] = \sum\limits_{i=-n}^0 a_{-i}X^{-i}.
$$
Define the distance to $\LL[X]^h$ for a point $\x\in\LL^h$ as
$$
|\langle \x \rangle | = \min_{\p\in\F[X]^h} \| \x-\p \|_{\infty}\, .
$$
Let $$
I_{\LL}= \{ x\in \LL : \, [x] = 0 \} = B(0,1) = \{ x\in\LL : \| x \| <1\}.
$$
%We will also use the set 
%$$
%\RR= \{ t^r \, : \, r\in \Z\},
%$$
Fix $n,m\in \N$ and let $\LL^{m \times n}$ to be the set of $m\times n$ dimensional matrices with entries from $\LL$ and $I^{m \times n}_\LL$ to be the $m\times n$ dimensional matrices with entries from $I_\LL$.\\
We will also make use of the fact about the number of polynomials of fixed degree, namely
\begin{equation}\label{numberpoly}
\# \{ \q\in \F[X]^m:\, \|\q \|_\infty = t^r\}=m(r-1)t^{m-1}t^{tm}.
\end{equation}

Denote by $\mu^\LL_{m \times n}$ the $mn$-dimensional Haar measure on $\LL^{m \times n}$ normalised by $\mu^\LL_{m \times n}(I^{m \times n}_\LL)=1$.

As in the previous sections, we have an $n$-tuple $\phi=(\phi_1, \ldots, \phi_n)$ and an $m$-tuple $\Phi=(\Phi_1, \ldots, \Phi_m)$ of positive functions defined on $\N$ such that
 %As in the previous sections, we have an $n$-tuple $\{ \phi_i\}_{1\leq i \leq n}$ and an $m$-tuple $\{ \Phi_k\}_{1\leq k \leq m}$ of positive functions defined on $\N$ such that
$$
\phi_i(u)\to 0 \quad \text{ as } \quad u\to\infty  \quad (1 \leq i \leq n)
$$
and 
$$
\Phi_k(u)\to \infty \quad \text{ as } \quad u\to\infty \quad (1\leq k \leq m).
$$  Define the set
\begin{equation*}
    W^{\LL}_{n,m}(\phi,\Phi):= \left\{ A \in I^{m \times n}_\LL :\begin{array}{l} \text{the system } \left\{ \begin{array}{l} | \langle \q A_i \rangle |< \phi_{i}(u) \quad (1\leq i \leq n),  \\[2ex]
    \| q_k \| \leq \Phi_k(u) \quad  (1 \leq k \leq m), \end{array}\right.\\[3ex]
     \text{ has a solution } \q \in \F[X]^{m}\backslash\{0\} \text{ for i.m. } u\in\N 
     \end{array}
   \right\}.
\end{equation*}

We prove the following weighted version of the Khintchine-Groshve type theorem for the formal power series.
\begin{theorem}\label{formalt1}
Assume that there are constants $N_0,M \in \N_{>1}$ and $c_1, c_2>1$ such that \begin{equation}\label{formalreg}
c_1\Phi_k(M^j) \leq \Phi_k(M^{j+1}) \leq c_2\Phi_k(M^j), \quad 1\leq k \leq m, \quad \forall j\in \N_{\geq N_0}.
\end{equation}
 Then
\begin{equation*}
    \mu^\LL_{m \times n}(W^{\LL}_{n,m}(\phi,\Phi))=\begin{cases}
    0 \quad \text{\rm if }\quad  \sum\limits_{r=1}^{\infty}\frac{1}{r} \prod_{k=1}^{m} \Phi_k(r) \prod_{i=1}^{n}\phi_{i}(r) < \infty , \\[2ex]
    1 \quad \text{\rm if }\quad  \sum\limits_{r=1}^{\infty}\frac{1}{r} \prod_{k=1}^{m} \Phi_k(r) \prod_{i=1}^{n}\phi_{i}(r)  = \infty .
    \end{cases}
\end{equation*}
\end{theorem}
Below we briefly highlight the results preceding this. These include but are not limited to
\begin{itemize}
    \item $n=m=1$, $\psi$ monotonic, de Mathan \cite{deMathan1970}.
    \item $nm\geq1$, $\Psi$  univariable monotonic, Kristensen \cite{Kristensen2003}
    
    \item $n\geq1$, $m\geq2$, $\Psi$  univariable monotonic inhomogeneuos, Kristensen \cite{Kristensen2011}.

   %  \item $n=1$ $m=1$, for the doubly metric setting and  non-monotonic approximating function $\psi$ proven by Ma and Su \cite{MaSu2008}

    % \item $nm\geq1$, for the doubly metric setting but for monotonic approximating function $\psi$ by Kristensen \cite{Kristensen2011}
    
\end{itemize}
To sum up, nothing is known prior to our result in the weighted settings. However, it is worth remarking that an asymptotic formula for the number of solutions to the Diophantine inequalities $ | \langle \q A \rangle |< \psi(\|\q\|)$ was proven in \cite{DKL}.

For the Hausdorff dimension result we consider a slightly different setup. For any vector $\boldsymbol{\tau}=(\tau_1,\ldots, \tau_n )\in \R^n$, satisfying
\begin{equation}\label{formalcondlambda}
\min_{1\leq j \leq n} \tau_j >1 \text{ \,\,\, and \,\,\, }\sum_j \tau_j \geq  n+m 
\end{equation}
define the set
\[
W^{\LL}_{n,m}(\tau)
:=
\left\{ A\in I^{m \times n}_\LL:  | \langle \q  A_j \rangle | < \|\q\|_{\infty}^{-\tau_j}\|\q\|_{\infty} \quad (1\leq j \leq n) \quad \text{ for i.m. } \q\in \F[X]^m\right\}
\]
and quantities
\[
s_j(\tau)
:= 
n(m-1) 
+ 
\, \frac{m + n - \sum\limits_{r\, :\,  \tau_{r}<\tau_{j}} (\tau_r-\tau_j)}{\tau_j}
\;\quad  (1\leq j \leq n)\, .
\]
Then we have
\begin{theorem}\label{formalt2}
For any vector $\boldsymbol{\tau}\in \R^n$, satisfying \eqref{formalcondlambda}, we have 
$$\dimh W^{\LL}_{n,m}( \tau) = \min\{s_1(\tau),\ldots, s_n(\tau)\}.
$$
\end{theorem}
Previously known results regarding the Hausdorff dimension of this set include:
\begin{itemize}
    \item $n=m=1$, Kristensen \cite{Kristensen2003}.
   % \item $n=m=1$, doubly metric non-monotonic approximating functions by Ma and Su \cite{MaSu2008}. 
    \item $nm\geq1$, $\Psi$ non-monotonic, Kristensen \cite{Kristensen2004}.
  %  \item $n\geq1$, $m\geq2$, for the inhomogeneous setting and monotonic approximating function by Kristensen \cite{Kristensen2011}.
  %  \item $nm\geq1$, doubly metric setting but for the monotonic approximating function by Kristensen \cite{Kristensen2011}.
\end{itemize}
Our weighted result is completely new.

\subsection{Proof of Theorem \ref{formalt1}}
First note that the convergence case is a simple consequence of the Borel-Cantelli lemma.
The proof for the divergence case is similar to the complex setup, so we will only highlight the differences. As in the other applications, we will make use of the following form of Minkowski's theorem on linear forms. 
\begin{lemma}[\cite{S67}] \label{formalmink}
     Suppose that for some $H\in\N$ one has
    \begin{equation*}
        \prod_{i=1}^n \phi_i(H) \prod_{j=1}^m \Phi_j(H)  \geq t^{n+m} .
    \end{equation*}
Then for any $A\in  I^{m \times n}_\LL$, the system
    \begin{align*}
     &  \| \q A_i -p_{i} \|< \phi_i(H) \hspace{0.9cm}   (1\leq i\leq n), \\[2ex]
     &   \|q_j \|\leq \Phi_j(H) \hspace{1.9cm}   (1 \leq j \leq m)
    \end{align*}
     has a non-trivial solution $(\q,\p)\in \F[X]^{m+n}$.
\end{lemma}

We work with the following objects:
\begin{enumerate}[i.]
\item $J:=\{ \alpha=(\q,\p) =(q_{1},\dots , q_{m},p_1,\ldots,p_n) \in \F[X]^{m+n}\,:\, \| \p\|_\infty \leq \| \q \|_{\infty} \}$,
\item     $\beta:J \to \R_{+}, \, \alpha \mapsto \beta_{\alpha}=\max\left\{ \Phi_1^{-1}\left( \|q_1\|\right), \ldots, \Phi_m^{-1}\left( \|q_m\| \right)\right\}$,
\item $u_{k}=M^{k}$ \text{ for some $M\in \N_{\geq 2}$ to be chosen later},
\item  $J_{k}=\{\alpha \in J : l_k \leq \beta_{\alpha} \leq u_k \}$ for some suitable sequence $l_{k}<u_{k}$ for all $k\in\N$.,
\item \text{Resonant sets as  $R_{\q}=\prod_{i=1}^{n}R_{\q,i}$,
 where $R_{\q,i}=\left\{A_{i} \in I^{m \times 1}_\LL:\q A_{i}=p_{i} \text{ for some } p_i\in \F[X]\right\}$},
 \item \text{In the current setting $\kappa_i=\frac{m-1}{m}$ and $\delta_{i}=m$ for each $i=1,\dots, n$}.
 \end{enumerate}
Finally, we define functions
    $$
\rho_j(u)= \frac{\phi_j(u)}{\|\Phi(u)\|_{\infty}} \left(t^{-(m+n)} \prod_{s=1}^n \phi_s(u)\prod_{k=1}^m \Phi_k(u)\right)^{-1/n}
$$
and 
$$
\psi_j(u) = \frac{\phi_j(u )}{ \|\Phi(u )\|_{\infty} }.
$$

Also note that with these choices of functions we have
$$
\prod_{j=1}^n \frac{\psi_j(u)}{\rho_j(u)}
=
t^{m+n}  \prod_{j=1}^n \phi_j(u) \prod_{k=1}^m \Phi_k(u)
\quad (u\in\N).
$$

As in the other applications, using Minkowski's theorem, we can make an additional assumption on $\phi$ and $\Phi$. If there is a strictly increasing sequence of natural numbers $(n_j)_{j\geq 1}$ such that
$$
\prod_{s=1}^n \phi_s(n_j)\prod_{k=1}^m \Phi_k(n_j) \geq t^{m+n} \quad (j\in\N),
$$
then $\mu^\LL_{m \times n}(W^{\LL}_{n,m}(\phi,\Phi))=1$ by Lemma~\ref{formalmink}. Therefore, we will assume that
\begin{equation}\label{assumption}
\prod_{j=1}^n \phi_j(u)\prod_{k=1}^m \Phi_k(u) < t^{m+n}
\;\text{ for large } u\in\N.
\end{equation}

The main ingredient of the proof is the formal power series analogue of a weighted ubiquitous system.
\begin{lemma}\label{formalubiq}
The system $(\left\{ R_{\alpha} \right\}_{\alpha\in J}, \beta)$ is a weighted ubiquitous system with respect to the function $\rho$ given above.
\end{lemma}

For each $j\in \{1,\ldots, n\}$, let $f_j:\N\to\R_{+}$ be the function given by
\[
f_j(u)
:=
\phi_j(u) \left( t^{-(m+n)} \prod_{s=1}^n \phi_s(u)  \prod_{k=1}^m \Phi_k(u) \right)^{-1/n}
\quad
(u\in\N).
\]
Observe that $\prod_j f_j(u) \prod_{k} \Phi_k(u) = t^{m+n}$ for all $u\in\N$.
\begin{proposition}
For each $u\in \N$ and each $A\in I^{m \times n}_\LL$ there is some $\alpha=(\q,\p)\in J$ such that 
\[
A\in \prod_{j=1}^n \Delta\left( \mathfrak{R}_{\alpha,j},  \, \frac{f_j(u)}{\|\q\|_\infty} \right)
\;\text{ and }\;
\beta_{\alpha} \leq u.
\]

\end{proposition}

\begin{proof}
By Lemma \ref{formalmink}, for every $H\in\N$ and $A\in I^{m \times n}_\LL$ there exists $(\q,\p)\in \F[X]^{m+n} \backslash \{0\}$, such that
\begin{align*}
     &  \| \q A_i -p_{i} \|< f_i(H) \hspace{1cm}  (1\leq i\leq n), \\[2ex]
     &   \|q_j \|\leq \Phi_j(H) \hspace{1.9cm}  \, (1 \leq j \leq m).
\end{align*}

Then in each coordinate $i$ we have
$$
\| \q \|_\infty \dist_\infty(A_i, R_{\q,i}) \leq \inf_{A_i'\in R_{\q,i}} \| \q A_i-\q A_i'\| \leq \|\q A_i-p_i \| < f_i(H).
$$
Dividing everything by $\| \q \|_\infty$ and noting that $\| q_k \| \leq \Phi_k(u)$ for all $k$ finishes the proof.
\end{proof}

For each $j\in\N$, define
\[
\widetilde{J}_j
:=
\left\{ (\q,\p)\in\F[X]^{m + n}: \frac{\Phi_k(u_j)}{M}\leq \|q_k\| \leq \Phi_k(u_j) \quad (1\leq k \leq m) \right\}.
\]
If $\alpha=(\q,\p)\in \widetilde{J}_j$, then $\beta_{\alpha}\leq u_j$. Since $\beta_{\alpha}\to \infty$ as $\alpha\to\infty$, we may pick a sequence $l_j$ ensuring $\widetilde{J}_j\subseteq J_j$. For each $j\in\N$ and $k\in \{1,\ldots, m\}$, define the set 

\[
J_{j,k}
:=
\left\{ \alpha\in J: \left \| q_k\right\| \leq \frac{\Phi_k(u_j)}{M} \quad \text{and}\quad  \|q_s\|\leq \Phi_s(u_j)  \quad (s\in \{1,\ldots m\}\setminus \{k\}) \right\}.
\]

Let $B=\prod_{k=1}^n B(X_k;r)$ be an arbitrary ball. Then, for all $s\in\N$, 
\begin{align*}
B 
&= 
B\cap \bigcup_{\alpha:\beta_{\alpha}\leq u_s} \prod_{k=1}^n\Delta\left( R_{\alpha,k}, \, \frac{f_k(u_s)}{\|\q\|_\infty} \right) \nonumber\\
&=
\left(B\cap \bigcup_{\alpha \in \widetilde{J}_s} \prod_{k=1}^n \Delta\left( R_{\alpha,k}, \, \frac{f_k(u_s)}{\|\q\|_\infty} \right)\right)
\cup
\left(B\cap \bigcup_{h=1}^m \bigcup_{\alpha \in J_{s,h}} \prod_{k=1}^n \Delta\left( R_{\alpha,k}, \, \frac{f_k(u_s)}{\|\q\|_\infty} \right)\right).
\end{align*}

For any fixed $\q=(q_1,\ldots,q_m)$, the number of $p_i$ such that the intersection is non-empty is not greater than $4 \| \q\|_\infty r$.

Finally, we can bound the Haar measure of the second term in the union from above. This is an analogue of Proposition \ref{complexprop4} from the complex setup.
\begin{proposition} \label{formal power series prop 7} If $M\geq 2^{2n+1}t^{2m+n}m$, then for every large $s\in\N$ we have
\begin{align*}
\mu^\LL_{m \times n}\left(B\cap \bigcup_{h=1}^m \bigcup_{\alpha \in J_{s,h}} \prod_{j=1}^n \Delta\left( R_{\alpha,j}, \, \frac{f_j(u_s)}{\|\q\|_\infty} \right)\right)\leq \frac{1}{2}\mu^\LL_{m \times n}(B).
\end{align*}

\end{proposition}

\begin{proof}
For each $s\in\N$ and $k\in\{1,\ldots, m\}$, write 

\[
J_{s,k}'
:=
\left\{ \q \in \F[X]^m: \alpha=(\q,\p)\in J_{s,k} \right\},
\]

so we can bound the number of elements as 
\[
\# J_{s,k}'
\leq
\frac{t^m}{M}\prod_{j=1}^m \Phi_j(u_s).
\]
If $G_\LL$ is the intersection in the statement of Proposition~\ref{formal power series prop 7}, then 
\begin{align*}
\mu^\LL_{m \times n}(G_\LL) &
\leq 
\sum\limits_{k=1}^m \sum\limits_{\q\in J_{s,k}'} 
(4r\|\q\|_\infty)^{n} \prod_{j=1}^n \,\frac{f_j(u_s)r^{m-1}}{\|\q\|_{\infty}} \\
& \leq 4^n \sum\limits_{k=1}^m \sum\limits_{\q\in J_{s,k}'} r^{nm} \prod_{j=1}^n \,f_j(u_s)\\
& \leq \frac{4^n r^{nm} t^{m+n}}{\Phi_1(u_s)\ldots\Phi_m(u_s)}  \sum\limits_{k=1}^m \# J_{s,k}' \\
& \leq \frac{4^n t^{2m+n}m}{M}r^{nm}\\& \leq \frac{1}{2}\mu^\LL_{m \times n}(B).
\end{align*}
\end{proof}

The definition of $\widetilde{J}_s$ tells us that for each $\alpha=(\q, \p) \in \widetilde{J}_s$ one has
\[
\|\q\|_{\infty}
\geq 
\frac{1}{M} \left\| \Phi_1(u_s), \ldots, \Phi_m(u_s) \right\|_{\infty},
\]
so
\begin{align*}
\mu^\LL_{m \times n}\left( B\cap  \bigcup_{\alpha \in J_{s}} \prod_{k=1}^m \Delta\left( R_{\alpha,k},  M\, \frac{f_k(u_s)}{\|\Phi(u_s)\|_{\infty}} \right)\right)
&\geq 
\mu^\LL_{m \times n}\left( B\cap  \bigcup_{\alpha \in \widetilde{J}_{n}} \prod_{k=1}^m \Delta\left( R_{\alpha,k},  M\, \frac{f_k(u_s)}{\|\Phi(u_s)\|_{\infty}} \right)\right) \\
&\geq 
\mu^\LL_{m \times n}\left( B\cap  \bigcup_{\alpha \in \widetilde{J}_{s}} \prod_{k=1}^m \Delta\left( R_{\alpha,k},  \, \frac{f_k(u_s)}{\|\q \|_{\infty}} \right)\right) \\
&\geq 
\frac{1}{2}\mu^\LL_{m \times n}(B).
\end{align*}
So the ubiquity property is proven.\\

The rest of the proof of Theorem \ref{formalt1} is very similar to the complex setup, but with the difference that we are considering
$$
 \prod_{j=1}^n \varphi_j(q) \prod_{k=1}^m \Phi_k(q) 
$$
instead of
$$
\left( \prod_{j=1}^n \varphi_j(q) \prod_{k=1}^m \Phi_k(q) \right)^2
$$
and therefore we skip it.

\subsection{Proof of Theorem \ref{formalt2}}\

\subsubsection{Upper bound.} We make a natural cover of the limsup set and then use the Hausdorff-Cantelli lemma.
\[
W^{\LL}_{n,m}(\tau)
\subseteq
\bigcap_{Q\in \N} 
\bigcup_{\substack{\q\in\F[X]^m \\ \|\q\|_\infty=t^Q}} 
\bigcup_{\substack{ \p : \\ (\q,\p)\in J}} 
\prod_{j=1}^n \triangle\left( R_{\alpha,j},  \|\q\|_{\infty}^{-\tau_j}\right).
\]
Also, observe that for all $j,l\in \{1,\ldots, n\}$ the number of balls of radius $\|\q\|_\infty^{-\tau_j}$ to cover $\triangle(R_{\alpha,l},\|\q\|_{\infty}^{- \tau_l})$ is asymptotically equivalent to
\[
\max\left\{ 1, \frac{\|\q\|_{\infty}^{-\tau_l} }{\|\q\|_{\infty}^{-\tau_j}}\right\} \|\q\|_{\infty}^{\tau_j(m-1)}.
\]
Hence, for all $s>0$, 
\begin{align*}
\mathcal{H}^s\left(W(\tau)\right)
&\ll
\liminf_{N\to\infty} \sum\limits_{\substack{\q\in\F[X]^m \\ \|\q\|_\infty \geq t^N}} \|\q\|_\infty^{m-1}\|\q\|_\infty^{n} \|\q\|_\infty^{-s\tau_j} \|\q\|_\infty^{\tau_j n(m-1)} \|\q\|_\infty^{ \sum\limits_{l: \, \tau_j>\tau_l} (\tau_j - \tau_l)} \\
&=
\liminf_{N\to\infty} \sum\limits_{\substack{\q\in\F[X]^m \\ \|\q\|_\infty \geq t^N}} \|\q\|_\infty^{m+n-1 + \tau_j n(m-1) +  \sum\limits_{l:\, \tau_j>\tau_l} (\tau_j - \tau_l)}  \|\q\|_\infty^{-s\tau_j}\\
&= m(t-1)t^{m-1} \liminf_{N\to\infty} \sum\limits_{r=N}^\infty \left(t^{m+n-1 + \tau_j n(m-1) +  \sum\limits_{l:\, \tau_j>\tau_l} (\tau_j - \tau_l) -s\tau_j} \right)^r.
\end{align*}

The series above converges if and only if the exponent of $e$ is strictly less than $0$, or equivalently, $s> s_j(\tau)$. The result follows by taking the infimum over $j$.\par

\subsubsection{Lower bound.} The proof of the lower bound is similar to the complex case with obvious modifications. Namely, given any $\boldsymbol{l}=(l_1,\ldots, l_n)\in \R^n$ such that
\begin{equation}%\label{Eq:CondOnA}
\min_{1\leq j\leq n} l_j \geq 1
\;\text{ and }\;
\sum\limits_{j=1}^n l_j = m+n,
\end{equation}
we define $\tilde{\rho} =(\tilde{\rho}_1 ,\ldots,\tilde{\rho}_n):\N\to\R^n$ by
\[
\tilde{\rho}_j(u) = \,\frac{1}{u^{l_j-1}}
\quad
(1\leq j\leq n,\; u\in\N).
\]
and then prove the ubiquity statement for this set of functions as in Lemma \ref{formalubiq}. For the last part of the proof, we reorder the coefficients $\tau$ to have $\tau_1\geq \ldots \geq \tau_2 \geq \ldots \geq \tau_n >1$ and then we consider two cases: $\tau_n \geq \frac{n+m}{m}$ and $\frac{n+m}{m}>\tau_n$ and then apply Theorem \ref{wang wu KG} to get the statement of the theorem. We refer reader to the Section \ref{hausdorff dimension}, as the choice of parameters and calculations are the same as in the $p$-adic case.

\noindent{\bf Acknowledgments:} We thank Johannes Schleischitz and Dong Han Kim for useful discussions. 

%\bibliographystyle{abbrv}

%\bibliography{biblioo}

\end{document}